\documentclass[a4paper,reqno,12pt]{amsart} 
\usepackage{amsmath}
\usepackage{graphicx}
\usepackage{mathrsfs}
\usepackage{amssymb,amsmath, amsfonts, amsthm}
\usepackage{color}
\usepackage[colorlinks]{hyperref}
\hypersetup{citecolor=red, linkcolor=blue}
\usepackage[margin=2.73cm]{geometry}
\usepackage{nccmath}

\newtheorem{theorem}{Theorem}[section]
\newtheorem{proposition}[theorem]{Proposition}

\newtheorem{lemma}[theorem]{Lemma}

\theoremstyle{definition}
\newtheorem{remark}[theorem]{Remark}

\newcommand{\ep}{\epsilon}
\newcommand{\ga}{\gamma}

\newcommand{\R}{\mathbb{R}}
\newcommand{\Sn}{\mathbb{S}}
\newcommand{\N}{\mathbb{N}}

\newcommand{\mca}{\mathcal{A}}

\newcommand{\mcp}{\mathcal{P}}
\newcommand{\mcq}{\mathcal{Q}}

\newcommand{\tsg}{\normalfont \textsf{g}}

\newcommand{\tix}{\tilde{x}}
\newcommand{\tith}{\tilde{\theta}}
\newcommand{\wtb}{\widetilde{B}}

\renewcommand{\(}{\left(}
\renewcommand{\)}{\right)}

\numberwithin{equation}{section}
\allowdisplaybreaks

\makeatletter
\@namedef{subjclassname@2020}{%
\textup{2020} Mathematics Subject Classification}
\makeatother

\begin{document}
\title[Sign-changing solutions for critical Hamiltonian systems]%
{Sign-changing solutions \\ for critical Hamiltonian systems in $\R^N$}

\author{Yuxia Guo, Seunghyeok Kim, Angela Pistoia, Shusen Yan}

\address[Yuxia Guo]{Department of Mathematical Science, Tsinghua University, Beijing 100084, P. R. China}
\email{yguo@tsinghua.edu.cn}

\address[Seunghyeok Kim]{Department of Mathematics and Research Institute for Natural Sciences, College of Natural Sciences, Hanyang University, 222 Wangsimni-ro Seongdong-gu, Seoul 04763, Republic of Korea}
\email{shkim0401@hanyang.ac.kr shkim0401@gmail.com}

\address[Angela Pistoia]{Dipartimento di Scienze di Base e Applicate per l'Ingegneria, Sapienza Universit\`a di Roma, Via Antonio Scarpa 10, 00161 Roma, Italy}
\email{angela.pistoia@uniroma1.it}

\address[Shusen Yan]{School of Mathematics and Statistics, Central China Normal University, Wuhan 430079, P. R. China}
\email{syan@ccnu.edu.cn}



\date{\today}

\begin{abstract}
We build infinitely many geometrically distinct non-radial sign-changing solutions for  the Hamiltonian-type elliptic systems
$$ -\Delta u =|v|^{p-1}v\  \hbox{in}\ \R^N,\
-\Delta v =|u|^{q-1}u\  \hbox{in}\ \R^N,$$
where the  exponents  $(p,q)$  satisfy $p,q>1$ and belong to the critical hyperbola
$$\frac1{p+1}+\frac1{q+1} =\frac {N-2}N.$$
To establish this result, we introduce several new ideas and strategies
that are both robust and potentially applicable to other critical problems lacking the Kelvin invariance.
\end{abstract}

\maketitle

\section{Introduction}
In this paper, we investigate the existence of infinitely many geometrically distinct non-radial sign-changing solutions for the following critical Hamiltonian system:
\begin{equation}\label{2}
\begin{cases}
-\Delta u =|v|^{p-1}v\;\;\; \text{in } \R^N,\\
-\Delta v =|u|^{q-1}u\;\;\; \text{in } \R^N,\\
(u,v)\in \dot{W}^{2,\frac{p+1}{p}}(\R^N) \times \dot{W}^{2,\frac{q+1}{q}}(\R^N).
\end{cases}
\end{equation}
Here, $N \ge 3$ and the pair $(p,q)$ satisfies the criticality condition
\begin{equation}\label{cri}
\frac1{p+1}+\frac1{q+1} =\frac {N-2}N.
\end{equation}
Also, we can always assume $\frac{2}{N-2} < p \le \frac{N+2}{N-2} \le q$ without loss of generality, because otherwise we can interchange the roles of $p$ and $q$.
Given $m \in \N$ and $r \ge 1$, $\dot{W}^{m,r}(\R^N)$ denotes the homogeneous Sobolev space, and two pairs of functions $(u_1,v_1)$ and $(u_2,v_2)$ are said to be geometrically distinct if $(u_1,v_1)$ cannot be obtained from $(u_2,v_2)$ through a combination of translations and dilations.

Critical system \eqref{2}--\eqref{cri} is realized as the Euler-Lagrange equation of the Sobolev embedding $\dot{W}^{2,\frac{p+1}{p}}(\R^N) \hookrightarrow L^{q+1}(\R^N)$.
By employing the concentration-compactness principle, Lions \cite{Li} found a positive ground state $(U,V)$ to system \eqref{2}--\eqref{cri} for all $N \ge 3$ and $p \in (\frac{2}{N-2},\frac{N+2}{N-2}]$.
Alvino, Lions \& Trombetti \cite{ALT} (see also \cite[Corollary I.2]{Li}) showed that a positive ground state is radially symmetric and decreasing in the radial variable, after a suitable translation.
It is worth mentioning that Chen, Li \& Ou \cite{CLO} (see also \cite[Section 4]{CL}) established the same symmetry result for any positive solutions to \eqref{2}--\eqref{cri}.
Then, Wang \cite{Wa} and Hulshof \& van der Vorst \cite{HV2} proved that the ground state is positive and unique up to translation and scaling.
In \cite{FKP}, Frank, Kim \& Pistoia deduced the non-degeneracy of each ground state. We refer to Section \ref{sec:pre} for a more precise description.

The existence of infinitely many geometrically distinct sign-changing solutions to critical system \eqref{2}--\eqref{cri} has remained a widely open problem for decades.
To the best of our knowledge, there is only one known result regarding the existence and multiplicity of sign-changing solutions in the literature, due to Clapp \& Salda\~na \cite{CS}.
They proved that \eqref{2}--\eqref{cri} admits at least $\lfloor \frac{N}{4} \rfloor$ geometrically distinct non-radial sign-changing solutions for all $N \ge 4$, by a careful analysis of the energy minimizing sequence on a certain symmetric class.

\medskip
Let us recall what is known in the case of the single equation
\begin{equation}\label{criti}
-\Delta u =|u|^{\frac{4}{N-2}}u\;\;\; \text{in } \R^N,
\end{equation}
which can be obtained from \eqref{2} by putting $p=q=\frac{N+2}{N-2}$ and $u=v$. It is well-known that all the positive solutions of
\eqref{criti}, often called bubbles, are given by
\begin{equation}\label{bub-cri}
U_{\xi,\delta}(y) := \delta^{-{\frac{N-2}{2}}} U\(\frac{y-\xi}{\delta}\),\quad (\delta,\xi) \in (0,+\infty) \times \R^N,
\end{equation}
where
\begin{equation}\label{bub-cri2}
U(y) := \frac{\alpha_N}{(1+|y|^2)^{\frac{N-2}{2}}},\quad \alpha_N := (N(N-2))^{\frac{N-2}{4}}.
\end{equation}
The first existence result concerning sign-changing solutions of \eqref{criti} is traced back to Ding \cite{ding}, who examined the problem in a variational viewpoint.
The invariance of \eqref{criti} under translations and dilations shows the lack of compactness of the energy functional naturally associated to \eqref{criti}.
However, Ding used the conformal invariance of the problem and looked for solutions that are invariant under the action of groups of conformal transformations of $\R^N$ whose orbits have positive dimension.
In such a way, he could restore the compactness and establish the existence of infinitely many sign-changing solutions for the single equation \eqref{criti}.
Successively,    sign-changing   solutions to \eqref{criti} have been built by del Pino, Musso, Pacard \& Pistoia in \cite{DMPP,DMPP2} via the gluing method.
The solutions obtained are invariant under the Kelvin transform   and can be visualized as a superposition of the radial bubble $U=U_{0,1}$ in \eqref{bub-cri2}
with a large number of negative bubbles of the form \eqref{bub-cri} with small scaling parameters $\delta$  whose peaks are arranged along  special submanifolds of $\R^N$.
Different type of solutions (still Kelvin invariant) have also been built by Medina \& Musso \cite{MM} and Medina, Musso \& Wei \cite{MMW}.

If $p = 1$, system \eqref{2}--\eqref{cri} reduces to the biharmonic equation
\begin{equation}\label{criti bi}
(-\Delta)^2 u =|u|^{\frac{8}{N-4}}u\;\;\; \text{in } \R^N,
\end{equation}
and similar construction to \cite{DMPP} was carried out by Guo, Li \& Wei in \cite{GLW}. In this paper, we treat the case $1 < p \le \frac{N+2}{N-2}$.

\medskip
A natural question  is if it is possible to construct sign-changing solutions to critical system \eqref{2}--\eqref{cri} in the spirit  of  the paper  \cite{DMPP}.
In this paper, we give an affirmative answer to this problem for $N \ge 6$, thereby finding infinitely many geometrically distinct non-radial sign-changing solutions to \eqref{2}--\eqref{cri}.

For a precise formulation of our result, we introduce the building blocks used in the construction:
Let $(U_{0,1},V_{0,1})$ be the unique radially symmetric positive ground state of critical system \eqref{2}--\eqref{cri} satisfying $\max_{y\in \R^N} U_{0,1}(y)=U_{0,1}(0)=1$.
In the rest of the paper, we will often regard it as a function in $r = |y|$. By the scaling and translation invariance of \eqref{2}--\eqref{cri}, if
\begin{equation}\label{6}
(U_{x,\mu}(y),V_{x,\mu}(y)) := \(\mu^{\frac{N}{q+1}} U_{0,1} (\mu(y-x)), \; \mu^{\frac{N}{p+1}} V_{0,1}(\mu(y-x))\) \end{equation}
for $y \in \R^N$, then $\{(U_{x,\mu},V_{x,\mu}): (x,\mu) \in \R^N \times (0,+\infty)\}$ is precisely the set of the ground states of \eqref{2}--\eqref{cri}.
Each pair $(U_{x,\mu},V_{x,\mu})$ will be called a bubble for the system.

\medskip
The following theorem is the main result of this paper.
\begin{theorem}\label{th1-25-2}
Suppose that \eqref{cri}  holds. Assume that $N \ge 6$ and $p\in (\frac N{N-2}, \frac{N+2}{N-2}]$.
Then there exists a positive integer $k_0$, such that for all $k\ge k_0$, system \eqref{2} possesses a solution $(u_k, v_k)$ of the form
\begin{equation}\label{eq:ukvkform}
(u_k, v_k) \approx  (U_{0,\mu_0},V_{0,\mu_0}) - \bigg(\sum_{j=1}^k U_{x_j,\mu},\sum_{j=1}^k V_{x_j,\mu}\bigg).
\end{equation}
Here, $\mu_0,\, r > 0$ are close to some constant independent of $k$,
\[x_j = \(r\cos\(\tfrac{2(j-1)\pi}{k}\), r\sin\(\tfrac{2(j-1)\pi}{k}\),{\bf{0}}\) \in \R^2 \times \R^{N-2}, \quad j=1,2,\ldots,k,\]
and $\mu>0$ depends on $k$ (more precisely, $\mu \simeq k^{\frac{(p+1)(N-2)}N}$) and $\mu\to +\infty$ as $k\to +\infty$.
\end{theorem}

As a matter of fact, a slight modification of our argument extends the above existence theorem to the case $N = 5$ and $p \in [\frac{35}{18},\frac73)$, as explained in Remarks \ref{Re1-25-3}--\ref{Re11-25-3}, \ref{Re4-25-3}, \ref{Re2-25-3} and \ref{Re3-25-3}.

As previously mentioned, when $p=q=\frac{N+2}{N-2}$ and $u=v$, system \eqref{2} reduces to the second-order equation \eqref{criti} with a critical nonlinearity, which is well-defined for $N\ge 3$.
In contrast, when $p=1$, system \eqref{2} becomes the fourth-order equation \eqref{criti bi} with a critical nonlinearity, which needs $N\ge 5$.
This indicates that, in lower dimensions, some restriction on $p$ is necessary; specifically that $p > \frac{2}{N-2}$ in light of \eqref{cri}.

Several aspects still require refinement. First, in the case $N=5$, our result does not cover the full range of
$p$, although we believe that such a restriction is purely technical.
For the lower dimensional cases $N=3$ and $N=4$, the existence result holds for $p$ close to $\frac{N+2}{N-2}$, but fails for $p$ close to 1.
It seems that more sophisticated estimates are required to find the exact range of $p$ such that the existence result holds.
However, to keep the paper concise, we do not pursue this generality here.
Finally, for $p\in (1,\frac{N}{N-2}]$, the estimates are quite different due to the different decay rate at infinity for $V_{x,\mu}(y)$ (see Lemma~\ref{lem:HV}), though the same procedure for the case $p>\frac N{N-2}$ could be carried out. We will briefly discuss this case in Section~\ref{sec:plow}.

\medskip
Our result is inspired by the work of del Pino, Musso, Pacard, \& Pistoia in \cite{DMPP} where   solutions to the single equation  \eqref{criti} were built.
To achieve the goal, they employed the Ljapunov-Schmidt reduction method using the number $k$ of bubbles as the parameter, the strategy introduced by Wei \& Yan in \cite{WY}.
However, applying the argument presented in \cite{DMPP} to our context introduces significant difficulties, making our construction considerably more challenging.
Now let us outline these difficulties and depict how to overcome them.

Problem \eqref{2} is invariant under scaling. In other words, if $(u, v)$ is a solution to \eqref{2}, then $ (t^{\frac N{q+1}} u(t y), t^{\frac N{p+1}} v(t y)    )$ is also a solution for all $t>0$.
This indicates that, without an additional condition, the parameter $\mu_0$ cannot be uniquely determined when constructing solutions of the form given in \eqref{eq:ukvkform}.

In the special case where $p=q$ and $u=v$ in \eqref{2}, the reduced single equation \eqref{criti} is invariant under the  Kelvin transformation.
By taking advantage of this invariance, del Pino,  Musso,  Pacard, \&  Pistoia built  solutions of  the form
\[u\approx U_{0,\mu_0} - \sum_{j=1}^k U_{x_j,\mu},\]
where the bubbles $U_{0,\mu_0}$ and $U_{x_j,\mu}$ are defined in \eqref{bub-cri}--\eqref{bub-cri2}.
They worked in a subspace $X$ of the Beppo-Levi space $D^{1,2}(\R^N)$ such that any element in $X$ is invariant under the Kelvin transformation in addition to suitable rotations and reflections.
It is easy to see that if $U_{x, \mu} \in X$, then it must hold that
\[|x|^2+\mu^2=1.\]
As a result, working in the space $X$ automatically yields
\[\mu_0=1,\quad |x_j|^2 +\mu^2=1,\]
and so $\mu_0$ is determined. More importantly, the linear operator
\[
L\omega := -\Delta \omega - \frac{N+2}{N-2} U_{0, 1}^{\frac 4{N-2}} \omega,\quad \omega\in X,
\]
is invertible in $X$. The invertibility of the linearized operator $L$ in $X$ makes the Ljapunov-Schmidt reduction argument workable and one can reduce the problem to just slightly adjusting the scaling parameter $\mu$ as a function of the number of peaks $k \gg 1$.
However, unlike what has been done for the single equation, we cannot construct Kelvin invariant solutions, since system \eqref{2}--\eqref{cri} is not invariant under the Kelvin transformation unless $p = 1$ or $\frac{N+2}{N-2}$!
This lack of the Kelvin symmetry forces us to work in a subspace $X'$ of the Sobolev space $\dot{W}^{2,\frac{p+1}{p}}(\R^N) \times \dot{W}^{2,\frac{q+1}{q}}(\R^N)$
that possesses only rotational and reflectional symmetries, thereby making the construction significantly challenging.

Our novelty in this paper is to leave $\mu_0$ as a free parameter for approximate solutions, instead of just taking
$ (U_{0,1},V_{0,1})$ as in the case of the single equation.
The introduction of a new free parameter is not an innocent matter since the linearized operator will not be invertible anymore in the space $X'$.
Therefore, the first challenge in employing the reduction argument is establishing the invertibility of the linear operator for the bubble $(U_{0,\mu_0},V_{0,\mu_0})$.
This is why we need to develop new ideas to study the linear theory.
We will examine the existence and qualitative behavior of the Green's function of a specific linear operator (the operator $L^*$ in \eqref{1-25-10n2}) in $X'$, which are new and can be also useful in different contexts.

In addition to introducing the free parameter $\mu_0$, another notable feature is that our approximate solutions are not merely given by the right-hand side of \eqref{eq:ukvkform}.
In fact, such a naive choice seems to work only for $p$  close to $\frac{N+2}{N-2}$ (for example, for $p>\frac{N+1}{N-2}$ as commented in \cite[Remark 1.3]{GLP}).
While we will choose the $v$-part of approximate solutions by $V_{0,\mu_0} - \sum_{j=1}^k V_{x_j,\mu}$,
 the $u$-part will be defined by a differential equation \eqref{Udef}.
 This modification of the approximation for the $u$-part makes
 the estimates much more sophisticated. See Section~\ref{sec:exp1}.
The latter choice is reminiscent of the nonlinear projection introduced by Kim \& Pistoia in their study of slightly subcritical Lane-Emden system in bounded domains \cite{KP}, and the construction of periodic positive solutions in \cite{GWY}.
It is interesting to note that they employed this idea exclusively for the case $p \in (1,\frac{N}{N-2})$.
In contrast, we find that it is effective for all cases $p \in (1,\frac{N+2}{N-2}]$ under our setting.

The introduction of the free parameter $\mu_0$ leads to a further issue in the study of the reduced problem.
If we consider a solution of the form \eqref{eq:ukvkform}, we can take any $\mu_0>0$ by appropriately adjusting $\mu$ and $r=|x_j|$.
This observation implies that the critical points of the main-order term in the expansion of the reduced energy functional cannot be isolated, and the perturbations of this functional may fail to preserve the existence of critical points.
To overcome this difficulty, we need to check the correction term $(\omega_1,\omega_2)$ obtained from the reduction procedure (see Proposition~\ref{proposition4-3}) still keeps the invariance under the scaling.

As one can see in the proof of Theorem~\ref{th1-25-2}, our additional ideas for developing functional analytic tools to investigate the linearized operator, constructing approximate solutions,
and selecting optimal norms that effectively capture the behavior of error terms are robust.
We believe that our approach will enable the discovery of positive solutions to the critical Hamiltonian-type system for \textbf{all} $p > 1$, or more broadly, possibly sign-changing solutions to other critical problems lacking the Kelvin invariance.

\medskip
It is natural to ask if it is possible to build solutions to system \eqref{2}--\eqref{cri} which concentrate along different configuration as it has been done for the single equation \eqref{criti}.
For example, we expect that one can perform the construction of Medina, Musso \& Wei in \cite{MMW} for the system by adapting our approach here.
We remind that they built Kelvin invariant solutions which are the superposition  of the bubble  $U_{0,1}$ in \eqref{bub-cri2}
with  a large number of negative scaled copies of  bubbles $U_{\xi_i,\delta}$ and  $U_{\eta_j,\epsilon}$ defined in \eqref{bub-cri} whose peaks $\xi_i$ and $\eta_j$ are vertices of   regular polygons lying in two orthogonal  planes of $\R^N$.

There have been a few works on the existence of positive solutions of the form $(u_k, v_k) \approx (\sum_{j=1}^k U_{x_j,\mu},\sum_{j=1}^k V_{x_j,\mu})$
for the critical Hamiltonian system with potentials, namely, system \eqref{2}--\eqref{cri} where the nonlinearities $(|v|^{p-1}v,|u|^{q-1}u)$ are replaced with $(K_1(x) |v|^{p-1}v, K_2(x) |u|^{q-1}u)$.
Refer to e.g. \cite{GLP, CH}, which addressed the case $p > \frac{N+1}{N-2}$.
For such $p$, the approximate solution can simply be taken as
$(\sum_{j=1}^k U_{x_j,\mu},\sum_{j=1}^k V_{x_j,\mu})$, and more coarse norms can be employed.
Moreover, the existence of solutions for this problem largely depend on a suitable behavior of the potentials,
whereas our result arises solely from a delicate balance between a positive bubble with free parameter $\mu_0$ and negative bubbles.

Let us make a final remark on
the Lane-Emden conjecture, which states that if $p,q > 0$ and $(p,q)$ is subcritical, i.e., $\frac1{p+1}+\frac1{q+1} > \frac {N-2}N$, then the Lane-Emden system
\begin{equation}\label{22}
\begin{cases}
-\Delta u =|v|^{p-1}v\;\;\; \text{in } \R^N,\\
-\Delta v =|u|^{q-1}u\;\;\; \text{in } \R^N
\end{cases}
\end{equation}
has no positive classical solutions. The lack of invariance of system \eqref{22} under the Kelvin transform is a key reason why the full Lane-Emden conjecture remains a long-standing open problem. Refer to e.g. \cite{So}.
Our observation of the existence of infinitely many geometrically distinct crown-shaped solutions to the critical Lane-Emden system \eqref{2}--\eqref{cri}, despite the absence of Kelvin invariance, might shed light on advancing the study of the Lane-Emden conjecture.

\medskip \noindent \textbf{Organization of the paper.}
In Section \ref{sec:pre}, we provide properties of bubbles for system \eqref{2}--\eqref{cri} needed in the paper.

From Section \ref{sec:exp1} to \ref{sec:exist}, we will prove Theorem~\ref{th1-25-2}.

In Section \ref{sec:exp1}, we build approximate solutions and derive the expansion of their energy.

In Section \ref{sec:lin}, we develop the linear theory, where we utilize the Green's function for the linear operator $L^*$ in \eqref{1-25-10n2}, whose properties are extensively studied in Appendix \ref{sec:Green}.
The carefully chosen weighted norms (see \eqref{eq:*-norm}--\eqref{eq:**-norm}) allows us to handle all $p \in (1,\frac{N+2}{N-2}]$.

In Section \ref{sec:red}, we conduct the error estimate and the Ljapunov-Schmidt reduction procedure to find the correction term $(\omega_1,\omega_2)$. We also make an important observation that $(\omega_1,\omega_2)$ is invariant under the scaling.

In Section \ref{sec:exist}, we complete the proof of Theorem~\ref{th1-25-2} by analyzing the reduced energy functional $K$ defined in \eqref{eq:redene}.
It is characterized as a perturbation of a function (see \eqref{n30-31-12}) with non-isolated critical points.
By exploiting the scaling property \eqref{Kinv} of $K$, we demonstrate that $K$ continues to have non-isolated critical points.

In Section \ref{sec:plow}, we explain necessary modifications for the case $p \in (1,\frac N{N-2})$.

In Appendix \ref{sec:tech}, we collect several technical estimates that are used throughout the paper.

\medskip \noindent \textbf{Notation.} In the following we agree that $C,\, c > 0$ are universal constants independent of $k \in \N$ that may vary from line to line and even in the same line.
Also, we say that $a \simeq b$ if $C^{-1}b \le a \le Cb$ holds for some universal constant $C \ge 1$ independent of associated parameters of $a,\, b$.

\section{preliminaries}\label{sec:pre}
We present properties of the bubbles $(U_{x,\mu},V_{x,\mu})$ in \eqref{6}, assuming that $N \ge 3$ and $(p,q)$ satisfies $p \in (\frac{2}{N-2},\frac{N+2}{N-2}]$ and the criticality condition \eqref{cri}.

\begin{lemma}[Hulshof and Van der Vorst \cite{HV2}]\label{lem:HV}
The pair $(U_{0,1},V_{0,1})$ is unique, radially symmetric, and decreasing in the radial variable $r$.
Also, there exist numbers $a_{N,p},\, b_{N,p} > 0$ depending only on $N$ and $p$ such that
\begin{equation}\label{eq:HV}
\begin{cases}
\lim\limits_{r \to +\infty} r^{N-2} U_{0,1}(r) = a_{N,p} &\text{if } p \in (\frac{N}{N-2}, \frac{N+2}{N-2}],\\
\lim\limits_{r \to +\infty} \dfrac{r^{N-2}}{\log r} U_{0,1}(r) = a_{N,p} &\text{if } p = \frac{N}{N-2},\\
\lim\limits_{r \to +\infty} r^{p(N-2)-2} U_{0,1}(r) = a_{N,p} &\text{if } p \in (\frac{2}{N-2}, \frac{N}{N-2})
\end{cases}
\quad \text{and} \quad
\lim_{r \to +\infty} r^{N-2} V_{0,1}(r) = b_{N,p}.
\end{equation}
\end{lemma}
\noindent The uniqueness assertion in the previous lemma was independently proved by Wang \cite{Wa}. Also, the decay estimate as $r \to +\infty$ can be improved as follows:
\begin{lemma}[Kim and Moon \cite{KM}]
There exists a constant $C > 0$ depending only on $N$ and $p$ such that
\begin{equation}\label{V10est}
\left|V_{0,1}(r) - \frac{b_{N,p}}{r^{N-2}}\right| \le \frac{C}{r^N}.
\end{equation}
Besides,
\begin{equation}\label{U10est}
\begin{cases}
\displaystyle \left|U_{0,1}(r) - \frac{a_{N,p}}{r^{N-2}}\right| \le \frac{C}{r^{N-2+\kappa_0}} &\text{if } p \in (\frac{N}{N-2}, \frac{N+2}{N-2}], \\
\displaystyle \left|U_{0,1}(r) - \frac{a_{N,p} \log r}{r^{N-2}}\right| \le \frac{C}{r^{N-2}} &\text{if } p = \frac{N}{N-2}, \\
\displaystyle \left|U_{0,1}(r) - \frac{a_{N,p}}{r^{p(N-2)-2}}\right| \le \frac{C}{r^{p(N-2)-2+\kappa_1}} &\text{if } p \in (\frac{2}{N-2}, \frac{N}{N-2}),
\end{cases}
\end{equation}
where $\kappa_0 := p(N-2)-N > 0$ and $\kappa_1$ is any number in $(0, \min\{N-p(N-2),2(p+1)\})$.
\end{lemma}
\noindent Very recently, the above decay estimate was further improved in \cite{GHPY}.

\medskip
Furthermore, we will rely on the non-degeneracy of the bubbles.
\begin{lemma}[Frank, Kim and Pistoia \cite{FKP}]\label{lemma:FKP}
For $x = (x_1,\ldots,x_N) \in \R^N$ and $\mu > 0$, let
\[\(Y_{x,\mu}^0,Z_{x,\mu}^0\) := \(\frac{\partial U_{x,\mu}}{\partial \mu}, \frac{\partial V_{x,\mu}}{\partial \mu}\)
\quad \text{and} \quad
\(Y_{x,\mu}^h,Z_{x,\mu}^h\) := \(\frac{\partial U_{x,\mu}}{\partial x_h}, \frac{\partial V_{x,\mu}}{\partial x_h}\)\]
for $h = 1,\ldots,N$. Then the space of solutions of the linearized system
\[\begin{cases}
-\Delta Y = pV_{x,\mu}^{p-1} Z \;\;\; \textup{in}\; \R^N,\\
-\Delta Z = qU_{x,\mu}^{q-1} Y \;\;\; \textup{in}\; \R^N,\\
(Y,Z) \in \dot{W}^{2,\frac{p+1}{p}}(\R^N) \times \dot{W}^{2,\frac{q+1}{q}}(\R^N)
\end{cases}\]
is spanned by
\[\left\{\(Y_{x,\mu}^0,Z_{x,\mu}^0\), \(Y_{x,\mu}^1,Z_{x,\mu}^1\), \ldots, \(Y_{x,\mu}^N,Z_{x,\mu}^N\)\right\}.\]
\end{lemma}

\section{The energy expansion for $p \in (\frac N{N-2},\frac{N+2}{N-2}]$}\label{sec:exp1}
In Sections \ref{sec:exp1}--\ref{sec:exist}, we assume that $N \ge 6$, $p \in (\frac N{N-2},\frac{N+2}{N-2}]$, and \eqref{cri}, unless otherwise stated.
In Remarks \ref{Re1-25-3}--\ref{Re11-25-3}, \ref{Re4-25-3}, \ref{Re2-25-3} and \ref{Re3-25-3}, we also discuss the case $N = 5$.

\medskip
Given numbers $\mu_{00},\, r_0,\, \lambda_0 > 1$ to be determined in Section \ref{sec:exist}, we define the configuration space of parameters
\begin{equation}\label{mcp}
\mcp := (\mu_{00}^{-1},\mu_{00}) \times (r_0^{-1},r_0) \times (\lambda_0^{-1},\lambda_0).
\end{equation}
For any $(\mu_0,r,\lambda) \in \mcp$ and $k \in \N$, we set
\[x_j := \(r\cos\(\tfrac{2(j-1)\pi}{k}\), r\sin\(\tfrac{2(j-1)\pi}{k}\),{\bf{0}}\) \in \R^2 \times \R^{N-2}, \quad j=1,2,\ldots,k,\]
and $\mu := \lambda k^{\frac{(p+1)(N-2)}{N}}$. Let also
\begin{equation}\label{Vdef}
(U_j, V_j) := (U_{x_j,\mu}, V_{x_j,\mu}), \quad V := \sum_{j=1}^k V_j,
\end{equation}
and $U$ be the unique solution of
\begin{equation}\label{Udef}
-\Delta U = V^p \quad \text{in } \R^N, \quad U \in \dot{W}^{2,\frac{p+1}{p}}(\R^N).
\end{equation}
Our approximate solution to system \eqref{2} is
\begin{equation}\label{U*V*}
\begin{aligned}
(U_*[\mu_0,r,\mu], V_*[\mu_0,r,\mu]) &:= (U_*,V_*) \\
&:= (U_{0, \mu_0}, V_{0, \mu_0})-(U, V).
\end{aligned}
\end{equation}

\medskip
The main result of this section is an expansion of the energy of $(U_*,V_*)$. The energy functional corresponding to \eqref{2} is
\begin{equation}\label{a1}
I(u,v) := \int_{\R^N} \nabla u \cdot \nabla v -\frac{1}{p+1} \int_{\R^N} |v|^{p+1} -\frac{1}{q+1} \int_{\R^N} |u|^{q+1}.
\end{equation}
\begin{proposition}\label{prop:Iexpan}
Assume that $N \ge 6$. For $k \in \N$ large enough, we have
\[I(U_*, V_*) = (k+1)A + k \left[-\frac{B_1k^{N-2}}{r^{N-2}\mu^{N-2}}
+ \frac{B_2 U_{0,\mu_0}(r)}{\mu^{\frac{N}{q+1}}} + O\(\frac{1}{\mu^{\frac{N}{q+1}+\sigma}}\)\right]\]
uniformly in $\mcp$, where $A := I(U_{0,\mu_0}, V_{0,\mu_0}) = I(U_{0,1}, V_{0,1})$, $B_1$ and $B_2$ are positive constants depending only on $N$ and $p$, and $\sigma>0$ is a sufficiently small number.
\end{proposition}
To prove Proposition~\ref{prop:Iexpan}, we need to analyze the function
\[\varphi(y) := U(y) - \sum_{j=1}^k U_j(y), \quad y \in \R^N.\]
For this aim, we will make use of the following auxiliary function and quantities:
\begin{itemize}
\item[-] Let $w$ be the solution of
\begin{equation}\label{4-29-7}
\begin{cases}
-\Delta w= V_{0,1}^{p-1} &\text{in } \R^N,\\
w(x)\to 0 &\text{as } |x|\to +\infty.
\end{cases}
\end{equation}
Indeed, equation \eqref{4-29-7} has a unique solution since $(p-1)(N-2)>2$. The function $w$ is positive in $\R^N$.
\item[-] According to \cite[Appendix A]{MM}, there exists a constant $\wtb_{11} > 0$ depending only on $N$ such that
\begin{equation}\label{xjx1sum}
\sum_{j=2}^k \frac{1}{|x_j-x_1|^{N-2}}
= \frac{\wtb_{11} k^{N-2}}{r^{N-2}} + k^{N-2} \times \begin{cases}
O(k^{-2}) &\text{if } N \ge 6,\\
O(k^{-1}\log k) &\text{if } N = 5.
\end{cases}
\end{equation}
\end{itemize}
By symmetry, it suffices to consider the function $\varphi(y)$ for $y \in \Omega_1$, where
\begin{equation}\label{Omegaj}
\Omega_j := \left\{y=(y',y'')\in \R^{2}\times \R^{N-2}: \Big\langle \frac{y'}{|y'|},\frac{x'_j}{|x_j|} \Big\rangle \ge \cos\frac{\pi}{k}\right\}, \quad j=1,2,\ldots,k.
\end{equation}
Let $B_{11} := pb_{N,p}\wtb_{11} > 0$, where the number $b_{N,p}$ is defined in \eqref{eq:HV}.
\begin{lemma}\label{nnl2-19-1}
Assume that $N \ge 5$. For $k \in \N$ large enough, we have
\begin{equation}\label{nn1-19-1}
\varphi(y) = \frac{B_{11}k^{N-2}}{r^{N-2} \mu^{\frac{N}{p+1}}} w(\mu(y-x_1)) + O\(\frac{k^{N-2}}{\mu^{\frac{N}{p+1}+\sigma}}\), \quad y \in \Omega_1
\end{equation}
uniformly in $\mcp$, where $r = |x_1|$ and $\sigma>0$ is a sufficiently small number.
\end{lemma}
\begin{proof}
Let $\tau := \frac{N}{(p+1)(N-2)} \in (0,1)$ so that $k \simeq \mu^{\tau}$, and $\kappa \in (\tau,1)$ be a fixed number slightly greater than $\tau$. The representation formula for $\varphi$ yields
\begin{equation}\label{nn2-19-1}
\varphi(y) = \int_{\R^N}\frac{\ga_N}{|y-z|^{N-2}} \left[\bigg(\sum_{j=1}^k V_j\bigg)^p- \sum_{j=1}^k V_j^p \right](z)\, dz > 0, \quad y \in \R^N,
\end{equation}
where $\ga_N := \frac{1}{(N-2)|\Sn^{N-1}|}$. We will establish \eqref{nn1-19-1} by analyzing the right-hand side of \eqref{nn2-19-1}.
Throughout the proof, $\delta > 0$ and $\sigma > 0$ denote sufficiently small numbers.

\medskip
It holds that
\begin{equation}\label{Vj}
\begin{aligned}
\sum_{j=2}^k V_j(y) &\le \frac{C}{\mu^{\frac{N}{q+1}}} \sum_{j=2}^k \frac{1}{|y-x_j|^{N-2}} \\
&\le \frac{C}{\mu^{\frac{N}{q+1}}} \sum_{j=2}^k \frac{1}{|x_j-x_1|^{N-2}} \le \frac{Ck^{N-2}}{\mu^{\frac{N}{q+1}}} \simeq C \mu^{\frac{N}{p+1}-\frac{N}{q+1}}, \quad y \in \Omega_1,
\end{aligned}
\end{equation}
which implies
\begin{equation}\label{Vj2}
\sum_{j=2}^k V_j \le CV_1 \quad \text{in } B_{\mu^{-\kappa}}(x_1)
\end{equation}
for some large $C > 0$. Using \eqref{Vj2} and the elementary inequality
\begin{equation}\label{eq:ele}
|1-t|^p = 1-pt+O(t^2)=1-pt+O(t^\alpha) \quad \text{for all } 0\le t\le c,
\end{equation}
where $1<\alpha\le 2$ and $c > 0$ are fixed numbers, we estimate
\begin{align}
&\quad \int_{B_{\mu^{-\kappa}}(x_1)} \frac{\ga_N}{|y-z|^{N-2}} \left[\bigg(\sum_{j=1}^k V_j\bigg)^p- \sum_{j=1}^k V_j^p \right](z)\, dz \nonumber \\
&= p\int_{B_{\mu^{-\kappa}}(x_1)} \frac{\ga_N}{|y-z|^{N-2}} \bigg(V_1^{p-1}\sum_{j=2}^k V_j\bigg)(z)\, dz
- \int_{B_{\mu^{-\kappa}}(x_1)}\frac{\ga_N}{|y-z|^{N-2}} \bigg(\sum_{j=2}^k V_j^p\bigg)(z)\, dz \nonumber \\
&\ + O\(\int_{B_{\mu^{-\kappa}}(x_1)}\frac{1}{|y-z|^{N-2}} \Bigg[V_1^{p-1-\delta}\bigg(\sum_{j=2}^k V_j\bigg)^{1+\delta}\Bigg](z)\, dz\). \label{1-29-7}
\end{align}
By \eqref{V10est} and \eqref{xjx1sum},
\begin{align}
&\quad p\int_{B_{\mu^{-\kappa}}(x_1)} \frac{\ga_N}{|y-z|^{N-2}} \bigg(V_1^{p-1}\sum_{j=2}^k V_j\bigg)(z)\, dz \nonumber \\
&= p\left[\mu^{\frac{N}{p+1}} \sum_{j=2}^k \frac{b_{N,p}}{(\mu|x_j-x_1|)^{N-2}} + O\(\frac{k^{N-2}}{\mu^{\frac{N}{q+1}+\sigma}}\)\right] \int_{B_{\mu^{-\kappa}}(x_1)}\frac{\ga_N}{|y-z|^{N-2}} V_1^{p-1}(z)\, dz \nonumber \\
&= \left[\frac{pb_{N,p}\wtb_{11} k^{N-2}}{r^{N-2}\mu^{\frac{N}{q+1}}} + O\(\frac{k^{N-2}}{\mu^{\frac{N}{q+1}+\sigma}}\)\right] \int_{B_{\mu^{-\kappa}}(x_1)}\frac{\ga_N}{|y-z|^{N-2}} V_1^{p-1}(z)\, dz. \label{2-29-7}
\end{align}
In addition,
\begin{equation}\label{3-29-7}
\begin{aligned}
&\quad \int_{B_{\mu^{-\kappa}}(x_1)}\frac{\ga_N}{|y-z|^{N-2}} V_1^{p-1}(z)\, dz \\
&= \frac1{\mu^{2- \frac{(p-1)N}{p+1} }} \int_{B_{\mu^{1-\kappa}}(0)}\frac{\ga_N}{ |z-\mu(y-x_1)|^{N-2}} V_{0,1}^{p-1}(z)\,dz\\
&= \frac{w(\mu(y-x_1))}{\mu^{2- \frac{(p-1)N}{p+1} }} - \frac1{\mu^{2- \frac{(p-1)N}{p+1} }}  \int_{\R^N\setminus B_{\mu^{1-\kappa}}(0)}\frac{\ga_N}{ |z-\mu(y-x_1)|^{N-2}} V_{0,1}^{p-1}(z)\,dz.
\end{aligned}
\end{equation}
If $\mu|y-x_1|\le \frac{1}{2}\mu^{1-\kappa}$, then
\begin{equation}\label{5-29-7}
\int_{\R^N\setminus B_{\mu^{1-\kappa}}(0)}\frac{\ga_N}{ |z-\mu(y-x_1)|^{N-2}} V_{0,1}^{p-1}(z)\,dz
\le C\int_{\R^N\setminus B_{\mu^{1-\kappa}}(0)}\frac 1{|z|^{p(N-2)}} \le \frac{C}{\mu^{\sigma}}.
\end{equation}
If $\mu|y-x_1|> \frac{1}{2}\mu^{1-\kappa}$, then
\begin{equation}\label{6-29-7}
\int_{\R^N\setminus B_{\mu^{1-\kappa}}(0)}\frac{1}{ |z-\mu(y-x_1)|^{N-2}} V_{0,1}^{p-1}(z)\,dz
\le \frac{C}{(\mu |y-x_1|)^{(p-1)(N-2)-2}} \le \frac{C}{\mu^{\sigma}}.
\end{equation}
It is worth noting that the condition $p(N-2) > N$ is crucially used in \eqref{5-29-7}--\eqref{6-29-7}. From \eqref{2-29-7}--\eqref{6-29-7} and the identity $\frac{N}{q+1}+2 = \frac{pN}{p+1}$, we see
\begin{multline}\label{7-29-7}
p\int_{B_{\mu^{-\kappa}}(x_1)} \frac{\ga_N}{|y-z|^{N-2}} \bigg(V_1^{p-1}\sum_{j=2}^k V_j\bigg)(z)\, dz \\
= \frac{B_{11}k^{N-2}}{r^{N-2}\mu^{\frac{N}{p+1}}} w(\mu(y-x_1)) + O\(\frac{k^{N-2}}{\mu^{\frac{N}{p+1}+\sigma}}\).
\end{multline}
Applying \eqref{Vj}, we also compute
\begin{align}
&\quad \int_{B_{\mu^{-\kappa}}(x_1)}\frac{1}{|y-z|^{N-2}} \Bigg[V_1^{p-1-\delta}\bigg(\sum_{j=2}^k V_j\bigg)^{1+\delta}\Bigg](z)\, dz \nonumber \\
&\le C\left[\sum_{j=2}^k \frac{\mu^{\frac{N}{p+1}}} {(\mu |x_j-x_1|)^{N-2}}\right]^{1+\delta} \int_{B_{\mu^{-\kappa}}(x_1)}\frac{1}{|y-z|^{N-2}} V_1^{p-1-\delta}(z)\, dz \label{1-30-7} \\
&\le C\left[\sum_{j=2}^k \frac1{(\mu |x_j-x_1|)^{N-2}}\right]^{1+\delta} \int_{\R^N}\frac{\mu^{\frac{pN}{p+1}}}{|y-z|^{N-2}} V_{0,1}^{p-1-\delta}(\mu(z-x_1))\, dz \nonumber \\
&\le C \mu^{\frac{pN}{p+1}-2} \left[\sum_{j=2}^k \frac{1} {(\mu |x_j-x_1|)^{N-2}}\right]^{1+\delta} \le \frac{Ck^{N-2}}{\mu^{\frac{N}{p+1}+\sigma}} \nonumber
\end{align}
and
\begin{align}
\int_{B_{\mu^{-\kappa}}(x_1)}\frac{\ga_N}{|y-z|^{N-2}} \bigg(\sum_{j=2}^k V_j^p\bigg)(z)\, dz &\le \int_{B_{\mu^{-\kappa}}(x_1)}\frac{C}{|y-z|^{N-2}} \bigg(V_1^{p-1-\delta}\sum_{j=2}^k V_j^{1+\delta}\bigg)(z)\, dz \nonumber \\
&\le C \mu^{\frac{pN}{p+1}-2} \sum_{j=2}^k \frac{1} {(\mu |x_j-x_1|)^{(1+\delta)(N-2)}} \label{8-29-7} \\
&\le \frac{Ck^{N-2}}{\mu^{\frac{N}{p+1}+\sigma}}. \nonumber
\end{align}
We infer from \eqref{1-29-7} and \eqref{7-29-7}--\eqref{8-29-7} that
\begin{multline}\label{2-30-7}
\int_{B_{\mu^{-\kappa}}(x_1)} \frac{\ga_N}{|y-z|^{N-2}} \left[\bigg(\sum_{j=1}^k V_j\bigg)^p- \sum_{j=1}^k V_j^p \right](z)\, dz \\
= \frac{B_{11}k^{N-2}}{r^{N-2}\mu^{\frac{N}{p+1}}} w(\mu(y-x_1)) + O\(\frac{k^{N-2}}{\mu^{\frac{N}{p+1}+\sigma}}\).
\end{multline}

On the other hand, given any $1< \tith \le N-2$, we have
\begin{equation}\label{e4}
\begin{aligned}
\sum_{j=2}^k\frac{1}{|y-x_j|^{N-2}} &\le \frac{1}{|y-x_1|^{N-2-\tith}}\sum_{j=2}^k\frac{1}{|y-x_j|^{\tith}} \\
&\le \frac{C}{|y-x_1|^{N-2-\tith}}\sum_{j=2}^k\frac{1}{|x_j-x_1|^{\tith}} \le
\frac{Ck^{ \tith }}{|y-x_1|^{N-2-\tith}}
\end{aligned}
\end{equation}
for any $y\in \Omega_1$. It follows that
\begin{equation}\label{1-29-9}
\begin{aligned}
&\quad \int_{\Omega_1\setminus B_{\mu^{-\kappa}}(x_1)}\frac{1}{|y-z|^{N-2}} \bigg(\sum_{j=2}^k V_j\bigg)^p(z) dz \\
&\le \frac{C}{\mu^{\frac{pN}{q+1}}}\int_{\Omega_1\setminus B_{\mu^{-\kappa}}(x_1)}\frac{1}{|y-z|^{N-2}} \(\frac{k^{1+\theta}}{|z-x_1|^{N-3-\theta}}\)^p dz \\
&\le \frac{Ck^{(1+\theta)p}}{\mu^{\frac{pN}{q+1}}} \mu^{\kappa(p(N-3-\theta)-2)} \le \frac{k^{N-2}}{\mu^{\frac{N}{p+1}+\sigma}}
\end{aligned}
\end{equation}
provided $\theta > 0$ so small that $p(N-3-\theta) > 2$, where the last inequality follows from $p(N-2) > N$.
By employing \eqref{1-29-9} and proceeding as in \eqref{5-29-7}--\eqref{6-29-7}, we deduce
\begin{equation}\label{3-30-7}
\begin{aligned}
0 &< \int_{\Omega_1\setminus B_{\mu^{-\kappa}}(x_1)} \frac{\ga_N}{|y-z|^{N-2}} \left[\bigg(\sum_{j=1}^k V_j\bigg)^p- \sum_{j=1}^k V_j^p \right](z)\, dz \\
&\le \int_{\Omega_1\setminus B_{\mu^{-\kappa}}(x_1)} \frac{\ga_N}{|y-z|^{N-2}} \left[\bigg(\sum_{j=1}^k V_j\bigg)^p- V_1^p \right](z)\, dz\\
&\le C\int_{\Omega_1\setminus B_{\mu^{-\kappa}}(x_1)}\frac{1}{|y-z|^{N-2}} \left[V_1^{p-1}\sum_{j=2}^k V_j + \bigg(\sum_{j=2}^k V_j\bigg)^p\right](z)\, dz \le \frac{Ck^{N-2}}{\mu^{\frac{N}{p+1}+\sigma}}.
\end{aligned}
\end{equation}

Putting \eqref{2-30-7} and \eqref{3-30-7} together, we arrive at
\begin{multline}\label{4-30-7}
\int_{\Omega_1} \frac{\ga_N}{|y-z|^{N-2}} \left[\bigg(\sum_{j=1}^k V_j\bigg)^p- \sum_{j=1}^k V_j^p \right](z)\, dz \\
= \frac{B_{11}k^{N-2}}{r^{N-2}\mu^{\frac{N}{p+1}}} w(\mu(y-x_1)) + O\(\frac{k^{N-2}}{\mu^{\frac{N}{p+1}+\sigma}}\).
\end{multline}

\medskip
We next derive
\begin{equation}\label{5-30-7}
\sum_{i=2}^k \int_{\Omega_i} \frac{\ga_N}{|y-z|^{N-2}} \left[\bigg(\sum_{j=1}^k V_j\bigg)^p- \sum_{j=1}^k V_j^p \right](z)\, dz = O\(\frac{k^{N-2}}{\mu^{\frac{N}{p+1}+\sigma}}\).
\end{equation}
Owing to \eqref{Vj} and Lemma~\ref{l5-30-7} below, the left-hand side of \eqref{5-30-7} is bounded by
\begin{align*}
&\quad \sum_{i=2}^k \int_{\Omega_i} \frac{C}{|y-z|^{N-2}} \left[V_i^{p-1} \sum_{j \ne i} V_j + \bigg(\sum_{j \ne i} V_j\bigg)^p\right](z)\, dz \\
&\begin{medsize}
\displaystyle \le \frac{C}{\mu^{\frac{pN}{q+1}}} \sum_{i=2}^k \int_{\Omega_i} \frac{1}{|y-z|^{N-2}} \left[\frac{1}{|z-x_i|^{(p-1)(N-2)}} \sum_{j \ne i} \frac{1}{|z-x_j|^{N-2}} + \(\sum_{j \ne i} \frac{1}{|z-x_j|^{N-2}}\)^p\right] dz
\end{medsize} \\
&= \frac{Ck^{p(N-2)-2}}{\mu^{\frac{pN}{q+1}}} \sum_{i=2}^k \int_{\Omega_i} \frac{1}{|ky-z|^{N-2}} \\
&\hspace{75pt} \times \left[\frac{1}{|z-kx_i|^{(p-1)(N-2)}} \sum_{j \ne i} \frac{1}{|z-kx_j|^{N-2}} + \(\sum_{j \ne i} \frac{1}{|z-kx_j|^{N-2}}\)^p\right] dz \\
&\le \dfrac{Ck^{p(N-2)-2}}{\mu^{\frac{pN}{q+1}}} = O\(\frac{k^{N-2}}{\mu^{\frac{N}{p+1}+\sigma}}\).
\end{align*}

\medskip
Now, \eqref{nn1-19-1} is a direct consequence of \eqref{4-30-7} and \eqref{5-30-7}.
\end{proof}

For future use, we set
\begin{equation}\label{eq:S}
S := B_{\frac{\pi}{2}r_0^{-1}k^{-1}}(x_1) = \left\{y \in \R^N: |y-x_1| < \frac{\pi}{2}r_0^{-1}k^{-1}\right\} \subset \Omega_1.
\end{equation}
Then $U_1 \ge c$ in $S$ for some small $c > 0$. By using \eqref{nn1-19-1} and arguing as in \eqref{Vj}, we also observe that $\sum_{j=2}^k U_j + \varphi \le C$ in $S$. In particular, $U \le CU_1$ in $S$.
\begin{lemma}\label{l1-28-12}
Assume that $N \ge 6$. For $k \in \N$ large enough, we have
\begin{equation}\label{1-28-12}
I(U, V)= k\left[A - \frac{B_1k^{N-2}}{r^{N-2}\mu^{N-2}} + O\(\frac{1}{\mu^{\frac{N}{q+1}+\sigma}}\)\right]
\end{equation}
uniformly in $\mcp$, where $A=I(U_{0,1}, V_{0,1})$, $B_1 > 0$ is a constant depending only on $N$ and $p$, and $\sigma>0$ is a sufficiently small number.
\end{lemma}
\begin{proof}
Throughout the proof, $\delta,\, \sigma,\, \theta \in (0,1)$ denote sufficiently small numbers.

\medskip
By symmetry, it holds that
\begin{align*}
I(U, V) &= \frac1{p+1} \int_{\R^N} \nabla U \cdot \nabla V-\frac1{p+1} \int_{\R^N}V^{p+1}\\
&\ + \(1-\frac1{p+1}\) \int_{\R^N} \nabla U \cdot \nabla V - \frac1{q+1}\int_{\R^N}U^{q+1}\\
&= \(1-\frac1{p+1}\) \int_{\R^N} U \sum_{j=1}^k U_j^q -\frac1{q+1}\int_{\R^N}U^{q+1}.
\end{align*}
We compute
\begin{align*}
\int_{\R^N} U \sum_{j=1}^k U_j^q &= k \int_{\Omega_1} U \sum_{j=1}^k U_j^q \\
&= k \left[\int_{S} \bigg(\sum_{j=1}^k U_j +\varphi\bigg)\sum_{j=1}^k U_j^q + \int_{\Omega_1 \setminus S} U\sum_{j=1}^k U_j^q\right] \\
&= k \left[\int_{S} U_1^{q+1}+ \int_{S} U_1^q \bigg(\sum_{j=2}^k U_j+\varphi\bigg) + \int_{S} U\sum_{j=2}^k U_j^q + \int_{\Omega_1 \setminus S} U\sum_{j=1}^k U_j^q\right]
\end{align*}
and
\begin{multline*}
\int_{\R^N}U^{q+1} = k \left[\int_{S} U_1^{q+1} + (q+1)\int_{S} U_1^q \bigg(\sum_{j=2}^k U_j+\varphi\bigg) \right.\\
\left. + O\(\int_{S} U_1^{q-\delta} \bigg(\sum_{j=2}^k U_j + \varphi\bigg)^{1+\delta}\) + \int_{\Omega \setminus S} U^{q+1}\right].
\end{multline*}
Hence, using \eqref{cri}, we obtain
\begin{equation}\label{eq:IUV}
\begin{aligned}
I(U, V) &= k \left[\frac{2}{N} \int_{S} U_1^{q+1} -\frac 1{p+1}\int_{S} U_1^q \bigg(\sum_{j=2}^k U_j+\varphi\bigg) \right. \\
&\qquad + \(1-\frac1{p+1}\) \int_{S} U\sum_{j=2}^k U_j^q + O\(\int_{S} U_1^{q-\delta} \bigg(\sum_{j=2}^k U_j + \varphi\bigg)^{1+\delta}\) \\
&\qquad \left. + \(1-\frac1{p+1}\) \int_{\Omega_1 \setminus S} U\sum_{j=1}^k U_j^q - \frac1{q+1} \int_{\Omega_1 \setminus S} U^{q+1}\right].
\end{aligned}
\end{equation}

We will estimate each term on the right-hand side of \eqref{eq:IUV}. By \eqref{6} and \eqref{eq:HV},
\begin{equation}\label{eq:IUV1}
\int_{S} U_1^{q+1} = \int_{\R^N} U_{0,1}^{q+1} + O\(( k^{-1}\mu)^{-(q+1)(N-2)-N}\) = \int_{\R^N} U_{0,1}^{q+1} + O\(\mu^{-\frac{N}{q+1}-\sigma}\).
\end{equation}
Moreover, \eqref{U10est} and \eqref{xjx1sum} show
\begin{align*}
\int_{S} U_1^q \sum_{j=2}^k U_j &= \int_{\R^N} U_{0,1}^q \sum_{j=2}^k \frac{a_{N,p}}{(\mu|x_j-x_1|)^{N-2}} + O\(\mu^{-\frac{N}{q+1}-\sigma}\) \\
&= \(a_{N,p}\wtb_{11}\int_{\R^N} U_{0,1}^q\) \frac{k^{N-2}}{r^{N-2}\mu^{N-2}} + O\(\mu^{-\frac{N}{q+1}-\sigma}\),
\end{align*}
while Lemma~\ref{nnl2-19-1} yields
\begin{align*}
\int_{S} U_1^q \varphi &= \int_{S} U_1^q \left[\frac{B_{11}k^{N-2}}{r^{N-2}\mu^{\frac{N}{p+1}}} w(\mu(y-x_1)) + O\(\frac{k^{N-2}}{\mu^{\frac{N}{p+1}+\sigma}}\)\right] \\
&= \(B_{11} \int_{\R^N} U_{0,1}^q w\) \frac{k^{N-2}}{r^{N-2}\mu^{N-2}} + O\(\mu^{-\frac{N}{q+1}-\sigma}\).
\end{align*}
Therefore,
\begin{equation}\label{eq:IUV2}
\int_{S} U_1^q \bigg(\sum_{j=2}^k U_j+\varphi\bigg)= (p+1)B_1 \frac{k^{N-2}}{r^{N-2}\mu^{N-2}} + O\(\mu^{-\frac{N}{q+1}-\sigma}\),
\end{equation}
where, in view of $w>0$,
\begin{align*}
B_1 &:= \frac{\wtb_{11}}{p+1} \(a_{N,p}\int_{\R^N} U_{0,1}^q + pb_{N,p} \int_{\R^N} U_{0,1}^q w\)  > 0.
\end{align*}
It is also easy to see that
\begin{equation}\label{eq:IUV3}
\int_{S} U\sum_{j=2}^k U_j^q \le C\int_{S} U_1 \cdot \mu^{-\frac{qN}{p+1}}k^{q(N-2)} \le \frac {C\mu^{-\frac{qN}{p+1}}k^{q(N-2)}}{k^2 \mu^{ \frac N{p+1} } } =O\(\mu^{-\frac{N}{q+1}-\sigma}\).
\end{equation}
On the other hand, the pointwise estimate of $U$ in Lemma~\ref{lemma:U} implies
\begin{equation}\label{eq:IUV41}
\begin{aligned}
\int_{\Omega_1 \setminus S} U^{q+1} &\le C\mu^N \int_{\Omega_1 \setminus S} \left[\sum_{j=1}^k \frac{1}{|\mu(y-x_j)|^{N-2}}\right]^{q+1} dy \\
&\ + \frac{C}{\mu^{pN}} \int_{\Omega_1 \setminus S} \left[\sum_{j=1}^k \frac{k^{p(N-2)-2}}{(1+k|y-x_j|)^{\min\{p(N-3-\theta)-2, N-2\}}}\right]^{q+1}dy.
\end{aligned}
\end{equation}
By \eqref{e4}, the first integral on the right-hand side of \eqref{eq:IUV41} is bounded by
\begin{multline*}
C\int_{\Omega_1 \setminus S} \frac{\mu^N dy}{|\mu(y-x_1)|^{(N-2)(q+1)}} + C\int_{ \Omega_1\setminus S} \frac{\mu^{N-(N-2)(q+1)} k^{(1+\theta)(q+1)}}{|y-x_1|^{(N-3-\theta)(q+1)}}dy \\
\le C\mu^{-\frac{N}{N-2} \cdot \frac{N}{p+1}} = O\(\mu^{-\frac{N}{q+1}-\sigma}\).
\end{multline*}
If $p(N-3)>N$, the second integral on the right-hand side of \eqref{eq:IUV41} is bounded by
\begin{multline*}
\frac{ C k^{(p(N-2)-N)(q+1)} }{\mu^{(p+1)N- (N-2)(q+1)}   }   \int_{\Omega_1 \setminus S} \mu^N\left[\sum_{j=1}^k \frac{1}{|\mu(y-x_j)|^{N-2}}\right]^{q+1}dy\\
= C \mu^{-(p-\frac{N}{N-2})N   }\int_{\Omega_1 \setminus S} \mu^N\left[\sum_{j=1}^k \frac{1}{|\mu(y-x_j)|^{N-2}}\right]^{q+1}dy
\le  C\mu^{-\frac{N}{N-2} \cdot \frac{N}{p+1}},
\end{multline*}
since $\frac{(p+1)N}{q+1} = (p+1)(N-2)-N$.
We next estimate the term on the right-hand side of \eqref{eq:IUV41} for the case $p(N-3)\le N$. If $N \ge 7$, then
\begin{equation}\label{1-25-3}
[p(N-3)-3](q+1)= \frac{[p(N-3)-3](p+1)N }{ pN - 2p -2 }>N
\end{equation}
holds for every $p>\frac N{N-2}$. Hence, the second  term on the right-hand side of \eqref{eq:IUV41} is bounded by
\begin{align*}
&\quad \frac{ C k^{(p(N-2)-2)(q+1)} }{\mu^{pN} k^N  }  \Bigl[ \int_{\R^N} \frac{1}{(1+|y|)^{[p(N-3-\theta)-2](q+1)}} +
\int_{\R^N} \frac{1}{(1+|y|)^{[p(N-3-2\theta)-3](q+1)}}\Bigr] dy
 \\
&\le \frac{ C k^{(p(N-2)-2)(q+1)} }{\mu^{pN} k^N  } = C\(\frac{k}{\mu}\)^{pN} = O( \mu^{ -\frac{pN}{N-2}\frac{N}{q+1} }).
\end{align*}
If $N = 5, 6$, then \eqref{1-25-3} does not hold for all $p>\frac N{N-2}$. If \eqref{1-25-3} is true, then the above argument continues to work.
Let us assume that $[p(N-3)-3](q+1)\le N$. For $R_0>1$ large such that $|x_i| \le \frac {R_0}2$ for all $i = 1,\ldots,k$, it holds
\begin{equation}\label{2-25-3}
\begin{aligned}
&\quad \frac{1}{\mu^{pN}} \int_{(\Omega_1 \setminus S)\cap \{|y|\ge R_0\}} \left[\sum_{j=1}^k \frac{k^{p(N-2)-2}}{(1+k|y-x_j|)^{p(N-3-\theta)-2}}\right]^{q+1}dy\\
& \le \frac{C}{\mu^{pN}}\int_{(\Omega_1 \setminus S)\cap \{|y|\ge R_0\}}
\left[k \frac{k^{p(N-2)-2}}{(k|y|)^{p(N-3-\theta)-2}}\right]^{q+1}dy
\le \frac{Ck ^{(1+p+p\theta)(q+1)}}{\mu^{pN}},
\end{aligned}
\end{equation}
since $[p(N-3)-2](q+1)>N$ for all $N \ge 5$ and $p>\frac N{N-2}$. Moreover,
\begin{align}
&\quad \frac{1}{\mu^{pN}} \int_{(\Omega_1 \setminus S)\cap \{|y|\le R_0\}} \left[\sum_{j=1}^k \frac{k^{p(N-2)-2}}{(1+k|y-x_j|)^{p(N-3-\theta)-2}}\right]^{q+1}dy \nonumber \\
&\le \frac{C}{\mu^{pN}}\int_{(\Omega_1 \setminus S)\cap \{|y|\le R_0\}}
\left[\frac{k^{[p(N-2)-2](q+1)}}{(k|y-x_1|)^{[p(N-3-\theta)-2](q+1)}}+
\frac{k^{[p(N-2)-2](q+1)}}{(k|y-x_1|)^{[p(N-3-\theta)-3-\theta](q+1)}}\right]dy
\nonumber \\
&\le C\left[{\frac{k ^{(p+1)(q+1)(1+\theta)}}{\mu^{pN}} + \(\frac{k}{\mu}\)^{pN}}\right], \label{3-25-3}
\end{align}
thanks to the assumption $[p(N-3)-3](q+1)\le N$.  Now, it is enough to verify
\[
\frac{k ^{(p+1)(q+1)(1+\theta)}}{\mu^{pN}}\le \frac{C}{\mu^{\frac N{q+1}+\sigma}},
\]
which holds if
\begin{equation}\label{6-25-3}
\tsg_1(p) := pN-\frac{N(q+1)}{N-2}-\frac N{q+1}>0.
\end{equation}
It is easy to check that $\tsg_1'(p)>0$. Given $N = 6$, direct computations show
$\tsg_1(\frac32)>0$, which gives the validity of \eqref{6-25-3} for all $p\in (\frac32, 2]$.
Thus, we have proved that
\begin{equation}\label{eq:IUV4}
0 < \int_{\Omega_1 \setminus S} U\sum_{j=1}^k U_j^q \le \int_{\Omega_1 \setminus S} U^{q+1} = O\(\mu^{-\frac{N}{q+1}-\sigma}\)
\end{equation}
for $N \ge 6$ and all $p \in (\frac{N}{N-2},\frac{N+2}{N-2}]$. Finally, by recalling that $\sum_{j=2}^k U_j + \varphi \le C$ in $S$, we find
\begin{equation}\label{eq:IUV5}
\int_{S} U_1^{q-\delta} \bigg(\sum_{j=2}^k U_j + \varphi\bigg)^{1+\delta} \le C\int_{S} U_1^{q-\delta} = O\(\mu^{-\frac{(1+\delta)N}{q+1}}\).
\end{equation}
Plugging \eqref{eq:IUV1}, \eqref{eq:IUV2}, \eqref{eq:IUV3}, \eqref{eq:IUV4}, and \eqref{eq:IUV5} into \eqref{eq:IUV}, we establish \eqref{1-28-12}.
\end{proof}

\begin{remark}\label{Re1-25-3}
Unfortunately, when $N=5$, we have $\tsg_1(\frac53)<0$, which shows that
\eqref{6-25-3} does not always hold for $p\in (\frac53,\frac73]$.
However, since $\tsg_1(\frac{16}9)>0$, \eqref{6-25-3} remains valid for all $p \in [\frac{16}9,\frac73]$.
\end{remark}

\begin{proof}[Proof of Proposition~\ref{prop:Iexpan}]
As before, $\delta,\, \sigma,\, \theta \in (0,1)$ denote sufficiently small numbers.

\medskip
We have
\begin{align*}
&\quad \int_{\R^N} \nabla U_* \cdot \nabla V_*\\
&= \int_{\R^N} \nabla U_{0,\mu_0} \cdot \nabla V_{0,\mu_0}+\int_{\R^N} \nabla U \cdot \nabla V-\int_{\R^N} \nabla U_{0,\mu_0} \cdot \nabla V- \int_{\R^N} \nabla V_{0,\mu_0} \cdot \nabla U\\
&= \int_{\R^N} \nabla U_{0,\mu_0} \cdot \nabla V_{0,\mu_0}+\int_{\R^N} \nabla U \cdot \nabla V
-\int_{\R^N} V^p_{0,\mu_0} V- \int_{\R^N} U^q_{0,\mu_0}  U.
\end{align*}
Hence, we can write
\begin{equation}\label{1-29-12}
\begin{aligned}
I(U_*, V_*) &= I(U_{0, \mu_0}, V_{0, \mu_0}) +I(U, V)
-\int_{\R^N} V^p_{0,\mu_0} V- \int_{\R^N} U^q_{0,\mu_0}  U\\
&\ -\frac1{p+1}\int_{\R^N} \(|V_{0,\mu_0}- V|^{p+1}-|V|^{p+1}-|V_{0,\mu_0}|^{p+1}\)
\\
&\ -\frac1{q+1}\int_{\R^N} \(|U_{0,\mu_0}- U|^{q+1}-|U|^{q+1}-|U_{0,\mu_0}|^{q+1}\)\\
&= I(U_{0, \mu_0}, V_{0, \mu_0})+I(U, V) + J_1 + J_2,
\end{aligned}
\end{equation}
where
\begin{equation}\label{J1J2}
\begin{cases}
\displaystyle J_1 := -k \int_{\Omega_1} V^p_{0,\mu_0} V-\frac k{p+1}\int_{\Omega_1} \(|V_{0,\mu_0}- V|^{p+1}-|V|^{p+1}-|V_{0,\mu_0}|^{p+1}\), \\
\displaystyle J_2 := -k \int_{\Omega_1} U^q_{0,\mu_0}  U-\frac k{q+1}\int_{\Omega_1} \(|U_{0,\mu_0}- U|^{q+1}-|U|^{q+1}-|U_{0,\mu_0}|^{q+1}\).
\end{cases}
\end{equation}

\medskip
We first examine the term $J_1$. Since \eqref{Vj} implies that $V \le CV_1$ in $S$ for some large $C > 0$, it holds that
\begin{equation}\label{e6}
\int_{S} V^{p-\delta} V_{0,\mu_0}^{1+\delta} \le C\int_{S} V_1^{p-\delta} = O\(\mu^{-\frac{(1+\delta)N}{q+1}}\).
\end{equation}
By applying \eqref{eq:ele} and \eqref{e6}, we get
\begin{equation}\label{e2}
\begin{aligned}
&\quad \int_{S} \(|V_{0,\mu_0}-V|^{p+1}-|V|^{p+1}-|V_{0,\mu_0}|^{p+1}\)\\
&= -(p+1) \int_{S} V^{p} V_{0,\mu_0} + O\(\int_{S} V^{p-\delta} V_{0,\mu_0}^{1+\delta} + \int_{S} V_{0,\mu_0}^{p+1}\)\\
&= -(p+1) \int_{S} V^{p} V_{0,\mu_0} + O\(\mu^{-\frac{(1+\delta)N}{q+1}} + \mu^{- \frac{N^2}{(p+1)(N-2)}}\).
\end{aligned}
\end{equation}
On the other hand,
\begin{equation}\label{e3}
\begin{aligned}
&\quad \int_{\Omega_1\setminus S} \(|V_{0,\mu_0}- V|^{p+1}-|V|^{p+1}-|V_{0,\mu_0}|^{p+1}\)\\
&=-(p+1)\int_{\Omega_1\setminus S} V_{0,\mu_0}^{p} V + O\(\int_{\Omega_1\setminus S} V_{0,\mu_0}^{p-1} V^{2}+\int_{\Omega_1\setminus S} V^{p+1} \).
\end{aligned}
\end{equation}
Combining \eqref{J1J2} and \eqref{e2}--\eqref{e3}, we arrive at
\begin{equation}\label{1-30-12}
\begin{aligned}
J_1 &= k\left[\int_{S} V^{p} V_{0,\mu_0}-\int_{S} V_{0,\mu_0}^{p} V \right. \\
&\qquad \left. + O\(\int_{\Omega_1\setminus S} V_{0,\mu_0}^{p-1} V^{2} + \int_{\Omega_1\setminus S} V^{p+1} \) + O\(\mu^{-\frac{N}{q+1}-\sigma}\)\right].
\end{aligned}
\end{equation}

Using \eqref{e4}, we compute
\begin{align*}
\int_{S} V_{0,\mu_0}^{p}V \le C \int_{S} V &\le C \sum_{j=1}^{k} \int_{S} \frac{\mu^{\frac{N}{p+1}}}{(\mu|y-x_j|)^{N-2}} dy \\
&\le \frac{C}{\mu^{\frac{N}{q+1}}} \int_{S}\(\frac{1}{|y-x_1|^{N-2}}+ \frac{k^{1+\theta}}{|y-x_1|^{N-3-\theta}}\) dy \\
&\le \frac{C}{\mu^{\frac{N}{q+1}}k^2} = O\(\mu^{-\frac{N}{q+1}-\sigma}\).
\end{align*}
Moreover,
\begin{align}
\int_{\Omega_1\setminus S} V_{0,\mu_0}^{p-1} V^{2} &\le C\int_{\Omega_1\setminus S}\frac1{ (1+|y|)^{(p-1)(N-2)}} \(\sum_{j=1}^{k}\frac{1}{\mu^{\frac{N}{q+1}} |y-x_j|^{N-2}}\)^{2} dy \label{e5} \\
&\le \frac C{\mu^{\frac{2N}{q+1}}} \int_{\Omega_1\setminus S}
\frac1{ (1+|y|)^{(p-1)(N-2)}}
\( \frac1{|y-x_1|^{2(N-2)}} + \frac{k^{2(1+\theta)}}{|y-x_1|^{2(N-3-\theta)}} \) dy. \nonumber
\end{align}
From the inequalities
\[
(p-1)(N-2)+2(N-3)>2+2(N-3)>N
\]
valid for $N \ge 5$ and all $p>\frac N{N-2}$, we know that the last integral in \eqref{e5} is bounded by
\begin{equation}\label{e51}
C\mu^{-\frac{2N}{q+1}} \(k^{N-4}+k^{2(1+\theta)+N-6-2\theta}\) \le
C\mu^{-\frac{2N}{q+1}}\mu^{\frac{(N-4)N}{(p+1)(N-2)}}= O\(\mu^{-\frac{N}{q+1}-\sigma}\),
\end{equation}
where the equality follows from
\begin{align*}
&\quad \frac{N}{q+1}-\frac{(N-4)N}{(p+1)(N-2)}=\frac1{p+1} \left[(p+1)(N-2) -\frac{(2N-6)N}{N-2}
 \right]\\
&> \frac1{p+1} \left[\frac{(N-2)(2N-2)}{N-2} -\frac{(2N-6)N}{N-2}
 \right]=\frac1{p+1}\frac4{N-2}>0.
\end{align*}
Invoking \eqref{e5} and \eqref{e51}, we obtain
\[\int_{\Omega_1\setminus S} V_{0,\mu_0}^{p-1} V^{2} = O\(\mu^{-\frac{N}{q+1}-\sigma}\).\]
We also have
\[\int_{\Omega_1\setminus S} V^{p+1}= O\(\mu^{-\frac{(p+1)N}{q+1}}k^{  (p+1)(N-2)-N}\)  = O\(\mu^{-\frac{N}{q+1}-\sigma}\),\]
since
\[\frac{pN}{q+1}-\frac{N[(p+1)(N-2)-N]  }{(p+1)(N-2)} = \frac{N}{q+1}\(p-\frac{N}{N-2}\) > 0.\]

So we have proved that
\begin{equation}\label{1-31-12}
J_1= k\left[\int_{S} V^{p} V_{0,\mu_0} + O\(\mu^{-\(\frac{N}{q+1}+\sigma\)}\)\right].
\end{equation}

\medskip
We next estimate $J_2$. We recall that $U_1\ge c>0$, $0< \varphi\le C$ and
$
\sum_{j=2}^kU_j\le C
$
in $S$. So, as in \eqref{1-30-12}, we can prove that
\begin{align*}
J_2 &= k\left[\int_{S} U^{q} U_{0,\mu_0}- \int_{S} U_{0,\mu_0}^{q} U \right. \\
&\qquad \left. + O\(\int_{\Omega_1\setminus S} U_{0,\mu_0}^{p-1} U^{2} + \int_{\Omega_1\setminus S} U^{q+1} \) + O\(\mu^{-\frac{(1+\delta)N}{q+1}}\)\right].
\end{align*}

In view of $0< \varphi\le C$ in $S$, it is easy to obtain
\[
\int_{S} U_{0,\mu_0}^{q}U\le \int_{S} \bigg(\sum_{j=1}^k U_j +\varphi\bigg)=
O\(\mu^{-\frac{N}{q+1}-\sigma}\).
\]
Also, \eqref{eq:IUV4} gives
\[
\int_{\Omega_1 \setminus S} U^{q+1} = O\(\mu^{-\frac{N}{q+1}-\sigma}\).
\]
As in \eqref{e5} and \eqref{e51}, we can derive
\begin{equation}\label{4-8-3}
\int_{\Omega_1\setminus S} U_{0,\mu_0}^{q-1}\left[ \sum_{j=1}^k \frac{\mu^{\frac{N}{q+1}}}{(1+\mu|y-x_j|)^{N-2}}  \right]^2= O\(\mu^{-\frac{N}{q+1}-\sigma}\).
\end{equation}

Next, we estimate
\begin{equation}\label{3-8-3}
\int_{\Omega_1\setminus S} U_{0,\mu_0}^{q-1}\left[  \frac{1}{\mu^{\frac{pN}{q+1}}} \sum_{j=1}^k \frac{k^{p(N-2)-2}}{(1+k|y-x_j|)^{\min\{p(N-3-\theta)-2, N-2\}}} \right]^2.
\end{equation}
If $p(N-3)>N$, the integral in \eqref{3-8-3} is bounded by
\begin{align*}
&\quad \mu^{- \frac{2pN}{q+1  }   }k^{2(p(N-2)-N)    }\int_{\Omega_1\setminus S} U_{0,\mu_0}^{q-1}\(  \sum_{j=1}^k \frac{1}{|y-x_j|^{N-2}} \)^2\\
\le & C \frac{ k^{2(p(N-2)-N)    } }{ \mu^{ \frac{2pN}{q+1  } -\frac{2N}{p+1}  }  }
\int_{\Omega_1\setminus S} U_{0,\mu_0}^{q-1}\left[  \sum_{j=1}^k \frac{\mu^{\frac{N}{q+1}}}{(\mu|y-x_j|)^{N-2}} \right]^2=O\(\mu^{-\frac{N}{q+1}-\sigma}\),
\end{align*}
where we used $p(N-2)-N = \frac{pN}{q+1}-\frac{N}{p+1}$ and \eqref{4-8-3}.
If $p(N-3)\le N$, then similar to \eqref{2-25-3} and \eqref{3-25-3}, the integral in \eqref{3-8-3} is bounded by
\begin{align*}
&\quad \frac{ C k^{2(p(N-2)-2)} }{\mu^{\frac{2pN}{q+1}}}\left[k^{-N} + k^{-2(p(N-3)-2)+2p\theta}\(1+k^{2(1+\theta)}\)\right] \\
&\le C\left[\mu^{-\frac{2pN}{q+1}-\frac{N}{(p+1)(N-2)}[N-2(p(N-2)-2)]} + \mu^{-\frac{2pN}{q+1}+\frac{2N}{N-2}(1+\theta)}\right] = O\(\mu^{ -\frac{N}{q+1}-\sigma }\),
\end{align*}
since
\begin{align*}
&\quad \frac{2pN}{q+1} +\frac{N}{ (p+1)(N-2)}   [N -  2(p(N-2)-2) ]\\
&=\frac{N}{q+1} \Bigl[2p+ \frac{q-2p-1}{p+1}\frac{N}{N-2}  \Bigr]\\
&=\frac{N}{q+1} \Bigl[2p-\frac{2N}{N-2}+ \frac{q+1}{p+1}\frac{N}{N-2}  \Bigr]
>\frac{N}{q+1}
\end{align*}
for any $N \ge 5$ and $p>\frac N{N-2}$, while the function
\begin{equation}\label{61-25-3}
\tsg_2(p) := \frac{2pN}{q+1}-\frac{2N}{N-2} - \frac{N}{q+1}
\end{equation}
is increasing in $p$ and $\tsg_2(\frac{N}{N-2}) > 0$ for all $N \ge 6$. Therefore, using Lemma~\ref{lemma:U}, together with the above estimates, we have proved that
\[\int_{\Omega_1\setminus S} U_{0,\mu_0}^{q-1} U^{2} = O\(\mu^{-\frac{N}{q+1}-\sigma}\).\]

As a consequence,
\begin{equation}\label{10-31-12}
J_2 = k\left[\int_{S} U^{q} U_{0,\mu_0} + O\(\mu^{-\(\frac{N}{q+1}+\sigma\)}\)\right].
\end{equation}

\medskip
Putting \eqref{1-29-12}, \eqref{1-31-12} and \eqref{10-31-12} together, we obtain
\[I(U_*, V_*)= I(U_{0, \mu_0}, V_{0, \mu_0})+I(U, V) +k\left[\int_{S} U^{q} U_{0,\mu_0}+\int_{S} V^{p} V_{0,\mu_0} + O\(\mu^{-\frac{N}{q+1}-\sigma}\)\right].\]
Furthermore, by employing \eqref{Vj}, we can easily verify
\[\int_{S} V^{p} V_{0,\mu_0} = \int_{S} V_1^{p} V_{0,\mu_0} + O\(\mu^{\frac{N}{p+1}-\frac{N}{q+1}} \int_{S} V_1^{p-1}\)
= \frac {V_{0,\mu_0}(r)}{\mu^{\frac N{p+1}}} \int_{\R^N} V_{0,1}^p + O\(\mu^{-\frac{N}{q+1}-\sigma}\),\]
and
\[\int_{S} U^{q} U_{0,\mu_0} = \frac{U_{0,\mu_0}(r)}{\mu^{\frac N{q+1}}} \int_{\R^N} U_{0,1}^q + O\(\mu^{-\frac{N}{q+1}-\sigma}\).\]
Since $U_{0,\mu_0} = V_{0,\mu_0}$ for $p = \frac{N+2}{N-2}$ by \cite[Lemma 2.7]{So} and $p \le q$, it follows that
\[I(U_*, V_*) = I(U_{0, \mu_0}, V_{0, \mu_0}) + I(U, V) +  k\left[\frac{B_2 U_{0,\mu_0}(r)}{\mu^{\frac N{q+1}}} + O\(\mu^{-\(\frac{N}{q+1}+\sigma\)}\)\right],\]
where $B_2 := \int_{\R^N} U_{0,1}^q$ for $p \in (\frac{N}{N-2},\frac{N+2}{N-2})$ and $B_2 := \int_{\R^N} U_{0,1}^q + \int_{\R^N} V_{0,1}^p$ for $p = \frac{N+2}{N-2}$.
This together with Lemma~\ref{l1-28-12} gives the desired result.
\end{proof}

\begin{remark}\label{Re11-25-3}
When $N=5$, we have that $\tsg_2(\frac53)<0$, so \eqref{61-25-3} does not always hold for $p\in (\frac53,\frac73]$.
Instead, since $\tsg_2(\frac{17}9)>0$, \eqref{61-25-3} remains valid for all $p \in [\frac{17}9,\frac73]$.
By this fact and Remark~\ref{Re1-25-3}, Proposition~\ref{prop:Iexpan} is true for $N = 5$ and $p \in [\frac{17}9,\frac73]$.
\end{remark}

\section{The invertibility of the linear operator for $p \in (\frac N{N-2},\frac{N+2}{N-2}]$}\label{sec:lin}
We will look for solutions to \eqref{2} having the form $(U_*+\omega_1, V_* +\omega_2)$, where $(U_*, V_*) = (U_{0, \mu_0}, V_{0, \mu_0})-(U, V)$; refer to \eqref{Vdef}--\eqref{U*V*}.
For this aim, we need to discuss the invertibility of the linear operator
\begin{equation}\label{3-2}
\begin{aligned}
L_k(\omega_1,\omega_2) &:= (L_{1,k}(\omega_1,\omega_2),\, L_{2,k}(\omega_1,\omega_2))\\
&:= \(-\Delta \omega_1 -p  |V_*|^{p-1} \omega_2,\,  -\Delta \omega_2 -q |U_*|^{q-1} \omega_1\)
\end{aligned}
\end{equation}
on a suitable function space.

\medskip
Henceforth, we will introduce several function spaces that will be used throughout the paper.
Given $h=1,\ldots,N$ and $j=1,\ldots,k$, let $\Phi_j$ be a rotation operator defined as
\begin{equation}\label{Aj}
\Phi_j(\rho\cos\phi, \rho\sin\phi, y'') := \(\rho\cos\(\phi+\tfrac{2(j-1)\pi}k\), \rho\sin\(\phi+\tfrac{2(j-1)\pi}k\), y''\)
\end{equation}
for $\rho > 0$, $\phi \in [0,2\pi)$, $y'' \in \R^{N-2}$, and $\Psi_h$ a reflection operator defined as
\begin{equation}\label{Bi}
\Psi_hy := (y_1,\ldots, y_{h-1}, -y_h, y_{h+1}, \ldots, y_N)
\end{equation}
for $y \in \R^N$. Then we set
\begin{align}
\mathbf{L}_s &= \left\{(u,v): (u,v) \;\text{is  Lebesgue measurable  on }\;\R^N, \right. \nonumber \\
&\qquad (u,v)(\Phi_j(\rho\cos\phi, \rho\sin\phi, y'')) = (u,v)(\rho\cos\phi, \rho\sin\phi, y'') \text{ for } j = 1,\ldots,k, \nonumber \\
&\qquad \left. (u,v)(\Psi_hy) = (u,v)(y) \text{ for } h = 2,\ldots,N\right\}. \label{Hs}
\end{align}

In addition, fixing $\tau = \frac{N}{(p+1)(N-2)} \in (0,1)$ so that $k \simeq \mu^{\tau}$, we define weighted $L^{\infty}(\R^N)$-norms:
\begin{equation}\label{eq:*-norm}
\begin{cases}
\displaystyle \|u\|_{*,1} := \sup_{y\in\R^N} \left[\sum_{j=1}^{k} \frac{\mu^{\frac{N}{q+1}}}{(1+\mu
|y-x_j|)^{\frac{N}{q+1}+\tau}}\right]^{-1}|u(y)|,\\
\displaystyle \|v\|_{*,2} := \sup_{y\in \R^N} \left[\sum_{j=1}^{k} \frac{\mu^{\frac{N}{p+1}}}{(1+\mu
|y-x_j|)^{\frac{N}{p+1}+\tau}}\right]^{-1}|v(y)|,
\end{cases}
\end{equation}
\begin{equation}\label{eq:**-norm}
\begin{cases}
\displaystyle \|f\|_{**,1} := \sup_{y\in \R^N} \left[\sum_{j=1}^{k} \frac{\mu^{\frac{N}{q+1}+2}}{(1+\mu
|y-x_j|)^{\frac{N}{q+1}+2+\tau}}\right]^{-1}|f(y)|,\\
\displaystyle \|g\|_{**,2} := \sup_{y\in \R^N} \left[\sum_{j=1}^{k} \frac{\mu^{\frac{N}{p+1}+2}}{(1+\mu
|y-x_j|)^{\frac{N}{p+1}+2+\tau}}\right]^{-1}|g(y)|,
\end{cases}
\end{equation}
and
\[\|(u,v)\|_* := \|u\|_{*,1}+\|v\|_{*,2} \quad \text{and} \quad \|(f,g)\|_{**} := \|f\|_{**,1}+\|g\|_{**,2}.\]
We also set Banach spaces
\[
\mathbf{X} := \left\{(u,v) \in \mathbf{L}_s\cap [C( \R^N)\times C(\R^N)] : \; \|(u,v)\|_*<+\infty\right\}
\]
and
\[
\mathbf{Y} := \left\{(f,g) \in \mathbf{L}_s\cap [C( \R^N)\times C(\R^N)]: \; \|(f,g)\|_{**}<+\infty  \right\}.
\]
Finally, denoting
\begin{equation}\label{Y012}
Y_{0} := \frac{\partial U_{0,\mu_0}}{\partial \mu_0},
\quad Y_{j,1} := \frac{\partial U_{x_j,\mu}}{\partial r} = \frac{\partial U_j}{\partial r},
\quad Y_{j,2} := \frac{\partial U_{x_j,\mu}}{\partial \mu},
\end{equation}
and
\begin{equation}\label{Z012}
Z_{0} := \frac{\partial V_{0,\mu_0}}{\partial \mu_0},
\quad Z_{j,1} := \frac{\partial V_{x_j,\mu}}{\partial r} = \frac{\partial V_j}{\partial r},
\quad Z_{j,2} := \frac{\partial V_{x_j,\mu}}{\partial \mu}
\end{equation}
for $j = 1,\ldots,k$, we set
\begin{align*}
\mathbf{E} := \Bigg\{(u,v)\in\mathbf{X}:\; & p\int_{\R^N} v V_{0,\mu_0}^{p-1}Z_0 + q\int_{\R^N} uU_{0,\mu_0}^{q-1}Y_0 = 0,\\
&\left. p\int_{\R^N} v \sum_{j=1}^k V_j^{p-1}Z_{j,l} + q\int_{\R^N} u \sum_{j=1}^k U_j^{q-1}Z_{j,l} = 0, \; l=1,2\right\}
\end{align*}
and
\begin{align*}
\mathbf{F} := \Bigg\{(f,g)\in\mathbf{Y}: \;& \int_{\R^N} f Z_0 + \int_{\R^N} g Y_0= 0,\\
&\left. \int_{\R^N} f \sum_{j=1}^k Z_{j,l} + \int_{\R^N} g \sum_{j=1}^k Y_{j,l} = 0, \; l=1,2\right\}.
\end{align*}

\medskip
We are now ready to present the main result of this section.
\begin{proposition}\label{l20-2-4}
Assume that $N \ge 6$. Suppose that
\begin{multline}\label{1-28-1}
L_k(u,v) = (f, g) \\
+ \(c_0 pV_{0,\mu_0}^{p-1}Z_0 + \sum_{l=1}^2 c_l p\sum_{j=1}^k V_j^{p-1}Z_{j,l},\,
c_0 qU_{0,\mu_0}^{q-1}Y_0 + \sum_{l=1}^2 c_l q\sum_{j=1}^k U_j^{q-1}Y_{j,l}\)
\end{multline}
for some $(u,v) \in \mathbf{E}\cap \left[\dot{W}^{2,\frac{p+1}{p}}(\R^N) \times \dot{W}^{2,\frac{q+1}{q}}(\R^N)\right]$,  $(f,g) \in \mathbf{Y}$, and $(c_0,c_1,c_2) \in \R^3$.
Then there exists a constant $C>0$ depending only on $N$ and $p$ such that
\[
\|(u, v)\|_*\le C \|(f, g)\|_{**}
\]
provided $k \in \N$ large.
\end{proposition}

To establish Proposition~\ref{l20-2-4}, we need a few preliminary lemmas.
\begin{lemma}\label{l1-18-2}
Assume that $N \ge 6$. The numbers $c_0$, $c_1$ and $c_2$ in \eqref{1-28-1} satisfy
\[
|c_0| + \mu|c_1| + \mu^{-1}|c_2| \le C\( \frac1{\mu^\sigma} \|(u, v)\|_*+ \|(f, g)\|_{**} \),
\]
where $\sigma>0$ is a fixed constant.
\end{lemma}
\begin{proof}
The numbers $c_0$, $c_1$ and $c_2$ in \eqref{1-28-1} are determined by
\begin{align}
&\quad c_0 \int_{\R^N} \(p V_{0,\mu_0}^{p-1}Z^2_0 + qU_{0,\mu_0}^{q-1}Y_0^2\)
+ \sum_{l=1}^2 c_l \int_{\R^N} \(p\sum_{j=1}^k V_j^{p-1}Z_{j,l}Z_0 + q\sum_{j=1}^k U_j^{q-1}Y_{j,l}Y_0\) \nonumber \\
&= \int_{\R^N} \(Z_0 L_{1,k}(u,v) +Y_0 L_{2,k}(u,v)\) - \int_{\R^N} \(Z_0 f +Y_0 g\) \label{1-18-2}
\end{align}
and
\begin{equation}\label{2-18-2}
\begin{aligned}
&\quad c_0 \int_{\R^N} \(p V_{0,\mu_0}^{p-1}Z_0 Z_{1,m} + qU_{0,\mu_0}^{q-1}Y_0 Y_{1,m}\) \\
&\quad + \sum_{l=1}^2 c_l \int_{\R^N} \(p\sum_{j=1}^k V_j^{p-1}Z_{j,l}Z_{1,m} + q\sum_{j=1}^k U_j^{q-1}Y_{j,l}Y_{1,m}\)\\
&= \int_{\R^N} \(Z_{1,m} L_{1,k}(u,v) +Y_{1,m} L_{2,k}(u,v)\) - \int_{\R^N} \(Z_{1,m} f +Y_{1,m} g\), \quad m = 1,2.
\end{aligned}
\end{equation}
We will estimate each integral in \eqref{1-18-2} and \eqref{2-18-2}.

\medskip
First of all, since $3 < \frac{N}{q+1}+2+\tau \le \frac{N}{p+1}+2+\tau < N$ for $N \ge 5$ and $p \in (\frac{N}{N-2},\frac{N+2}{N-2}]$, we have
\begin{equation}\label{3-18-2}
\begin{aligned}
&\quad \left|\int_{\R^N} \(Z_{0} f +Y_{0} g\)\right| \\
&\le C\sum_{j=1}^k \left[\|f\|_{**,1} \int_{\R^N} \frac{|Z_0(y)| \mu^{\frac{N}{q+1}+2}}{(1+\mu|y-x_j|)^{\frac{N}{q+1}+2+\tau}} dy \right. \\
&\left. \hspace{175pt} + \|g\|_{**,2}\int_{\R^N} \frac{|Y_0(y)| \mu^{\frac{N}{p+1}+2}}{(1+\mu |y-x_j|)^{\frac{N}{p+1}+2+\tau}} dy\right] \\
&\le \frac{C}{\mu^{\tau}} \sum_{j=1}^k \left[\|f\|_{**,1} \int_{\R^N} \frac{|Z_0(y)|}{|y-x_j|^{\frac{N}{q+1}+2+\tau}} dy + \|g\|_{**,2} \int_{\R^N} \frac{|Y_0(y)|}{|y-x_j|^{\frac{N}{p+1}+2+\tau}} dy \right] \\
&\le \frac{Ck}{\mu^\tau}\|(f, g)\|_{**}\le C\|(f, g)\|_{**}
\end{aligned}
\end{equation}
and
\begin{equation}\label{31-18-2}
\begin{aligned}
&\quad \left|\int_{\R^N} \(Z_{1,2} f +Y_{1,2} g\)\right| \\
&\le \frac{C}{\mu} \sum_{j=1}^k \left[\|f\|_{**,1} \int_{\R^N} \frac{\mu^{N}}{(1+|z-\mu(x_j-x_1)|)^{\frac{N}{q+1}+2+\tau}} \frac{1}{(1+|z|)^{N-2}} dz \right. \\
&\left. \hspace{55pt} + \|g\|_{**,2} \int_{\R^N} \frac{\mu^{N}}{(1+|z-\mu(x_j-x_1)|)^{\frac{N}{p+1}+2+\tau}} \frac{1}{(1+|z|)^{N-2}} dz\right] \\
&\le \frac{C}{\mu} \|(f, g)\|_{**} \left[1+ \sum_{j=2}^k\(\frac{1}{|\mu(x_j-x_1)|^{\frac{N}{q+1}+\tau}} + \frac{1}{|\mu(x_j-x_1)|^{\frac{N}{p+1}+\tau}}\)\right] \\
&\le \frac{C}{\mu} \|(f, g)\|_{**},
\end{aligned}
\end{equation}
where we use \eqref{cri} and set $z = \mu(y-x_1)$ to get the first inequality.
Analogously, we also have
\begin{equation}\label{4-18-2}
\left|\int_{\R^N} \(Z_{1,1} f +Y_{1,1} g\)\right| \le C\mu \|(f, g)\|_{**}.
\end{equation}

\medskip
Second, we calculate
\begin{align*}
&\quad \left|\int_{\R^N} \(Z_0 L_{1,k}(u,v) +Y_0 L_{2,k}(u,v)\)\right| \\
& \le \left|\int_{\R^N} p Z_0\(|V_*|^{p-1}-V_{0,\mu_0}^{p-1}\)v\right|
+ \left|\int_{\R^N} q Y_0\(|U_*|^{q-1}-U_{0,\mu_0}^{q-1}\)u\right|=: J_3+J_4.
\end{align*}

We estimate $J_3$. Because $p\le 2$ for $N \ge 6$, we get
\begin{equation}\label{6-18-2}
\begin{aligned}
J_3 &\le C\|v\|_{*,2}\int_{\R^N}|Z_0(y)| V^{p-1}(y)
\sum_{j=1}^{k}\frac{\mu^{\frac{N}{p+1}}}{(1+\mu|y-x_j|)^{\frac{N}{p+1}+\tau}} dy.
\end{aligned}
\end{equation}
Using Lemma~\ref{l10-18-2}, we deduce
\begin{align*}
&\quad \int_{\R^N\setminus \cup_{i=1}^k B_{\mu^{-1+\sigma'}}(x_i)}|Z_0(y)| V^{p-1}(y)\sum_{j=1}^{k}\frac{\mu^{\frac{N}{p+1}}}{(1+\mu
|y-x_j|)^{\frac{N}{p+1}+\tau}} dy \\
&\le \frac{C}{\mu^{\sigma}}\int_{\R^N\setminus \cup_{i=1}^k B_{\mu^{-1+\sigma'}}(x_i)}|Z_0(y)| \left[\sum_{j=1}^{k}\frac{\mu^{\frac{N}{p+1}}}{(1+\mu
|y-x_j|)^{\frac{N}{p+1}+\tau}}\right]^p dy \\
&\le \frac{C}{\mu^{\sigma}}\int_{\R^N\setminus \cup_{i=1}^k B_{\mu^{-1+\sigma'}}(x_i)}|Z_0(y)| \sum_{j=1}^{k}\frac{\mu^{\frac{N}{q+1}+2}}{(1+\mu
|y-x_j|)^{\frac{N}{q+1}+2+\tau}} dy \\
&\le \frac{C}{\mu^{\sigma+\tau}}\int_{\R^N\setminus \cup_{i=1}^k B_{\mu^{-1+\sigma'}}(x_i)}|Z_0(y)| \sum_{j=1}^{k}\frac{1}{|y-x_j|^{\frac{N}{q+1}+2+\tau}} dy \le \frac{C}{\mu^{\sigma}}
\end{align*}
for $\sigma,\, \sigma' >0$ small. Furthermore, for each $i = 1,\ldots,k$,
\begin{align*}
&\quad \int_{B_{\mu^{-1+\sigma'}}(x_i)}|Z_0(y)| V^{p-1}(y)\sum_{j=1}^{k}\frac{\mu^{\frac{N}{p+1}}}{(1+\mu
|y-x_j|)^{\frac{N}{p+1}+\tau}} dy \\
&\le C\int_{B_{\mu^{-1+\sigma'}}(x_i)}|Z_0(y)| \left[\sum_{j=1}^{k}\frac{\mu^{\frac{N}{p+1}}}{(1+\mu
|y-x_j|)^{\frac{N}{p+1}+\tau}}\right]^p dy \\
&\le  C\mu^{\frac{pN}{p+1}}\int_{ B_{\mu^{-1+\sigma'}}(x_i)} \left[\frac{1}{(1+\mu|y-x_i|)^{\frac{pN}{p+1}+p\tau}}+\frac{1}{(1+\mu|y-x_i|)^{\frac{pN}{p+1}}}\right] dy \\
&\le \frac{C}{\mu^{\frac {N(1-\sigma')}{p+1}}  }.
\end{align*}
Thus,
\begin{equation}\label{10-18-2}
\int_{\R^N}|Z_0(y)| V^{p-1}(y)\sum_{j=1}^{k}\frac{\mu^{\frac{N}{p+1}}}{(1+\mu
|y-x_j|)^{\frac{N}{p+1}+\tau}} dy \le \frac{C}{\mu^{\sigma}}.
\end{equation}

We now proceed to estimate $J_4$. Using Lemma~\ref{l1-23-4}, we can prove that
\begin{align*}
&\quad \int_{\R^N\setminus \cup_{i=1}^k B_{\mu^{-1+\sigma'}}(x_i)} |Y_0(y)| U^{q-1}(y) \sum_{j=1}^{k}\frac{\mu^{\frac{N}{q+1}}}{(1+\mu
|y-x_j|)^{\frac{N}{q+1}+\tau}} dy \\
&\le \frac{C}{\mu^{\sigma}} \int_{\R^N\setminus \cup_{i=1}^k B_{\mu^{-1+\sigma'}}(x_i)}|Y_0(y)| \left[\sum_{j=1}^{k}\frac{\mu^{\frac{N}{q+1}}}{(1+\mu
|y-x_j|)^{\frac{N}{q+1}+\tau}}\right]^q dy \\
&\le \frac{C}{\mu^{\sigma}}\int_{\R^N\setminus \cup_{i=1}^k B_{\mu^{-1+\sigma'}}(x_i)}|Y_0(y)| \sum_{j=1}^{k}\frac{\mu^{\frac{N}{p+1}+2}}{(1+\mu
|y-x_j|)^{\frac{N}{p+1}+2+\tau}} dy \le \frac{C}{\mu^{\sigma}}
\end{align*}
for $\sigma,\, \sigma' >0$ small. Moreover, by applying $U \le C\sum_{j=1}^k U_j$ in $\cup_{i=1}^k B_{\mu^{-1+\sigma'}}(x_i)$, we can show that
\begin{align*}
&\quad \int_{B_{\mu^{-1+\sigma'}}(x_i)}|Y_0(y)| U^{q-1}(y) \sum_{j=1}^{k}\frac{\mu^{\frac{N}{q+1}}}{(1+\mu
|y-x_j|)^{\frac{N}{q+1}+\tau}} dy \\
&\le C\int_{B_{\mu^{-1+\sigma'}}(x_i)}|Y_0(y) | U_i^{q-1}(y) \sum_{j=1}^{k}\frac{\mu^{\frac{N}{q+1}}}{(1+\mu
|y-x_j|)^{\frac{N}{q+1}+\tau}} dy \\
&\le \frac{C}{\mu^{\frac {N(1-\sigma')}{q+1}}  }.
\end{align*}
So we have proved that
\begin{equation}\label{11-18-2}
\int_{\R^N}|Y_0(y)| U^{q-1}(y)
\sum_{j=1}^{k}\frac{\mu^{\frac{N}{q+1}}}{(1+\mu
|y-x_j|)^{\frac{N}{q+1}+\tau}} dy \le \frac{C}{\mu^{\sigma}}.
\end{equation}
If $q\le 2$, then we infer from \eqref{11-18-2} that
\begin{equation}\label{nn6-18-2}
\begin{aligned}
J_4
&\le C\|u\|_{*,1}\int_{\R^N}|Y_0(y)| U^{p-1}(y)
\sum_{j=1}^{k}\frac{\mu^{\frac{N}{q+1}}}{(1+\mu|y-x_j|)^{\frac{N}{q+1}+\tau}} dy
\le \frac{C}{\mu^{\sigma}}\|u\|_{*,1}.
\end{aligned}
\end{equation}
If $q>2$, a similar argument to that of $J_3$ yields the desired estimate.

Combining \eqref{6-18-2}--\eqref{10-18-2} and \eqref{nn6-18-2}, we obtain a number $\sigma>0$ such that
\begin{equation}\label{12-18-2}
\left|\int_{\R^N} \left[Z_0 L_{1,k}(u,v) +Y_0 L_{2,k}(u,v)\right]\right| \le \frac {C}{\mu^\sigma} \|(u, v)\|_*.
\end{equation}

\medskip
Third, for $m = 1,2$, we evaluate
\begin{align*}
&\quad \left| \int_{\R^N} \left[ Z_{1,m} L_{1,k}(u,v) + Y_{1,m} L_{2,k}(u,v)\right]\right|\\
&\le \left|\int_{\R^N} \left[p\(V_1^{p-1}- V^{p-1}\)Z_{1,m}v + q\(U_1^{q-1}-U^{q-1}\)Y_{1,m}u \right]\right|\\
&\ + \left|\int_{\R^N} \left[p\(|V_*|^{p-1}- V^{p-1}\)Z_{1,m}v + q\(|U_*|^{q-1}-U^{q-1}\)Y_{1,m}u \right]\right|\\
&=: J_{5,m} + J_{6,m}.
\end{align*}
It is easy to check that
\begin{equation}\label{11-19-3}
\begin{aligned}
&\int_{\R^N} \left[\bigg(\sum_{j=2}^k V_j(y)\bigg)^{p-1}  V_1(y) \sum_{j=1}^{k}\frac{\mu^{\frac{N}{p+1}}}{(1+\mu
|y-x_j|)^{\frac{N}{p+1}+\tau}} dy \right]
\\
&\ + \int_{\R^N}  \left[\bigg(\sum_{j=2}^k U_j(y)\bigg)^{q-1}  U_1(y) \sum_{j=1}^{k}\frac{\mu^{\frac{N}{q+1}}}{(1+\mu
|y-x_j|)^{\frac{N}{q+1}+\tau}} dy \right]
\le \frac {C }{\mu^{\sigma}}
\end{aligned}
\end{equation}
and
\begin{equation}\label{12-19-3}
\begin{aligned}
& \int_{\R^N} \left[\sum_{j=2}^k \(V_j V_1^{p-1}\)(y) \sum_{j=1}^{k}\frac{\mu^{\frac{N}{p+1}}}{(1+\mu
|y-x_j|)^{\frac{N}{p+1}+\tau}} dy \right]
\\
&\ + \int_{\R^N}  \left[\sum_{j=2}^k \(U_j U_1^{q-1}\)(y) \sum_{j=1}^{k}\frac{\mu^{\frac{N}{q+1}}}{(1+\mu
|y-x_j|)^{\frac{N}{q+1}+\tau}}dy \right]
\le \frac {C }{\mu^{\sigma}}.
\end{aligned}
\end{equation}
We see that \eqref{11-19-3} and \eqref{12-19-3} imply
\[
J_{5,m} \le \frac {C}{\mu^{\alpha(m)+\sigma}}\|(u,v)\|_*,
\]
where $\alpha(1) := -1$ and $\alpha(2) := 1$. Similarly, we can derive
\[J_{6,m} \le \frac {C }{\mu^{\alpha(m)+\sigma}}\|(u,v)\|_*.\]
So we have proved
\begin{equation}\label{2-23-2}
\left| \int_{\R^N} \left[Z_{1,m} L_{1,k}(u,v) +Y_{1,m} L_{2,k}(u,v)\right]\right| \le \frac {C}{\mu^{\alpha(m)+\sigma}}\|(u,v)\|_*.
\end{equation}

\medskip
Finally, we observe
\begin{equation}\label{3-23-2}
\int_{\R^N} \(p V_{0,\mu_0}^{p-1}Z^2_0+qU_{0,\mu_0}^{q-1}Y_0^2\) >0,
\end{equation}
\begin{equation}\label{4-23-2}
\int_{\R^N} \(p\sum_{j=1}^k V_j^{p-1}Z_{j,l}Z_0 + q\sum_{j=1}^k U_j^{q-1}Y_{j,l}Y_0\)
= o\(\frac 1{\mu^{\alpha(l)}}\),
\end{equation}
\begin{equation}\label{5-23-2}
\int_{\R^N} \(p\sum_{j=1}^k V_j^{p-1}Z_{j,l}Z_{1,m} + q\sum_{j=1}^k U_j^{q-1}Y_{j,l}Y_{1,m}\)
= \frac 1{\mu^{\alpha(l)+\alpha(m)}}\delta_{lm} (a_l+o(1)),
\end{equation}
and
\begin{equation}\label{6-23-2}
\begin{aligned}
\int_{\R^N} \(p V_{0,\mu_0}^{p-1}Z_0  Z_{1,m} + qU_{0,\mu_0}^{q-1}Y_0  Y_{1,m}\) = o\(\frac 1{\mu^{\alpha(m)}}\),
\end{aligned}
\end{equation}
where $l,m=1,2$, $\delta_{lm}=1$ if $l=m$, $\delta_{lm}=0$ if $l\ne m$, and $a_l>0$.

\medskip
Using  \eqref{3-23-2}--\eqref{6-23-2}, \eqref{3-18-2}--\eqref{4-18-2},
\eqref{12-18-2} and \eqref{2-23-2},
we can solve \eqref{1-18-2}--\eqref{2-18-2} to obtain the desired estimates
for $c_0$, $c_1$, and $c_2$.
\end{proof}

\begin{lemma}\label{l10-23-2}
Assume that $N \ge 5$. For $x \in \R^N$, we have
\begin{equation}\label{1-24-2}
\int_{\R^N} \frac{\left|(c_0 V_{0,\mu_0}^{p-1}Z_0)(y)\right|}{|y-x|^{N-2}} dy
\le C\(\frac1{\mu^\sigma}\|(u, v)\|_* + \|(f, g)\|_{**}\) \sum_{j=1}^{k}\frac{\mu^{\frac{N}{q+1}}}{(1+\mu |x-x_j|)^{\frac{N}{q+1}+\tau}},
\end{equation}
\begin{equation}\label{5-24-2}
\int_{\R^N} \frac{\left|(c_0 U_{0,\mu_0}^{q-1}Y_0)(y)\right|}{|y-x|^{N-2}} dy
\le C\(\frac1{\mu^\sigma}\|(u, v)\|_* + \|(f, g)\|_{**}\) \sum_{j=1}^{k}\frac{\mu^{\frac{N}{q+1}}}{(1+\mu |x-x_j|)^{\frac{N}{q+1}+\tau}},
\end{equation}
and
\begin{equation}\label{2-24-2}
\begin{medsize}
\displaystyle \int_{\R^N} \frac1{|y-x|^{N-2}} \Bigg|\Bigg(c_h \sum_{j=1}^k V_j^{p-1}Z_j\Bigg)(y)\Bigg| dy
\le C\(\frac1{\mu^\sigma}\|(u, v)\|_* + \|(f, g)\|_{**}\) \sum_{j=1}^{k}\frac{\mu^{\frac{N}{q+1}}}{(1+\mu |x-x_j|)^{\frac{N}{q+1}+\tau}}.
\end{medsize}
\end{equation}
Moreover, if $N \ge 7$, then for $x \in \R^N$,
\begin{equation}\label{6-24-2}
\begin{medsize}
\displaystyle \int_{\R^N} \frac1{|y-x|^{N-4}} \Bigg|\Bigg(c_h \sum_{j=1}^k U_j^{q-1}Y_j\Bigg)(y)\Bigg| dy \\
\le C\(\frac1{\mu^\sigma}\|(u, v)\|_* + \|(f, g)\|_{**}\) \sum_{j=1}^{k}\frac{\mu^{\frac{N}{q+1}}}{(1+\mu |x-x_j|)^{\frac{N}{q+1}+\tau}}.
\end{medsize}
\end{equation}
Here, $h = 1, 2$, and $\sigma>0$ is a fixed constant.

Let $R>1$ be a fixed large number. If $N = 5, 6$, then \eqref{6-24-2} holds for $x \in B_R(0)$.
\end{lemma}
\begin{proof}
By Lemma~\ref{l1-18-2} and $p(N-2)>N$, we obtain
\begin{align*}
\int_{\R^N} \frac{\left|(c_0 V_{0,\mu_0}^{p-1}Z_0)(y)\right|}{|y-x|^{N-2}} dy
&\le C\(\frac1{\mu^\sigma}\|(u, v)\|_* + \|(f, g)\|_{**}\) \int_{\R^N} \frac1{|y-x|^{N-2}} V_{0,\mu_0}^{p}(y)\,dy\\
&\le C\(\frac1{\mu^\sigma}\|(u, v)\|_* + \|(f, g)\|_{**}\) \frac1{(1+|x|)^{N-2}}\\
&\le C\(\frac1{\mu^\sigma}\|(u, v)\|_* + \|(f, g)\|_{**}\) \sum_{j=1}^{k}\frac{\mu^{\frac{N}{q+1}}}{(1+\mu |x-x_j|)^{\frac{N}{q+1}+\tau}},
\end{align*}
since $\frac{N}{q+1}+\tau < N-2$ implies
\[
\frac1k \frac1{(1+|x|)^{N-2}} \simeq
\frac C{\mu^{\tau}} \frac1{(1+|x|)^{N-2}} \le \frac{C\mu^{\frac{N}{q+1}}}{(1+\mu
|x-x_j|)^{\frac{N}{q+1}+\tau}}
\]
for all $j = 1,\ldots,k$. This proves \eqref{1-24-2}. Inequality \eqref{5-24-2} can be verified in an analogous manner.

On the other hand, by employing $|Z_j| \le \frac {C}{\mu^{\alpha(h)}}V_j$, Lemma~\ref{l1-18-2}, $p(N-2)>N$, and $\frac{N}{q+1}+\tau < N-2$, we find
\begin{align*}
&\quad \int_{\R^N} \frac1{|y-x|^{N-2}} \Bigg|\Bigg(c_h \sum_{j=1}^k V_j^{p-1}Z_j\Bigg)(y)\Bigg| dy
\\
&\le C\(\frac1{\mu^\sigma}\|(u, v)\|_* + \|(f, g)\|_{**}\)
\int_{\R^N} \frac1{|y-x|^{N-2}}\sum_{j=1}^k V_j^{p}(y) \,dy\\
&\le C\(\frac1{\mu^\sigma}\|(u, v)\|_* + \|(f, g)\|_{**}\)
\sum_{j=1}^{k}\frac{\mu^{\frac{pN}{p+1}-2}}{(1+\mu |x-x_j|)^{N-2}} \\
&\le C\(\frac1{\mu^\sigma}\|(u, v)\|_* + \|(f, g)\|_{**}\) \sum_{j=1}^{k}\frac{\mu^{\frac{N}{q+1}}}{(1+\mu|x-x_j|)^{\frac{N}{q+1}+\tau}}.
\end{align*}
Thus, \eqref{2-24-2} follows. If $N \ge 7$, then $\frac{N}{q+1}+\tau \le N-4$ for all $p \in (\frac{N}{N-2},\frac{N+2}{N-2}]$, so we see
\begin{align*}
&\quad \int_{\R^N} \frac1{|y-x|^{N-4}} \Bigg|\Bigg(c_h \sum_{j=1}^k U_j^{q-1}Y_j\Bigg)(y)\Bigg| dy
\\
&\le C\(\frac1{\mu^\sigma}\|(u, v)\|_* + \|(f, g)\|_{**}\)
\sum_{j=1}^{k}\frac{\mu^{\frac{qN}{q+1}-4}}{(1+\mu |x-x_j|)^{N-4}} \\
&\le C\(\frac1{\mu^\sigma}\|(u, v)\|_* + \|(f, g)\|_{**}\) \sum_{j=1}^{k}\frac{\mu^{\frac{N}{q+1}}}{(1+\mu|x-x_j|)^{\frac{N}{q+1}+\tau}},
\end{align*}
where we also used $\frac{qN}{q+1}-4 = \frac{N}{p+1}-2 \le \frac{N}{q+1}$ for the last inequality. Hence, \eqref{6-24-2} holds for all $x \in \R^N$.
If $N = 5$ or $6$, then $\frac{N}{q+1}+\tau \ge N-4$ for all $p \in (\frac{N}{N-2},\frac{N+2}{N-2}]$. In this case,
\begin{multline*}
\frac{\mu^{\frac{qN}{q+1}-4}}{(1+\mu |x-x_j|)^{N-4}} \le \frac{\mu^{\frac{N}{q+1}}}{(1+\mu|x-x_j|)^{\frac{N}{q+1}+\tau}} \quad \text{for } |x| \le R \\
\Leftrightarrow \; \mu^{\frac{N}{q+1}+\tau-(N-4)} \le C\mu^{\frac{N}{q+1}-\frac{N}{p+1}+2}
\Leftrightarrow \; \tau \le \frac{N}{q+1},
\end{multline*}
and the last inequality is always true. Hence, \eqref{6-24-2} holds for all $x \in B_R(0)$.
\end{proof}

\begin{lemma}\label{l20-23-2}
Assume that $N \ge 5$. For $x \in \R^N$, we have
\[\int_{\R^N}\frac{|f(y)|}{|y-x|^{N-2}} dy
\le C\|f\|_{**,1} \sum_{j=1}^{k}\frac{\mu^{\frac{N}{q+1}}}{(1+\mu
|x-x_j|)^{\frac{N}{q+1}+\tau}}\]
and
\[\int_{\R^N}\frac{|g(y)|}{|y-x|^{N-2}} dy
\le C\|g\|_{**,2} \sum_{j=1}^{k}\frac{\mu^{\frac{N}{p+1}}}{(1+\mu
|x-x_j|)^{\frac{N}{p+1}+\tau}}.\]
\end{lemma}
\begin{proof}
The first inequality is true because
\begin{align*}
\int_{\R^N}\frac{|f(y)|}{|y-x|^{N-2}} dy
&\le \|f\|_{**,1} \int_{\R^N}\frac1{|y-x|^{N-2}} \sum_{j=1}^{k}\frac{\mu^{\frac{N}{q+1}+2}}{(1+\mu |y-x_j|)^{\frac{N}{q+1}+2+\tau}} dy \\
&\le C\|f\|_{**,1} \sum_{j=1}^{k}\frac{\mu^{\frac{N}{q+1}}}{(1+\mu |x-x_j|)^{\frac{N}{q+1}+\tau}} dy.
\end{align*}
The second inequality can be verified  in a similar way.
\end{proof}

\begin{lemma}\label{l21-23-2}
Assume that $N \ge 6$. Let $R>1$ be a fixed large number. For $|x|\le R$,
\begin{equation}\label{eq:f4}
\int_{\R^N} \frac{|f(y)|}{|y-x|^{N-4}} dy \le C \|f\|_{**,1}\sum_{j=1}^{k}\frac{\mu^{\frac{N}{p+1}}}{(1+\mu |x-x_j|)^{\frac{N}{p+1}+\tau}}
\end{equation}
and
\begin{equation}\label{eq:g4}
\int_{\R^N} \frac{|g(y)|}{|y-x|^{N-4}} dy \le C \|g\|_{**,2}\sum_{j=1}^{k}\frac{\mu^{\frac{N}{q+1}}}{(1+\mu |x-x_j|)^{\frac{N}{q+1}+\tau}}.
\end{equation}
\end{lemma}
\begin{proof}
We focus only on \eqref{eq:g4}, since \eqref{eq:f4} can be handled analogously. We have
\begin{align*}
\int_{\R^N} \frac{|g(y)|}{|y-x|^{N-4}} dy
&\le C\|g\|_{**,2} \sum_{j=1}^{k}\frac{\mu^{\frac{N}{p+1}-2}}{(1+\mu |x-x_j|)^{\frac{N}{p+1}-2+\tau}} \\
&\le C\|g\|_{**,2} \sum_{j=1}^{k}\frac{\mu^{\frac{N}{q+1}}}{(1+\mu |x-x_j|)^{\frac{N}{q+1}+\tau}},
\end{align*}
since $4 < \frac{N}{q+1}+2+\tau \le \frac{N}{p+1}+2+\tau < N$ for $N \ge 6$ and $p \in (\frac{N}{N-2},\frac{N+2}{N-2}]$, and
\[
\frac{\mu^{\frac{N}{p+1}-2}}{(1+\mu
|x-x_j|)^{\frac{N}{p+1}-2+\tau}}\le C \frac{\mu^{\frac{N}{q+1}}}{(1+\mu
|x-x_j|)^{\frac{N}{q+1}+\tau}}
\]
for $|x|\le R$ and each $j = 1,\ldots,k$. Here we have used $\frac{N}{p+1}-2 \le \frac{N}{q+1}$.
\end{proof}

\begin{proof}[Proof of Proposition~\ref{l20-2-4}]
To the contrary, suppose that there exist sequences $\{(u_k, v_k)\}_{k \in \N} \subset \mathbf{E}$, $\{(f_k,g_k)\}_{k \in \N} \subset \mathbf{Y}$
and $\{(c_{0,k}, c_{1,k})\}_{k \in \N} \subset \R^2$ satisfying \eqref{1-28-1},
$\|(u_k, v_k)\|_{*}=1$ for all $k \in \N$ and $\|(f_k,g_k)\|_{**}\to 0$ as $k \to +\infty$. We write
\[
L_k (u_k, v_k) = L(u_k, v_k)- \(p \(|V_*|^{p-1}-V_{0,\mu_0}^{p-1}\)v_k, q\(|U_*|^{q-1}-U_{0,\mu_0}^{q-1}\)u_k\),
\]
where
\[L(u,v) := \(-\Delta u - p V_{0,\mu_0}^{p-1} v, -\Delta v - q U_{0,\mu_0}^{q-1} u\), \quad (u,v) \in \mathbf{E}.\]

Choose a large number $R_0 > 1$ such that $|x_i| \le \frac {R_0}2$ for all $i = 1,\ldots,k$.
It follows from \eqref{G1k}--\eqref{G2k2}, \eqref{G3k1}--\eqref{G4k2}, \eqref{10-27-1}--\eqref{11-27-1}, and Lemmas~\ref{l10-23-2}--\ref{l21-23-2} that for $x \in B_{R_0}(0)$,
\begin{align}
|u_k(x)| &\le C\int_{\R^N} \frac1{|y-x|^{N-2}} \(| |V_*|^{p-1}-V_{0,\mu_0}^{p-1}||v_k|+|f_k|\)(y)\,dy \nonumber \\
&\ + C\int_{\R^N} \frac1{|y-x|^{N-2}} \left|\(c_0p V_{0,\mu_0}^{p-1}Z_0+ \sum_{h=1}^2c_h p\sum_{j=1}^k V_j^{p-1}Z_j\)(y)\right| dy \nonumber \\
&\ +C\int_{\R^N} \frac1{|y-x|^{N-4}} \(| |U_*|^{q-1}-U_{0,\mu_0}^{q-1}||u_k|+|g_k|\)(y)\,dy \nonumber \\
&\ +C\int_{\R^N} \frac1{|y-x|^{N-2}} \left|\(c_0 qU_{0,\mu_0}^{q-1}Y_0\)(y)\right| dy \label{1-31-1} \\
&\ +C\int_{\R^N} \frac1{|y-x|^{N-4}} \left|\(\sum_{h=1}^2c_h q\sum_{j=1}^k U_j^{q-1}Y_j\)(y)\right| dy \nonumber \\
&\le C\left[\int_{\R^N} \frac{\(V^{p-1}|v_k|\)(y)}{|y-x|^{N-2}} dy + \int_{\R^N} \frac{\(U^{q-1} |u_k|\)(y)}{|y-x|^{N-4}} dy + \int_{\R^N} \frac{\(U_{0,\mu_0}^{q-2}U |u_k|\)(y)}{|y-x|^{N-4}} dy\right] \nonumber \\
&\ + C\(\frac1{\mu^\sigma}\|(u_k,v_k)\|_* + \|(f_k,g_k)\|_{**}\) \sum_{j=1}^{k}\frac{\mu^{\frac{N}{q+1}}}{(1+\mu|x-x_j|)^{\frac{N}{q+1}+\tau}} \nonumber
\end{align}
and
\begin{align}
|v_k(x)| &\le C\left[\int_{\R^N} \frac{\(U^{q-1} |u_k|\)(y)}{|y-x|^{N-2}} dy + \int_{\R^N} \frac{\(U_{0,\mu_0}^{q-2}U |u_k|\)(y)}{|y-x|^{N-2}} dy + \int_{\R^N} \frac{\(V^{p-1}|v_k|\)(y)}{|y-x|^{N-4}} dy\right] \nonumber \\
&\ + C\(\frac1{\mu^\sigma}\|(u_k,v_k)\|_* + \|(f_k,g_k)\|_{**}\) \sum_{j=1}^{k}\frac{\mu^{\frac{N}{p+1}}}{(1+\mu|x-x_j|)^{\frac{N}{p+1}+\tau}}. \label{1-31-13}
\end{align}
In \eqref{1-31-1}--\eqref{1-31-13}, the integrals involving the function $U_{0,\mu_0}^{q-2}U |u_k|$ can be substituted with $0$ provided $q \le 2$, which happens for all $N \ge 9$ and $p \in (\frac{N}{N-2},\frac{N+2}{N-2}]$.

Furthermore, by appealing to the fundamental solution of the Laplace equation in $\R^N$, we obtain that for $x \in \R^N \setminus B_{R_0}(0)$,
\begin{equation}\label{1-31-14}
\begin{aligned}
|u_k(x)| &\le C\int_{\R^N} \frac1{|y-x|^{N-2}}\(V^{p-1}|v_k|+|f_k|\)(y)\, dy + C\int_{\R^N} \frac{\(V_{0,\mu_0}^{p-1}|v_k|\)(y)}{|y-x|^{N-2}} dy \\
&\ + C\int_{\R^N} \frac1{|y-x|^{N-2}} \left|\(c_0p V_{0,\mu_0}^{p-1}Z_0+ \sum_{h=1}^2c_h p\sum_{j=1}^k V_j^{p-1}Z_j\)(y)\right| dy \\
&\le C\left[\int_{\R^N} \frac{\(V^{p-1}|v_k|\)(y)}{|y-x|^{N-2}} dy + \int_{\R^N} \frac{\(V_{0,\mu_0}^{p-1}|v_k|\)(y)}{|y-x|^{N-2}} dy\right] \\
&\ + C\(\frac1{\mu^\sigma}\|(u_k,v_k)\|_* + \|(f_k,g_k)\|_{**}\) \sum_{j=1}^{k}\frac{\mu^{\frac{N}{q+1}}}{(1+\mu|x-x_j|)^{\frac{N}{q+1}+\tau}}
\end{aligned}
\end{equation}
and
\begin{equation}\label{1-31-15}
\begin{aligned}
|v_k(x)| &\le C\left[\int_{\R^N} \frac{\(U^{q-1}|u_k|\)(y)}{|y-x|^{N-2}} dy + \int_{\R^N} \frac{\(U_{0,\mu_0}^{q-1}|u_k|\)(y)}{|y-x|^{N-2}} dy\right] \\
&\ + C\(\frac1{\mu^\sigma}\|(u_k,v_k)\|_* + \|(f_k,g_k)\|_{**}\) \sum_{j=1}^{k}\frac{\mu^{\frac{N}{p+1}}}{(1+\mu|x-x_j|)^{\frac{N}{p+1}+\tau}}.
\end{aligned}
\end{equation}
To reach a contradiction, we will prove that the function
\begin{equation}\label{eq:functuvk}
\left[\sum_{j=1}^{k}\frac{\mu^{\frac{N}{q+1}}}{(1+\mu |x-x_j|)^{\frac{N}{q+1}+\tau}}\right]^{-1}|u_k(x)| +
\left[\sum_{j=1}^{k}\frac{\mu^{\frac{N}{p+1}}}{(1+\mu |x-x_j|)^{\frac{N}{p+1}+\tau}}\right]^{-1}|v_k(x)|
\end{equation}
can achieve its maximum only in $\cup_{i=1}^k \overline{B_{M\mu^{-1}}(x_i)}$ provided $M>1$ large enough.
For this purpose, we now estimate the integrals in the parentheses on the rightmost sides of \eqref{1-31-1}--\eqref{1-31-13} for $x \in B_{R_0}(0) \setminus \cup_{i=1}^k B_{M\mu^{-1}}(x_i)$ and \eqref{1-31-14}--\eqref{1-31-15} for $x \in \R^N \setminus B_{R_0}(0)$.

\medskip
We begin by considering the integrals in \eqref{1-31-1} for $x \in B_{R_0}(0) \setminus \cup_{i=1}^k B_{M\mu^{-1}}(x_i)$.

First, we note that
\begin{equation}\label{12-1-2}
\(V^{p-1}|v_k|\)(y) \le C \|v_k\|_{*,2} \left[\sum_{i=1}^{k}\frac{\mu^{\frac{N}{p+1}}}{(1+\mu
|y-x_i|)^{N-2}} \right]^{p-1}\sum_{j=1}^{k}\frac{\mu^{\frac{N}{p+1}}}{(1+\mu |y-x_j|)^{\frac{N}{p+1}+\tau}}
\end{equation}
for $y \in \R^N$. Since $N-2>\frac{N}{p+1}+\tau$, we see
\begin{multline}\label{n1-12-1-2}
\int_{\R^N \setminus \cup_{j=1}^k B_{M\mu^{-1}}(x_j)} \frac1{|y-x|^{N-2}} \left[\sum_{i=1}^{k}\frac{1}{(1+\mu|y-x_i|)^{N-2}}\right]^{p-1} \sum_{j=1}^{k}\frac{dy}{(1+\mu|y-x_j|)^{\frac{N}{p+1}+\tau}} \\
\le \frac {C}{M^{\sigma}}\int_{\R^N \setminus \cup_{j=1}^k B_{M\mu^{-1}}(x_j)} \frac1{|y-x|^{N-2}} \left[\sum_{j=1}^{k}\frac{1}{(1+\mu|y-x_j|)^{\frac{N}{p+1}+\tau}}\right]^p dy
\end{multline}
for some $\sigma > 0$. From \eqref{12-1-2}, \eqref{n1-12-1-2} and Lemma~\ref{l10-18-2}, we get
\begin{align*}
&\quad \int_{\R^N \setminus \cup_{j=1}^k B_{M\mu^{-1}}(x_j)} \frac{\(V^{p-1}|v_k|\)(y)}{|y-x|^{N-2}} dy \\
&\le \frac {C}{M^{\sigma}} \|v_k\|_{*,2} \int_{\R^N \setminus \cup_{j=1}^k B_{M\mu^{-1}}(x_j)} \frac1{|y-x|^{N-2}} \sum_{j=1}^k \frac{\mu^{\frac{pN}{p+1}}}{ (1+\mu |y-x_j| )^{  \frac{N}{q+1}+2+\tau } } dy \\
&\le \frac {C}{M^{\sigma}} \|v_k\|_{*,2} \sum_{j=1}^k \frac{ \mu^{\frac{N}{q+1}} }{ (1+\mu |x-x_j| )^{  \frac{N}{q+1}+\tau } }.
\end{align*}
Moreover, for $y\in B_{M\mu^{-1}}(x_i)$ with $i = 1,\ldots,k$, it holds that $V \le CV_i$ and
\[
\sum_{j=1}^{k}\frac{1}{(1+\mu|y-x_j|)^{\frac{N}{p+1}+\tau}} \le C \le \frac{C}{(1+\mu|y-x_i|)^{\frac{N}{p+1}+\tau}}.
\]
Thus,
\begin{align*}
&\quad \int_{B_{M\mu^{-1}}(x_i)} \frac{\(V^{p-1}|v_k|\)(y)}{|y-x|^{N-2}} dy \\
&\le C\|v_k\|_{*,2} \int_{  B_{M\mu^{-1}}(x_i)} \frac1{|y-x|^{N-2}} \frac{\mu^{\frac{pN}{p+1}}}{(1+\mu|y-x_i|)^{(N-2)(p-1)+\frac{N}{p+1}+\tau}} dy\\
&\le C\|v_k\|_{*,2} \frac{\mu^{\frac{N}{q+1}}}{(1+\mu|x-x_i|)^{\min\{(N-2)(p-1)+\frac{N}{p+1}+\tau-2,N-2-\theta\}}} \\
&\le \frac{C}{M^{\sigma}} \|v_k\|_{*,2} \frac{\mu^{\frac{N}{q+1}}}{(1+\mu|x-x_i|)^{\frac{N}{q+1}+\tau}}
\end{align*}
for any $\theta \in (0,1)$ small. So it is true that for $x \in B_{R_0}(0) \setminus \cup_{i=1}^k B_{M\mu^{-1}}(x_i)$,
\begin{equation}\label{14-1-2}
\int_{\R^N} \frac{\(V^{p-1}|v_k|\)(y)}{|y-x|^{N-2}} dy
\le \frac{C}{M^{\sigma}} \|v_k\|_{*,2} \sum_{j=1}^{k}\frac{ \mu^{\frac{N}{q+1}} }{ (1+\mu |x-x_j| )^{ \frac{N}{q+1}+\tau }}.
\end{equation}

Second, by applying Lemma~\ref{lemma:U} and arguing as in the proof of Lemma~\ref{l21-23-2}, we can show that for $x \in B_{R_0}(0) \setminus \cup_{i=1}^k B_{M\mu^{-1}}(x_i)$,
\begin{equation}\label{15-1-2}
\int_{\R^N} \frac{\(U^{q-1} |u_k|\)(y)}{|y-x|^{N-4}} dy
\le \frac{C}{M^{\sigma}} \|u_k\|_{*,1} \sum_{j=1}^k \frac{ \mu^{\frac{N}{q+1}} }{ (1+\mu |x-x_j| )^{ \frac{N}{q+1}+\tau }}.
\end{equation}

Lastly, we assume that $q > 2$, which can happen only if $N \le 8$. It holds that for $x \in B_{R_0}(0) \setminus \cup_{i=1}^k B_{M\mu^{-1}}(x_i)$,
\begin{equation}\label{16-1-21}
\begin{aligned}
&\quad \int_{\R^N} \frac{\(U_{0,\mu_0}^{q-2}U |u_k|\)(y)}{|y-x|^{N-4}} dy \\
&\le C \|u_k\|_{*,1} \int_{\R^N \setminus B_{2R_0}(0)} \frac1{|y-x|^{N-4}} \frac1{|y|^{(N-2)(q-2)}} \frac{(k\mu^{\frac{N}{q+1}})^2}{(\mu|y|)^{N-2+\frac{N}{q+1}+\tau}} dy \\
&\ + \frac{C}{M^{\sigma}} \|u_k\|_{*,1} \int_{B_{2R_0}(0) \setminus \cup_{j=1}^k B_{M\mu^{-1}}(x_j)} \frac1{|y-x|^{N-4}} \left[\sum_{j=1}^{k}\frac{\mu^{\frac{N}{q+1}}}{(1+\mu
|x-x_j|)^{\frac{N}{q+1}+\tau}}\right]^2 dy \\
&\ + C \|u_k\|_{*,1} \sum_{j=1}^k \int_{B_{M\mu^{-1}}(x_j)} \frac1{|y-x|^{N-4}} \frac{\mu^{\frac{2N}{q+1}}}{(1+\mu
|x-x_j|)^{N-2+\frac{N}{q+1}+\tau}} dy.
\end{aligned}
\end{equation}
The first term on the right-hand side of \eqref{16-1-21} is bounded by
\[C\|u_k\|_{*,1} \mu^{-(N-3)\tau} \le \frac{C}{\mu^{\sigma}}\|u_k\|_{*,1} \le \frac{C}{\mu^{\sigma}} \|u_k\|_{*,1} \sum_{j=1}^k \frac{ \mu^{\frac{N}{q+1}} }{ (1+\mu |x-x_j| )^{ \frac{N}{q+1}+\tau }}.\]
Since $2 < \frac{N}{q+1}+\tau < \frac{N}{2}$ for $N \ge 6$, the second term is bounded by
\begin{equation}\label{17-1-21}
\begin{aligned}
&\quad \frac{C}{M^{\sigma}} \|u_k\|_{*,1} \int_{B_{2R_0}(0) \setminus \cup_{j=1}^k B_{M\mu^{-1}}(x_j)} \frac1{|y-x|^{N-4}} \cdot k\sum_{j=1}^{k}\frac{\mu^{\frac{2N}{q+1}} }{(1+\mu
|x-x_j|)^{2(\frac{N}{q+1}+\tau)}} dy \\
&\le \frac{C}{M^{\sigma}} \|u_k\|_{*,1} \sum_{j=1}^{k}\frac{k\mu^{\frac{2N}{q+1}-4}}{(1+\mu
|x-x_j|)^{2(\frac{N}{q+1}+\tau-2)}} \\
&\le \frac{C}{M^{\sigma}} \|u_k\|_{*,1} \sum_{j=1}^k \frac{ \mu^{\frac{N}{q+1}} }{ (1+\mu |x-x_j| )^{ \frac{N}{q+1}+\tau }}.
\end{aligned}
\end{equation}
Besides, the third term is bounded by
\[C \|u_k\|_{*,1} \sum_{j=1}^k \frac{\mu^{\frac{2N}{q+1}-4}}{(1+\mu|x-x_j|)^{N-4}} \le \frac{C}{M^{\sigma}} \|u_k\|_{*,1} \sum_{j=1}^k \frac{ \mu^{\frac{N}{q+1}} }{ (1+\mu |x-x_j| )^{ \frac{N}{q+1}+\tau }}.\]
We have shown that for $x \in B_{R_0}(0) \setminus \cup_{i=1}^k B_{M\mu^{-1}}(x_i)$,
\begin{equation}\label{16-1-2}
\int_{\R^N} \frac{\(U_{0,\mu_0}^{q-2}U |u_k|\)(y)}{|y-x|^{N-4}} dy
\le \frac{C}{M^{\sigma}} \|u_k\|_{*,1} \sum_{j=1}^k \frac{ \mu^{\frac{N}{q+1}} }{ (1+\mu |x-x_j| )^{ \frac{N}{q+1}+\tau }}, \quad q > 2.
\end{equation}

Combining \eqref{14-1-2}, \eqref{15-1-2}, and \eqref{16-1-2}, we see that for $x \in B_{R_0}(0) \setminus \cup_{i=1}^k B_{M\mu^{-1}}(x_i)$,
\begin{equation}\label{20-1-6-21}
\left[\sum_{j=1}^{k}\frac{\mu^{\frac{N}{q+1}}}{(1+\mu
|x-x_j|)^{\frac{N}{q+1}+\tau}}\right]^{-1}|u_k(x)|
\le \frac{C}{M^{\sigma}}\|(u_k, v_k)\|_{*} + C\|(f_k, g_k)\|_{**}.
\end{equation}

Similarly, by applying \eqref{1-31-13}, we can prove that for $x \in B_{R_0}(0) \setminus \cup_{i=1}^k B_{M\mu^{-1}}(x_i)$,
\begin{equation}\label{n20-1-6-21}
\left[\sum_{j=1}^{k}\frac{\mu^{\frac{N}{p+1}}}{(1+\mu
|x-x_j|)^{\frac{N}{p+1}+\tau}}\right]^{-1}|v_k(x)| \le \frac{C}{M^{\sigma}}\|(u_k, v_k)\|_{*} + C\|(f_k, g_k)\|_{**}.
\end{equation}

\medskip
We next consider the integrals in \eqref{1-31-14}. The above computations imply that \eqref{14-1-2} remains valid for $x \in \R^N \setminus B_{R_0}(0)$. We also have
\begin{align*}
&\quad \int_{\R^N} \frac{(V_{0,\mu_0}^{p-1}|v_k|)(y)}{|y-x|^{N-2}} dy \\
&\le C \|v_k\|_{*,2} \int_{\R^N} \frac1{|y-x|^{N-2}} \frac{1}{(1+|y|)^{(N-2)(p-1)}} \sum_{j=1}^{k}\frac{\mu^{\frac{N}{p+1}}}{(1+\mu |y-x_j|)^{\frac{N}{p+1}+\tau}} dy \\
&\le C \|v_k\|_{*,2} \frac{1}{|x|^{\min\{(N-2)(p-1)+\frac{N}{p+1}+\tau-2,N-2-\theta\}}} \\
&\le C \|v_k\|_{*,2} \frac1{|x|^\sigma} \cdot \frac{1}{|x|^{\frac{N}{q+1}+\tau}}
\le C \|v_k\|_{*,2} \frac1{|x|^\sigma} \sum_{j=1}^{k} \frac{\mu^{\frac{N}{q+1}}}{(1+\mu |x-x_j|)^{\frac{N}{q+1}+\tau}}
\end{align*}
for some $\sigma>0$. Hence, for $x \in \R^N \setminus B_{R_0}(0)$,
\begin{equation}\label{21-1-6-21}
\begin{aligned}
\left[\sum_{j=1}^{k}\frac{\mu^{\frac{N}{q+1}}}{(1+\mu |x-x_j|)^{\frac{N}{q+1}+\tau}}\right]^{-1}|u_k(x)|
\le C\(\frac1{|x|^\sigma}+\frac1{M^{\sigma}}\)\|(u_k, v_k)\|_{*} + C\|(f_k, g_k)\|_{**}.
\end{aligned}
\end{equation}

Analogously, by employing \eqref{1-31-15}, we can prove that for $x \in \R^N \setminus B_{R_0}(0)$,
\begin{equation}\label{n21-1-6-21}
\left[\sum_{j=1}^{k}\frac{\mu^{\frac{N}{p+1}}}{(1+\mu
|x-x_j|)^{\frac{N}{p+1}+\tau}}\right]^{-1}|v_k(x)|
\le C\(\frac1{|x|^\sigma}+\frac1{M^{\sigma}}\)\|(u_k, v_k)\|_{*} + C\|(f_k, g_k)\|_{**}.
\end{equation}

\medskip
Consequently, we observe from \eqref{20-1-6-21}--\eqref{n20-1-6-21} and \eqref{21-1-6-21}--\eqref{n21-1-6-21} that the maximum of the function in \eqref{eq:functuvk} can only be achieved in $\cup_{i=1}^k \overline{B_{M\mu^{-1}}(x_i)}$ provided $M>1$ large.
Now, a standard argument based on Lemma~\ref{lemma:FKP} and $(u_k, v_k)\in  \mathbf{E}$ tells us that $\|(u_k, v_k)\|_*\to 0$ as $k \to +\infty$.
This is a contradiction.
\end{proof}

For $(f,g)\in \mathbf{Y}$,  we define the operator $\mathbf{P}$ as follows:
\begin{multline}\label{eq:mathbfP}
\mathbf{P} (f,g) := (f, g) \\
+ \(c_0 pV_{0,\mu_0}^{p-1}Z_0 + \sum_{l=1}^2 c_l p\sum_{j=1}^k V_j^{p-1}Z_{j,l},\,
c_0 qU_{0,\mu_0}^{q-1}Y_0 + \sum_{l=1}^2 c_l q\sum_{j=1}^k U_j^{q-1}Y_{j,l}\),
\end{multline}
where $c_0, c_1, c_2 \in \R$ are chosen to be $\mathbf{P} (f,g) \in \mathbf{F}$.
Following the proof of Lemma~\ref{l1-18-2} and employing Lemma~\ref{l10-18-4}, $p(N-2) \ge \frac{pN}{p+1}+\tau$ and $q(N-2) \ge \frac{qN}{q+1}+\tau$, we can check that $\| \mathbf{P} (f,g)\|_{**}\le C \|(f,g)\|_{**}$.

By Lemma~\ref{l20-23-2}, $(-\Delta)^{-1} $ is a bounded linear operator from $\mathbf Y$ to $\mathbf X$.
Also, if $(f, g)\in \mathbf F$, then $(-\Delta)^{-1}(f, g) := ((-\Delta)^{-1}f, (-\Delta)^{-1}g) \in \mathbf E$.
On the other hand, if $(u, v)\in \mathbf X$, then $(p |V_*|^{p-1} v, \,q |U_*|^{q-1} u) \in \mathbf Y$, for Lemmas~\ref{l10-18-2} and \ref{l1-23-4} imply
\[
\left\|\(p |V_*|^{p-1} v, \,q |U_*|^{q-1} u\)\right\|_{**} \le C \|(u, v)\|_*.
\]
As a result, if we define the operator
\begin{equation}\label{10-24-2}
T(u, v) := (-\Delta)^{-1} \mathbf{P}\(p |V_*|^{p-1} v, \,q |U_*|^{q-1} u\),
\quad (u, v)\in \mathbf E,
\end{equation}
then $T$ is a bounded linear operator from $\mathbf E$ to itself.

\begin{proposition}\label{proposition3-5}
Let $N \ge 5$ and $I$ be the identity operator on $\mathbf E$. Then $I-T$ is a bijective bounded linear operator from $\mathbf E$ to itself.
Furthermore, there exists a constant $C>0$ independent of $k \in \N$ large, such that if $(u, v)\in \mathbf E$ satisfies
\begin{equation}\label{30-24-2}
(u, v)= (I-T)^{-1} (-\Delta)^{-1} (f,g)
\end{equation}
for some $(f, g)\in \mathbf Y$ with $(-\Delta)^{-1} (f,g) \in \mathbf E$, then
\[\|(u, v)\|_*\le C \|(f, g)\|_{**}.\]
\end{proposition}
\begin{proof}
Because $|U_*(y)|+|V_*(y)|\to 0$ as $|y|\to +\infty$, one can prove that $T$ is compact.
Also, Proposition~\ref{l20-2-4} implies that $I-T$ is injective on $\mathbf E$. Hence the Fredholm alternative gives that $I-T$ is bijective on $\mathbf E$.

If \eqref{30-24-2} holds, then
\[-\Delta(u,v)-\mathbf P \(p |V_*|^{p-1} v,\, q|U_*|^{q-1} u\) = (f, g).\]
By Proposition~\ref{l20-2-4}, we have that $\|(u, v)\|_*\le C \|(f, g)\|_{**}$.
\end{proof}

\begin{remark}\label{Re4-25-3}
We assert that Proposition \ref{l20-2-4} holds for $N = 5$ and all $p \in [\frac{35}{18},\frac73)$.

To deduce Lemma~\ref{l1-18-2}, we must deal with $J_3$ for $p \in (2,\frac73]$. In view of the inequality $||a-b|^{p-1}- a^{p-1}|\le C(a^{p-2}b+b^{p-1})$ for $a,\, b > 0$ and \eqref{10-18-2}, it suffices to verify
\[\int_{\R^N} (|Z_0|V_{0,\mu_0}^{p-2}V)(y)\sum_{j=1}^{k}\frac{\mu^{\frac{N}{p+1}}}{(1+\mu
|y-x_j|)^{\frac{N}{p+1}+\tau}} dy \le \frac{C}{\mu^{\sigma}},\]
which can be done as in the derivation of \eqref{10-18-2}.
To estimate $J_4$, we employ Lemma~\ref{l1-23-4} as before, which needs the condition $p \in [\frac{35}{18},\frac73]$.

Lemmas~\ref{l10-23-2} and \ref{l20-23-2} continue to hold without any modification.
On the other hand, Lemma~\ref{l21-23-2} has to be changed as follows:
Note that $\frac{5}{q+1}+2+\tau \le 4 \le \frac{5}{p+1}+2+\tau < 5$, with the equalities holding if and only if $p = q = \frac73$.
Hence, if $p \in (\frac53,\frac73)$, then \eqref{eq:g4} remains true, whereas \eqref{eq:f4} must be replaced by
\begin{equation}\label{eq:f45}
\int_{B_{2R_0}(0)} \frac{|f(y)|}{|y-x|^{5-4}} dy + \int_{\R^N \setminus B_{2R_0}(0)} \frac{|f(y)|}{|y-x|^{5-2}} dy \le C \|g\|_{**,2}\sum_{j=1}^{k}\frac{\mu^{\frac{5}{p+1}}}{(1+\mu |x-x_j|)^{\frac{5}{p+1}+\tau}}.
\end{equation}
To prove this, we proceed as in the proof of Lemma~\ref{l21-23-2}. Then, we can see that the left-hand side of \eqref{eq:f45} is bounded by $C\mu^{\frac{5}{q+1}-2}k\mu^{4-(\frac{5}{q+1}+2+\tau)}+C = C$, which is again bounded by its right-hand side for $|x| \le R_0$.
If $p = \frac73$, then the left-hand side of \eqref{eq:f45} is bounded by $C\log\mu+C$, which cannot be bounded by its right-hand side.

Finally, to establish Proposition \ref{l20-2-4}, we must pay attention to the following parts:

\medskip \noindent - If $p \in (2,\frac73)$, we need to include an additional term $\int_{\R^N} \frac1{|y-x|^{N-2}} \(V_{0,\mu_0}^{p-2}V|v_k|\)(y) dy$ on the rightmost side of \eqref{1-31-1}.
It can be treated as in \eqref{16-1-2}, but the corresponding part to \eqref{17-1-21} must be substituted with
\begin{align*}
&\quad \frac{C}{M^{\sigma}} \|v_k\|_{*,2} \int_{B_{2R_0}(0) \setminus \cup_{j=1}^k B_{M\mu^{-1}}(x_j)} \frac1{|y-x|^{5-2}} \left[\sum_{j=1}^{k}\frac{\mu^{\frac{5}{p+1}}}{(1+\mu
|x-x_j|)^{\frac{5}{p+1}+\tau}}\right]^2 dy \\
&\le \frac{C}{M^{\sigma}} \|v_k\|_{*,2} \int_{B_{2R_0}(0) \setminus \cup_{j=1}^k B_{M\mu^{-1}}(x_j)} \frac1{|y-x|^{5-2}} \sum_{j=1}^{k}\frac{\mu^{\frac{2 \cdot 5}{p+1}} }{(1+\mu
|x-x_j|)^{\frac{2 \cdot 5}{p+1}+\tau}} dy \\
&\le \frac{C}{M^{\sigma}} \|v_k\|_{*,2} \sum_{j=1}^{k}\frac{\mu^{\frac{2 \cdot 5}{p+1}-2}}{(1+\mu
|x-x_j|)^{\frac{2 \cdot 5}{p+1}+\tau-2}}
\le \frac{C}{M^{\sigma}} \|v_k\|_{*,2} \sum_{j=1}^k \frac{ \mu^{\frac{5}{q+1}} }{ (1+\mu |x-x_j| )^{ \frac{5}{q+1}+\tau }},
\end{align*}
where we employed H\"older's inequality and $\sum_{j=1}^k (1+\mu |y-x_j|)^{-\tau} \le C$ for $y \in \R^N$ for the first inequality, $2 < \frac{2 \cdot 5}{p+1}+\tau < 5$ for the second inequality, and $p>2$ for the third inequality.

\medskip \noindent - If $p \in (\frac53,\frac73)$, then $2(\frac{5}{q+1}+\tau)<4$. Thus, the leftmost side in \eqref{17-1-21} is bounded by
\[\frac{C}{M^{\sigma}} \|u_k\|_{*,1} \mu^{-\tau}
\le \frac{C}{M^{\sigma}} \|u_k\|_{*,1} \\
\le \frac{C}{M^{\sigma}} \|u_k\|_{*,1} \sum_{j=1}^k \frac{ \mu^{\frac{5}{q+1}} }{ (1+\mu |x-x_j| )^{ \frac{5}{q+1}+\tau }}.\]

\medskip \noindent The remainder of the proof follow similarly and are omitted.
\end{remark}

\section{Ljapunov-Schmidt reduction for $p \in (\frac N{N-2},\frac{N+2}{N-2}]$}\label{sec:red}
A direct computation shows that the function $(U_*+\omega_1, V_*+\omega_2)$ solves \eqref{2} if and only if
\[L_k(\omega_1,\omega_2) =l +N(\omega_1,\omega_2),\]
where $L_k$ is the linear operator defined in \eqref{3-2},
\begin{equation}\label{3-3}
l := (l_1,l_2) = \( \Delta U_*+ |V_*|^{p-1}V_*,\,  \Delta V_*+ |U_*|^{q-1}U_* \) \in \mathbf Y,
\end{equation}
and
\begin{equation}\label{3-4}
N(\omega_1,\omega_2) := (N_1(\omega_2),N_2(\omega_1))
\end{equation}
with
\begin{equation}\label{3-5}
\begin{aligned}
\begin{cases}
\displaystyle N_1(\omega_2) := |V_*+\omega_2|^{p-1}( V_*+\omega_2) -|V_*|^{p-1} V_*-p|V_*|^{p-1}\omega_2, \\
\displaystyle N_2(\omega_1) := |U_*+\omega_1|^{q-1}( U_*+\omega_1) -|U_*|^{q-1} U_*-q|U_*|^{q-1}\omega_1.
\end{cases}
\end{aligned}
\end{equation}

We estimate the quantities $\| l\|_{**} $ and $\|N(\omega_1,\omega_2)\|_{**}$. We know
\[
l_1=\Delta U_*+ |V_*|^{p-1}V_*=-V_{0,\mu_0}^p+V^p+|V_*|^{p-1}V_*
\]
and
\[
l_2=\Delta V_*+ |U_*|^{q-1}U_* = -U_{0,\mu_0}^q +\sum_{j=1}^k U_j^q+ |U_*|^{q-1}U_*.
\]

\begin{lemma}\label{lem2.5}
Assume that $N \ge 6$. Then there is a constant $\sigma>0$ such that
\[\|(l_{1},l_{2})\|_{**} \le C \mu^{-\frac{N}{2(q+1)}-\sigma}.\]
\end{lemma}
\begin{proof}
Let $\Omega_1 \subset \R^N$ be the set defined in \eqref{Omegaj}. By symmetry, it is sufficient to estimate $l_1$ and $l_2$ in $\Omega_1$. Let $S = B_{\frac{\pi}{2}r_0^{-1}k^{-1}}(x_1) \subset \Omega_1$.

\medskip \noindent \textbf{Estimation of $l_1$ in $S$.}
It holds that $\sum_{j=2}^k V_j \le C\mu^{\frac{N}{p+1}-\frac{N}{q+1}} \le CV_1$ in $S$. Recalling
that $k\simeq \mu^{\tau}$ with $\tau = \frac{N}{(p+1)(N-2)} \in (0,1)$, we also find that $V_1 \ge c$ in $S$. Therefore,
\[|l_1| \le C + \left||V_*|^{p-1}V_*+V^{p}\right| \le CV_1^{p-1} \le C\frac{\mu^{\frac{(p-1)N}{p+1}}}{(1+\mu|y-x_1|)^{(p-1)(N-2)}},\quad y\in S.\]
We shall show that there exists $t > \frac{N}{2(q+1)}$ such that
\[\frac{\mu^{\frac{(p-1)N}{p+1}}}{(1+\mu|y-x_1|)^{(p-1)(N-2)}} \le \frac{C}{\mu^t}
\frac{\mu^{\frac{N}{q+1}+2}}{(1+\mu|y-x_1|)^{\frac{N}{q+1}+2+\tau}}, \quad y \in S,\]
which is equivalent to
\begin{equation}\label{eq:error1}
(1+\mu|y-x_1|)^{\frac{N}{q+1}+2+\tau-(p-1)(N-2)} \le C \mu^{\frac{N}{q+1}+2-\frac{(p-1)N}{p+1}-t}, \quad y \in S.
\end{equation}

If $\frac{N}{q+1}+2+\tau-(p-1)(N-2)\le 0$, then the left-hand side of \eqref{eq:error1} is bounded. So we can take
\[
t=\frac{N}{q+1}+2-\frac{(p-1)N}{p+1}=\frac N{p+1}>\frac N{2(q+1)}.
\]

If $\frac{N}{q+1}+2+\tau-(p-1)(N-2)> 0$, then \eqref{eq:error1} is valid provided
\[\mu^{(1-\tau)\left[\frac{pN}{p+1}+\tau-(p-1)(N-2)\right]} \le C \mu^{\frac{N}{p+1}-t}.\]
So we can take
\[t = \frac{N}{p+1} - (1-\tau)\left[\frac{pN}{p+1}+\tau-(p-1)(N-2)\right].\]
Let  $\tilde{p} := p(N-2) \in (N,N+2]$. Then $t > \frac{N}{2(q+1)}$ is equivalent to
\[\begin{medsize}
\displaystyle \frac{N}{\frac{ \tilde{p} }{N-2}+1}-\(1-\frac{N}{\tilde p +N-2}\) \left[\frac{\frac{ \tilde pN }{N-2}}{\frac{ \tilde{p} }{N-2}+1}+\frac{N}{\tilde p +N-2}-\tilde p +(N-2)\right]
-\frac12 \left[ N-2 - \frac{N}{\frac{ \tilde{p} }{N-2}+1} \right]>0.
\end{medsize}\]
Multiplying the both sides of the above inequality by $2(\tilde p +N-2)^2$, we
are led to
\begin{equation}\label{eq:error2}
2\tilde{p}^3 - (3N+2)\tilde{p}^2 - (N^2-12N+16)\tilde{p} + 2(N^3-N^2-6N+12) > 0.
\end{equation}
Write $\tilde p = N +h$ with $h\in (0, 2]$. Then \eqref{eq:error2} is reduced to
\[
(8- h   )N^2+  (3h^2+8h-28  )N+ 2(h-2)^2(h+3) >0,
\]
which is true for all $N \ge 5$ and $h\in (0, 2]$. 

\medskip \noindent \textbf{Estimation of $l_1$ in $\Omega_1\setminus S$.}
For $p\in (1, 2]$, we have the following equality
\[
||a-b|^{p-1} (a-b)- a^p + b^p|\le Cb^{p-1} a, \quad a,\, b > 0.
\]
Since $\frac{N+2}{N-2} \le 2$ for $N \ge 6$, it holds
\begin{equation}\label{30-25-3}
|l_1|\le CV_{0,\mu_0}^{p-1} V =CV_{0,\mu_0}^{p-1}\sum_{j=1}^k V_j \quad \text{in } \Omega_1\setminus S
\end{equation}
for all $N \ge 6$ and $p \in (\frac{N}{N-2},\frac{N+2}{N-2}]$. We determine $t>0$ such that
\[
V_{0,\mu_0}^{p-1}(y) \frac{\mu^{\frac{N}{p+1}}}{(1+\mu|y-x_j|)^{N-2}}\le \frac{C}{\mu^t}
\frac{\mu^{\frac{N}{q+1}+2}}{(1+\mu|y-x_j|)^{\frac{N}{q+1}+2+\tau}}, \quad y \in \Omega_1\setminus S,
\]
which is equivalent to
\begin{equation}\label{1-13-2}
V_{0,\mu_0}^{p-1}(y) (1+\mu|y-x_j|)^{-\frac{N}{p+1}+2+\tau}\le C
\mu^{N-\frac{2N}{p+1}-t}, \quad y \in \Omega_1\setminus S.
\end{equation}

We first study the case
$
\frac{N}{p+1}-2-\tau\ge 0.
$
In $\Omega_1\setminus S$, it holds
\[\mu|y-x_j|\ge c\mu k^{-1}\simeq \mu^{1-\tau}\]
for some $c > 0$. Thus, \eqref{1-13-2} will follow if
\[
\mu^{(1-\tau)(-\frac{N}{p+1}+2+\tau)}\le
C\mu^{N-\frac{2N}{p+1}-t}.
\]
So we can choose
\[t =N-\frac{2N}{p+1}+(1-\tau)\(\frac{N}{p+1}-2-\tau\).\]
We will show $t>\frac{N}{2(q+1)}$ for all $N \ge 6$ and $p \in (\frac{N}{N-2},\frac{N+2}{N-2}]$. We define
\[\begin{medsize}
\displaystyle \tsg_3(p) :=N-\frac{2N}{p+1}+\left[1-\frac{N}{(p+1)(N-2)}\right] \left[\frac{N}{p+1}-2-\frac{N}{(p+1)(N-2)}\right]-\frac{N}{2(q+1)}.
\end{medsize}\]
We have
\begin{align*}
\begin{medsize}
\displaystyle \tsg_3'(p)
\end{medsize}
&\begin{medsize}
\displaystyle = \frac{2N}{(p+1)^2}-\frac{N}{2(p+1)^2} +\frac{N}{(p+1)^2(N-2)}\left[\frac{N}{p+1}-2-\frac{N}{(p+1)(N-2)}\right]
\end{medsize} \\
&\begin{medsize}
\displaystyle \ -\left[1-\frac{N}{(p+1)(N-2)}\right]\frac{N(N-3)}{(p+1)^2(N-2)} \end{medsize} \\
&\begin{medsize}
\displaystyle = \frac{N [(N-2)(N-4)p+5N^2-18N+8]}{2(p+1)^3(N-2)^2} > 0.
\end{medsize}
\end{align*}
Thus, we only need to prove that $\tsg_3(\frac N{N-2})>0$. We rewrite $\tsg_3(p)$ as
\[\begin{medsize}
\displaystyle \tsg_3(p)=\frac{(p-1)N}{p+1} -\left[1-\frac{N}{(p+1)(N-2)}\right]\left[2+\frac{N}{(p+1)(N-2)}\right] +\frac{N}{q+1}\left[ \frac{N}{(p+1)(N-2)}-
\frac12\right].
\end{medsize}\]
It follows that
\begin{align*}
\begin{medsize}
\displaystyle \tsg_3(\tfrac N{N-2})
\end{medsize}
&\begin{medsize}
\displaystyle = \frac{N}{N-1} -\frac{N-2}{2N-2}\(2+\frac{N}{2N-2}\) +\frac{N}{q+1} \cdot \frac{1}{2N-2}
\end{medsize} \\
&\begin{medsize}
\displaystyle = \frac{1}{2(N-1)} \left[4+\frac{N}{q+1}-\frac{N(N-2)}{2(N-1)}\right]
\end{medsize} \\
&\begin{medsize}
\displaystyle = \frac{1}{2(N-1)} \left[N+2-\frac{N(N-2)}{N-1}\right] = \frac{3N-2}{2(N-1)^2} > 0.
\end{medsize}
\end{align*}

Next, we consider the case $\frac{N}{p+1}-2-\tau< 0.$
Inequality \eqref{1-13-2} will follow if
\begin{equation}\label{1-18-3}
\frac{1}{ (1+|y|)^{(p-1)(N-2)}  } (\mu|y-x_j|)^{-\frac{N}{p+1}+2+\tau}\le
C\mu^{N-\frac{2N}{p+1}-t}, \quad y \in \Omega_1\setminus S.
\end{equation}
It is easy to see that \eqref{1-18-3} is equivalent to
\begin{equation}\label{2-18-3}
\frac{|y-x_j|^{-\frac{N}{p+1}+2+\tau}}{ (1+|y|)^{(p-1)(N-2)}  } \le
C\mu^{\frac{N}{q+1}-t-\tau}, \quad y \in \Omega_1\setminus S.
\end{equation}
Because
\[
(p-1)(N-2) +\frac N{p+1} -2 -\tau \ge \frac N{p+1}  -\tau>0,
\]
the left-hand side of \eqref{2-18-3} is bounded and we can take $t=\frac{N}{q+1}-\tau$, which is greater than $\frac {N}{2(q+1)}$ for all $N \ge 6$ and $p\in (\frac N{N-2}, \frac{N+2}{N-2}]$.

\medskip \noindent \textbf{Estimation of $l_2$ in $S$.}
It holds that $U_1\ge c>0$ and $\sum_{j=2}^k  U_j + \varphi\le C$ in $S$. So we have
\[||U_*|^{q-1}U_*+ U_{1}^q| \le C U_1^{q-1} \( U_{0,\mu_0} +\sum_{j=2}^k  U_j\) \le C U_1^{q-1}\]
and
\[
\sum_{j=2}^k  U^q_j\le C U_1^{q-1} \sum_{j=2}^k  U_j \le C U_1^{q-1}.
\]
Hence,
\[
|l_2|\le CU_1^{q-1} \quad \text{in } S.
\]
Now we determine $t>0$ such that
\[
U_1^{q-1}(y) \le \frac C{\mu^t}\frac{\mu^{\frac{N}{p+1}+2}}{(1+\mu
|y-x_{1}|)^{\frac{N}{p+1}+2+\tau}}, \quad y \in S.
\]
This is equivalent to
\begin{equation}\label{10-9-3}
(1+\mu|y-x_{1}|)^{\frac{N}{p+1}+2+\tau-(q-1)(N-2)}\le
C\mu^{\frac{N}{p+1}+2-\frac{(q-1)N}{q+1}-t}, \quad y \in S.
\end{equation}

If $\frac{N}{p+1}+2+\tau-(q-1)(N-2)\le 0$, then the left-hand side of \eqref{10-9-3} is bounded. In this case, we can take
\[
t=\frac{N}{p+1}+2-\frac{(q-1)N}{q+1}=\frac{N}{q+1}>\frac{N}{2(q+1)}.
\]

We deal with the case $\frac{N}{p+1}+2+\tau-(q-1)(N-2)> 0$. In $S$, we have $\mu
|y-x_{1}|\le C\mu^{1-\tau}$. So, \eqref{10-9-3} holds if
\[
\mu
^{(1-\tau) [\frac{N}{p+1}+2+\tau-(q-1)(N-2)]}\le
C\mu^{\frac{N}{p+1}+2-\frac{(q-1)N}{q+1}-t},
\]
and we can choose
\[
t =\frac{N}{q+1}-(1-\tau) \left[\frac{qN}{q+1}+\tau-(q-1)(N-2)\right].
\]
Let us prove that  $t>\frac{N}{2(q+1)}$. Since $q\ge \frac{N+2}{N-2}$, it holds that
$\tilde q := q(N-2)\ge N+2$. Hence it is enough to verify
\begin{equation}\label{eq:error3}
\frac12\frac{N(N-2)}{\tilde q+N-2}-\frac{N}{\tilde q +N-2}\(\frac{\tilde q N}{
\tilde q+N-2}+1-\frac{N}{\tilde q +N-2}-\tilde q+N-2\)>0,
\end{equation}
which is equivalent to
\[
2\tilde q^2-(N+4)\tilde q-N(N-4)>0.
\]
Let $\tilde q = N+ h$ with $h\ge 2$. We need to check
\[
2(N+h)^2-(N+4)(N+h)-N(N-4) = h(3N+2h-4) >0,
\]
which is clearly true for $h \ge 2$. Thus, \eqref{eq:error3} holds.

\medskip \noindent \textbf{Estimation of $l_2$ in $\Omega_1\setminus S$.}
We present this part in four steps.

\medskip \noindent \textbf{Step~1.} We first show that there exists $t > \frac{N}{2(q+1)}$ such that
\[
U_{0,\mu_0}^{q-1}\sum_{j=1}^k U_j
\le
\frac{C}{\mu^t}\sum_{j=1}^k
\frac{\mu^{\frac{N}{p+1}+2}}{(1+\mu|y-x_j|)^{\frac{N}{p+1}+2+\tau}} \quad \text{in } \Omega_1\setminus S.
\]
It suffices to prove that for each $j = 1,\ldots,k$,
\[\frac{1}{(1+|y|)^{(q-1)(N-2)}} \frac{\mu^{\frac{N}{q+1}}}{(\mu|y-x_j|)^{N-2}} \le \frac{C}{\mu^t} \frac{\mu^{\frac{N}{p+1}+2}}{(\mu|y-x_j|)^{\frac{N}{p+1}+2+\tau}}, \quad y \in \Omega_1\setminus S,\]
which can be rewritten as
\begin{equation}\label{eq:l2o11}
\frac{1}{(1+|y|)^{(q-1)(N-2)}} (\mu|y-x_j|)^{-\frac{N}{q+1}+2+\tau} \le C\mu^{N-\frac{2N}{q+1}-t}, \quad y \in \Omega_1\setminus S.
\end{equation}

If $\frac{N}{q+1}-2-\tau \ge 0$, then \eqref{eq:l2o11} will be true provided
\[\mu^{(1-\tau)(-\frac{N}{q+1}+2+\tau)} \le C\mu^{N-\frac{2N}{q+1}-t}.\]
So we can choose
\[t =N-\frac{2N}{q+1}+(1-\tau)\(\frac{N}{q+1}-2-\tau\).\]
Moreover, it holds that
\begin{align*}
t > \frac{N}{2(q+1)} \; &\Leftrightarrow \; \frac{(q-1)N}{q+1}+\frac{N}{(q+1)(N-2)}\(\frac{N(N-1)}{(q+1)(N-2)}-3\) > \frac{N}{2(q+1)} \\
&\Leftrightarrow \; \tsg_4(q) := q(N-2)+\frac{N(N-1)}{(q+1)(N-2)} > \frac{3N}{2}.
\end{align*}
The last inequality is true for $q \ge \frac{N+2}{N-2}$, because the function $\tsg_4(q)$ is increasing and so
\begin{equation}\label{eq:l2o12}
\tsg_4(q) \ge \tsg_4(\tfrac{N+2}{N-2}) = \frac{3(N+1)}{2} > \frac{3N}{2}.
\end{equation}

If $\frac{N}{q+1}-2-\tau < 0$, then
\begin{align*}
(q-1)(N-2)+\frac{N}{q+1}-2-\tau > 0 \; &\Leftrightarrow \; (q-1)(N-2)+\frac{N(N-1)}{(q+1)(N-2)}-3 > 0 \\
&\Leftrightarrow \; \tsg_4(q) > N+1,
\end{align*}
and the last inequality is clearly true because of \eqref{eq:l2o12}. Thus, \eqref{eq:l2o11} is valid with the choice
\[t = N-\frac{2N}{q+1}+\(\frac{N}{q+1}-2-\tau\) = \frac{N}{p+1}-\tau,\]
and $t \ge \frac{N}{q+1}-\tau > \frac{N}{2(q+1)}$ for all $N \ge 6$ and $p\in (\frac N{N-2}, \frac{N+2}{N-2}]$.

\medskip \noindent \textbf{Step~2.} We claim that $U\le C U_{0,\mu_0}$ in $(\Omega_1\setminus S)\cap \{|y|\le R\}$ for any large $R > 0$.

In fact, if $p\in (\frac N{N-3}, \frac{N+2}{N-2}]$, then it follows from Lemma~\ref{lemma:U} that $U\le C\displaystyle\sum_{j=1}^k U_j$ in $\R^N$, since the inequality
\[\frac{1}{\mu^{\frac{pN}{q+1}}}  \frac{k^{p(N-2)-2}}{(1+k|y-x_i|)^{N-2}} \le \frac{C\mu^{\frac{N}{q+1}}}{(1+\mu|y-x_i|)^{N-2}}, \quad y \in \R^N\]
is true for all $p > \frac{N}{N-2}$. Furthermore,
\[
\sum_{j=1}^k U_j\le \frac{C}{\mu^{\frac N{p+1}}}\sum_{j=1}^k \frac{1}{|y-x_j|^{N-2}}\le \frac{Ck^{N-2}}{\mu^{\frac N{p+1}}}\le C
\le C U_{0,\mu_0}.
\]

If $p\in (\frac{N}{N-2}, \frac N{N-3}],$ then $p(N-3)-2 > 1$ for $N\ge 5$ and so we have
\[
\frac{1}{\mu^{\frac{pN}{q+1}}} \sum_{i=1}^k \frac{k^{p(N-2)-2}}{(1+k|y-x_i|)^{p(N-3-\theta)-2}}\le  \frac{Ck^{p(N-2)-2}}{\mu^{\frac{pN}{q+1}}}
\le \frac{C}{\mu^{\frac{pN}{q+1}-p(N-2)\tau +2\tau}} \le
C \le CU_{0,\mu_0}.
\]
Here, we used
\begin{align*}
&\quad \frac{pN}{q+1}-p(N-2)\tau +2\tau=\frac{pN}{q+1}-\frac{pN}{p+1} +2\tau\\
&= \frac{(p-1)N}{q+1}-2(1-\tau)=\frac{N}{q+1}\( p-1-\frac2{N-2}\)>0
\end{align*}
to deduce the third inequality. By Lemma~\ref{lemma:U}, it holds that $U\le C U_{0,\mu_0}$.

\medskip \noindent \textbf{Step~3.} We estimate $l_2$ in $(\Omega_1\setminus S)\cap \{|y|\le R\}$.

By Step~2, we see that
\[
\left||U_*|^{q-1}U_*- U_{0,\mu_0}^q\right|\le C U_{0,\mu_0}^{q-1}U \quad \text{in}\; (\Omega_1\setminus S)\cap \{|y|\le R\}.
\]
As a result, in view of
\[
U_j\le C U_{0,\mu_0} \quad \text{in } \Omega_1\setminus S,
\]
we see
\begin{equation}\label{1-23-4}
|l_2| \le \left||U_*|^{q-1}U_*- U_{0,\mu_0}^q\right| + \sum_{j=1}^k U^q_j \le C U_{0,\mu_0}^{q-1}U +\sum_{j=1}^k U^q_j \le C U_{0,\mu_0}^{q-1}U  \quad \text{in } \Omega_1\setminus S.
\end{equation}
So the result follows from Step~1.

\medskip \noindent \textbf{Step~4.} Let $R_0 > 1$ be a large number satisfying $|x_i| \le \frac {R_0}2$ for all $i = 1,\ldots,k$. We estimate $l_2$ in $(\Omega_1\setminus S)\cap \{|y|> R_0\}$. By \eqref{1-23-4}, we have
\[|l_2| \le CU_{0,\mu_0}^{q-1} U + C U^q.
\]
So from Step~1, we only need to estimate $U^q$.

From Lemma~\ref{lemma:U}, we find  that for $|y|>R_0$,
\[
U(y)\le \frac{C k}{\mu^{\frac N{p+1}}|y|^{N-2}} + \frac{C k^{1+p(1+\theta)}}{\mu^{\frac  {pN}{q+1}}|y|^{p(N-3-\theta)-2}}.
\]

We now prove that there is $t>\frac{N}{2(q+1)}$, such that
\begin{equation}\label{1-9-6}
\frac{C k^q}{\mu^{\frac {qN}{p+1}}|y|^{q(N-2)}}\le \frac{C}{\mu^t|y|^{\frac{N}{p+1}+2+\tau}} \simeq \frac{C}{\mu^t}\sum_{j=1}^k
\frac{\mu^{\frac{N}{p+1}+2}}{(1+\mu|y-x_j|)^{\frac{N}{p+1}+2+\tau}}, \quad |y|>R_0,
\end{equation}
and
\begin{equation}\label{2-9-6}
\frac{C k^{q[1+p(1+\theta)]}}{\mu^{\frac  {qpN}{q+1}}|y|^{q[p(N-3-\theta)-2]}}\le \frac{C}{\mu^t}\sum_{j=1}^k
\frac{\mu^{\frac{N}{p+1}+2}}{(1+\mu|y-x_j|)^{\frac{N}{p+1}+2+\tau}},\quad |y|>R_0.
\end{equation}

To prove \eqref{1-9-6}, we first observe that
\[
q(N-2)\ge N+2>\frac{N}{p+1}+2+\tau=2+\frac{N(N-1)}{(p+1)(N-2)}.
\]
Therefore, we can take
\[
t=\frac {qN}{p+1}-q\tau = \frac{N-3}{N-2} \cdot \frac {qN}{p+1} >\frac{N}{2(q+1)},
\]
which is true since for $N\ge 5$,
\[
\frac{N-3}{N-2} \cdot \frac {qN}{p+1} >\frac{qN}{2(p+1)} >\frac{N}{2(q+1)}.
\]

Now, we prove \eqref{2-9-6}. We first note that
\begin{equation}\label{2-9-6a}
q[p(N-3)-2] \ge \frac{[p(N-3)-2] (N+2)}  {N-2} >\frac{N}{p+1}+2+\tau=2+\frac{N(N-1)}{(p+1)(N-2)}
\end{equation}
for $N \ge 6$. To achieve the second inequality in \eqref{2-9-6a}, we make use of the fact that the function
\[
\tsg_5(p) := \frac{[p(N-3)-2] (N+2)}  {N-2}-\frac{N(N-1)}{(p+1)(N-2)}-2
\]
is increasing. Then, for $p\in (\frac N{N-2},\frac{N+2}{N-2}]$,
\[
\tsg_5(p)>\tsg_5(\tfrac N{N-2})=\frac{N+2}{N-2} \left[\frac{N(N-3)}{N-2}-2 \right]-\frac N2-2 =
\frac{N^2(N-6)}{2(N-2)^2} \ge 0
\]
for $N\ge 6$. So we can take
\[
t=\frac{pqN}{q+1} -\tau q[ 1+p  (1+\theta)].
\]
Let us prove that $t>\frac{N}{2(q+1)}$ for some small $\theta > 0$. The relation
\[\frac{pqN}{q+1} -\tau q(1+p) = \frac{pqN}{q+1} - \frac{qN}{N-2} > \frac{N}{N-2} \cdot \frac{qN}{q+1} - \frac{qN}{N-2} > \frac{N}{2(q+1)}\]
holds provided
\[\tsg_6(q) := 2q^2 - 2(N-1)q +N-2 < 0, \quad q \in [\tfrac{N+2}{N-2},\tfrac{N^2+2N-4}{(N-2)^2}].\]
If $N \ge 7$, then $\frac{N^2+2N-4}{(N-2)^2} < \frac{N-1}{2}$, so it follows that
\[\tsg_6(q) \le \tsg_6(\tfrac{N+2}{N-2}) = -\frac{N^3+2N^2-28N+8}{(N-2)^2} < 0.\]
If $N = 6$, then $\tsg_6(q) = 2(q^2-5q+2) < 0$ for $q \in [2,\frac{11}{4}]$ as desired.
\end{proof}

\begin{remark}\label{Re2-25-3}
If $N=5$, Lemma~\ref{lem2.5} remains valid for all $p\in (\frac{16}{9},\frac73]$.
To verify this, it suffices to examine the estimates of $l_1$ and $l_2$ in $\Omega_1\setminus S$.

\medskip \noindent \textbf{Estimation of $l_1$ in $\Omega_1\setminus S$.}
If $p \le 2$, then the argument in the previous proof continues to hold provided $\frac{N}{q+1}-\tau > \frac{N}{2(q+1)}$, which is true for all $p > \frac{16}{9}$.

Assume that $p \in (2,\frac73]$. On estimating $l_1$ in $\Omega_1\setminus S$, if
$V_{0,\mu_0}\ge V$, then \eqref{30-25-3} holds and we can proceed as before.
If $V_{0,\mu_0}\le V$, then the condition $p > 2$ yields
\[
|l_1|\le CV_{0,\mu} V^{p-1}\le CV_{0,\mu} \mu^{ \frac{5(p-1)}{p+1} }
k^{p-2} \sum_{j=1}^k  \frac1{(1+\mu|y-x_j|)^{3(p-1)}     }.
\]
Consider the inequality
\begin{equation}\label{10-25-3}
V_{0,\mu}(y)
 \frac{\mu^{ \frac{5(p-1)}{p+1} + (p-2)\tau}}{(1+\mu|y-x_j|)^{ 3(p-1)}     }\le
 \frac{C}{\mu^t}
\frac{\mu^{\frac{5}{q+1}+2}}{(1+\mu|y-x_j|)^{\frac{5}{q+1}+2+\tau}}, \quad y \in \Omega_1 \setminus S.
\end{equation}
It is easy to see that \eqref{10-25-3} is equivalent to
\begin{equation}\label{101-25-3}
\frac{|y-x_j|^{\frac{5p}{p+1}+\tau-3(p-1)}}{(1+|y|)^{3}}
\le C\mu^{(p-1)(\frac{5}{q+1}-\tau)-t}, \quad y \in \Omega_1 \setminus S,
\end{equation}
and $0 \le \frac{5p}{p+1}+\tau-3(p-1) = \frac{(3p+2)(7-3p)}{3(p+1)} < 3$ for $p \in (2,\frac73]$.
Hence the left-hand side of \eqref{101-25-3} is bounded and we can take
\[t = (p-1)\(\frac{5}{q+1}-\tau\) = (p-1)\left[3-\frac{20}{3(p+1)}\right],\]
which is greater than $\frac{5}{2(q+1)}$ for $p \in (2,\frac73]$.

\medskip \noindent \textbf{Estimation of $l_2$ in $\Omega_1\setminus S$.}
For $p \in (\frac{16}9,\frac73]$ and correspondingly $q \in [\frac73,\frac{19}6)$, we see that $\frac{N}{q+1}-\tau > \frac{N}{2(q+1)}$ and $\tsg_6(q) = 2q^2-8q+3 < 0$.
Also, the inequality $q[p(N-3)-2] >\frac{N}{p+1}+2+\tau$ (cf. \eqref{2-9-6a}) is reduced to $6p^3+12p^2-39p+5 > 0$, which turns out to be true. With these inequalities, we can proceed as before.
\end{remark}

\begin{lemma}\label{lemma4-1}
Assume that $N \ge 6$ so that $p \le 2$. It holds that
\[\|N(\omega_1,\omega_2)\|_{**}\le C  \|(\omega_1,\omega_2)\|_*^{p }.\]
\end{lemma}
\begin{proof}
In light of \eqref{3-5}, we have
$$
|N_1(\omega_2)|\le C|\omega_2|^{p}.
$$
Using Lemma~\ref{l10-18-2}, we find
\begin{align*}
|N_1(\omega_2)(y)| &\le C \|\omega_2\|_{*,2}^p \left[  \sum_{j=1}^k\frac{\mu^{\frac{N}{p+1}}}{(1+\mu|y-x_j|)
^{\frac{N}{p+1}+\tau}}\right]^{p} \\
&\le C \|\omega_2\|_{*,2}^p \sum_{j=1}^k \frac{ \mu^{\frac{pN}{p+1}}  }{ (1+\mu |y-x_j|)^{  \frac{N}{q+1}+2+\tau } }, \quad y \in \R^N.
\end{align*}
Since $\frac{pN}{p+1} = \frac{N}{q+1}+2$, we arrive at
$$
\|N_1(\omega_2)\|_{**,1} \le C \|\omega_2\|_{*,2}^{  p  }.
$$

In a similar manner, we have $|N_2(\omega_1)|\le C|\omega_1|^{q}$ for $q \le 2$ and $|N_2(\omega_1)|\le C(|\omega_1|^{q} + |U_*|^{q-2}\omega_1^2)$ for $q > 2$.
In addition, $\||\omega_1|^q\|_{**,2} \le C \|\omega_1\|_{*,1}^{ q  }$, and for $q > 2$,
\begin{align*}
&\quad |U_*(y)|^{q-2}\omega_1^2(y) \\
&\le C \|\omega_1\|_{*,1}^2 \left[\frac{1}{(1+|y|)^{(q-2)(N-2)}} + U^{q-2}(y)\right] \left[\sum_{j=1}^k\frac{\mu^{\frac{N}{q+1}}}{(1+\mu|y-x_j|)^{\frac{N}{q+1}+\tau}}\right]^{2} \\
&\le C \|\omega_1\|_{*,1}^2 \left[\frac{1}{(1+|y|)^{q(N-2)}} + \(\sum_{j=1}^k\frac{\mu^{\frac{N}{q+1}}}{(1+\mu|y-x_j|)^{\frac{N}{q+1}+\tau}}\)^{q}\right] \\
&\le C \|\omega_1\|_{*,1}^2 \sum_{j=1}^k \frac{ \mu^{\frac{qN}{q+1}}  }{ (1+\mu |y-x_j|)^{  \frac{N}{p+1}+2+\tau } }, \quad y \in \R^N,
\end{align*}
where we used Young's inequality, Lemma~\ref{l1-23-4} and $N-2 > \frac{N}{q+1}+\tau$ to get the second inequality, and employed Lemma~\ref{l10-18-4} for the third inequality.
Consequently, we obtain that $\|N_2(\omega_1)\|_{**,2} \le C \|\omega_1\|_{*,1}^{\min\{q,2\}} \le C \|\omega_1\|_{*,1}^{p}$.

By \eqref{3-4}, we conclude that
\[
\|N(\omega_1,\omega_2)\|_{**}\le C  \|(\omega_1,\omega_2)\|_*^{p }. \qedhere
\]
\end{proof}

\begin{remark}\label{Re3-25-3}
When $N = 5$, it holds that
\[\|N(\omega_1,\omega_2)\|_{**}\le C  \|(\omega_1,\omega_2)\|_*^{\min\{p,2\}}.\]
If $p \le 2$, the previous argument carries over without requiring any changes.
If $p \in (2,\frac73]$, then $|N_1(\omega_2)|\le C(|\omega_2|^{p} + |V_*|^{p-2}\omega_2^2)$.
By employing the inequality $3 > \frac{5}{p+1}+\tau$ and Lemma~\ref{l10-18-4}, one can derive $\||V_*|^{p-2}\omega_2^2\|_{**,1} \le C\|\omega_2\|_{*,2}^2$.
\end{remark}

Now, we take into account the problem
\begin{equation}\label{4-1}
(\omega_1,\omega_2) - T(\omega_1,\omega_2) = (-\Delta)^{-1}(\mathbf{P} l) +(-\Delta)^{-1}(\mathbf{P} N(\omega_1,\omega_2)),
\end{equation}
where $T$ and $\mathbf{P}$ are the operators defined in \eqref{10-24-2} and \eqref{eq:mathbfP}, respectively.

\begin{proposition}\label{proposition4-3}
Given $(\mu_0,r,\mu) \in \mcp$ (see \eqref{mcp}) and sufficiently large $k \in \N$, equation \eqref{4-1} admits a unique solution $(\omega_1,\omega_2) = (\omega_1[\mu_0,r,\mu],\omega_2[\mu_0,r,\mu]) \in \mathbf E$ such that
$$
\| (\omega_1,\omega_2) \|_{*} \le  C
\mu^{ -\frac {N}{2(q+1)}-\sigma},
$$
where $C > 0$ is a constant and $\sigma > 0$ is a small number depending only on $N$, $p$ and $\mcp$. Moreover, the map $(\mu_0,r,\lambda) \in \mcp \mapsto \|(\omega_1[\mu_0,r,\mu],\omega_2[\mu_0,r,\mu])\|_*$ is $C^1$,  and
\begin{equation}\label{10-12-3}
\begin{cases}
\displaystyle \omega_1\left[t\mu_0, t^{-1}r, t\mu\right](y) = t^{\frac{N}{q+1}} \omega_1[\mu_0,r,\mu](ty),\\
\displaystyle \omega_2\left[t\mu_0, t^{-1}r, t\mu\right](y) = t^{\frac{N}{p+1}} \omega_2[\mu_0,r,\mu](ty)
\end{cases}
\end{equation}
for $y \in \R^N$.
\end{proposition}
\begin{proof}
Proposition~\ref{proposition3-5} shows that $I-T$ is invertible. Hence, \eqref{4-1} is equivalent to
\begin{equation}\label{4-9}
(\omega_1,\omega_2) = \mca(\omega_1,\omega_2) :=(I-T)^{-1} (-\Delta)^{-1} \left( \mathbf P l + \mathbf P N(\omega_1,\omega_2)\right),\;\;(\omega_1,\omega_2) \in \mathbf E.
\end{equation}
By Proposition~\ref{proposition3-5} again, we have
$$
\|\mca(\omega_1,\omega_2) \|_* \le C(\| l\|_{**} + \|N(\omega_1,\omega_2)\|_{**})
$$
and
$$
\|\mca(\omega_1,\omega_2) -\mca(\widetilde{\omega_1},\widetilde{\omega_2})\|_* \le C \|  N(\omega_1,\omega_2) -N(\widetilde{\omega_1},\widetilde{\omega_2}) \|_{**}.
$$
Using Lemmas~\ref{lem2.5} and \ref{lemma4-1}, we can prove that $A$ is a contraction mapping from the set
$$\left\{ (\omega_1,\omega_2) \in \mathbf E: \; \|(\omega_1,\omega_2)\|_* \le C_0\mu^{ -\frac {N}{2(q+1)}-\sigma  }\right\}$$
to itself, provided $C_0 > 0$ large enough. Hence the contraction mapping theorem implies that \eqref{4-9} has a unique solution $(\omega_1,\omega_2) \in \mathbf E$ satisfying
$$
\| (\omega_1,\omega_2) \|_{*}\le C \| l\|_{**} \le C_0 \mu^{ -\frac {N}{2(q+1)}-\sigma   }
$$
provided $k > 0$ large. Besides, it is standard to  show that $\|(\omega_1,\omega_2)\|_*$ is $C^1$ in $(\mu_0,r,\lambda)$ by employing the implicit function theorem.

Let $t > 0$. It is simple to check that
\[\begin{cases}
\displaystyle U_*\left[t\mu_0, t^{-1}r, t\mu\right](y) = t^{\frac{N}{q+1}} U_*[\mu_0,r,\mu](ty),\\
\displaystyle V_*\left[t\mu_0, t^{-1}r, t\mu\right](y) = t^{\frac{N}{p+1}} V_*[\mu_0,r,\mu](ty)
\end{cases}\]
for $y \in \R^N$. Also, if $(\omega_1[\mu_0,r,\mu],\omega_2[\mu_0,r,\mu])$ is a solution of \eqref{4-1} with parameters $(\mu_0,r,\mu)$,
then the pair $(t^{\frac{N}{q+1}} \omega_1[\mu_0,r,\mu](t\cdot), t^{\frac{N}{p+1}} \omega_2[\mu_0,r,\mu](t\cdot))$ will satisfy \eqref{4-1} with $(t\mu_0, t^{-1}r, t\mu)$. By the uniqueness of solutions for \eqref{4-1}, we see that
\eqref{10-12-3} holds.
\end{proof}

\section{Existence of solutions}\label{sec:exist}
Given $k \in \N$ large enough, we set a functional
\begin{equation}\label{eq:redene}
K(\mu_0,r,\lambda) := I(U_*[\mu_0,r,\mu]+\omega_1[\mu_0,r,\mu], V_*[\mu_0,r,\mu]+\omega_2[\mu_0,r,\mu]),
\end{equation}
where $(U_*,V_*)$ is defined by \eqref{Vdef}--\eqref{U*V*}, $(\omega_1,\omega_2)$ is built in Proposition~\ref{proposition4-3}, and $\mu = \lambda k^{\frac{(p+1)(N-2)}{N}}$.

By the standard reduction argument, if we prove the existence of a critical point for $K$, then we will obtain a solution for system \eqref{2} of the form $(u,v) = (U_*+\omega_1, V_*+\omega_2)$.

\begin{proof}[Proof of Theorem~\ref{th1-25-2}]
By exploiting \eqref{a1}, \eqref{4-1}, \eqref{3-3}--\eqref{3-5} and $(\omega_1,\omega_2) \in \mathbf E$, we see
\begin{align}
K(\mu_0,r,\lambda) - I( U_*, V_* )
&= O\(\int_{\R^N} |\omega_2|^{p+1} + \int_{\R^N} (|l_1|+|N_1(\omega_2)|)|\omega_2|\) \label{eq:diff}
\\
&\ + O\(\int_{\R^N} |U_*|^{q-2} \omega_1^3 + \int_{\R^N} |\omega_1|^{q+1} + \int_{\R^N} (|l_2|+|N_2(\omega_1)|)|\omega_1|\). \nonumber
\end{align}
Here, the integral $\int_{\R^N} |U_*|^{q-2} \omega_1^3$ can be replaced by $0$ whenever $q \le 2$, which happens for every $N \ge 9$ and $p \in (\frac{N}{N-2},\frac{N+2}{N-2}]$.
Furthermore, reasoning as in the proof of \cite[Proposition 3.1]{WY}, we discover
\begin{equation}\label{eq:diff2}
\int_{\R^N} |\omega_2|^{p+1} + \int_{\R^N} |U_*|^{q-2} \omega_1^3 + \int_{\R^N} |\omega_1|^{q+1} \le Ck^{1+\theta} \|(\omega_1, \omega_2)\|_*^{\min\{p+1,\, 3,\, q+1\}}
\end{equation}
for some $\theta > 0$ arbitrarily small, and
\begin{multline}\label{eq:diff3}
\int_{\R^N} (|l_1|+|N_1(\omega_2)|)|\omega_2| + \int_{\R^N} (|l_2|+|N_2(\omega_1)|)|\omega_1| \\
\le Ck(\|l\|_{**}+\|N(\omega_1,\omega_2)\|_{**})\|(\omega_1, \omega_2)\|_*.
\end{multline}

As an illustration, let us derive $\int_{\R^N} |\omega_1|^{q+1} \le Ck^{1+\theta}\|\omega_1\|_{*,1}^{q+1}$. It holds that
\begin{align*}
\int_{\R^N} |\omega_1|^{q+1} &\le \|\omega_1\|_{*,1}^{q+1} \int_{\R^N} \left[\sum_{j=1}^{k} \frac{\mu^{\frac{N}{q+1}}}{(1+\mu |y-x_j|)^{\frac{N}{q+1}+\tau}}\right]^{q+1} dy \\
&= \|\omega_1\|_{*,1}^{q+1} \int_{\R^N} \left[\sum_{j=1}^{k} \frac{1}{(1+|y-\mu x_j|)^{\frac{N}{q+1}+\tau}}\right]^{q+1} dy.
\end{align*}
As shown in \cite[Lemma B.1]{WY}, given any $\alpha,\, \beta > 0$, $0 < \sigma \le \min\{\alpha,\beta\}$, and $\tix_1,\, \tix_2 \in \R^N$ such that $\tix_1 \ne \tix_2$, there exists a constant $C > 0$ such that
\begin{equation}\label{eq:b1}
\begin{medsize}
\displaystyle \frac{1}{(1+|y-\tix_1|)^{\alpha}}\frac{1}{(1+|y-\tix_2|)^{\beta}} \le \frac{C}{|\tix_1-\tix_2|^{\sigma}} \left[\frac{1}{(1+|y-\tix_1|)^{\alpha+\beta-\sigma}}+ \frac{1}{(1+|y-\tix_2|)^{\alpha+\beta-\sigma}}\right]
\end{medsize}
\end{equation}
for all $y \in \R^N$. A closer inspection on its proof yields that \eqref{eq:b1} holds for all $y \in \R^N$ such that $|y-\tix_2| \ge \frac{1}{2}|\tix_1-\tix_2|$ provided $\alpha,\, \beta > 0$ and $0 < \sigma \le \beta$ only. Thus,
\begin{align*}
\sum_{j=2}^{k} \frac{1}{(1+|y-\mu x_j|)^{\frac{N}{q+1}+\tau}} &\le \sum_{j=2}^{k} \frac{1}{(1+|y-\mu x_1|)^{\frac{N}{q+1}}}\frac{1}{(1+|y-\mu x_j|)^{\tau}} \\
&\le C\, \frac{1}{(1+|y-\mu x_1|)^{\frac{N}{q+1}+\tau\theta}} \sum_{j=2}^{k} \frac{1}{|\mu x_1-\mu x_j|^{\tau(1-\theta)}} \\
&\le Ck^{\theta}\, \frac{1}{(1+|y-\mu x_1|)^{\frac{N}{q+1}+\tau\theta}}
\end{align*}
for all $y \in \Omega_1$ (so that $|y-\mu x_j| \ge \frac{1}{2}|\mu x_1-\mu x_j|$ for all $j = 2,\ldots,k$) and any $\theta \in (0,\tau)$ small enough. It follows that
\[\int_{\R^N} \left[\sum_{j=1}^{k} \frac{1}{(1+|y-\mu x_j|)^{\frac{N}{q+1}+\tau}}\right]^{q+1} dy \le Ck^{(q+1)\theta},\]
proving the assertion.

From \eqref{eq:diff}--\eqref{eq:diff3}, Lemmas~\ref{lem2.5} and \ref{lemma4-1}, and Proposition~\ref{proposition4-3}, we obtain the expansion
\begin{align*}
K(\mu_0,r,\lambda) &= I( U_*, V_* ) + O\(k\|(\omega_1, \omega_2)\|_*^2 + k(\|l\|_{**}+\|N(\omega_1,\omega_2)\|_{**})\|(\omega_1, \omega_2)\|_*\) \\
&= I( U_*, V_* ) + O\(k\mu^{-\frac N{q+1}-\sigma}\)
\end{align*}
for some $\sigma > 0$ small. Hence, by Proposition~\ref{prop:Iexpan}, we have that
\begin{equation}\label{61-24-2}
K(\mu_0,r,\lambda) = (k+1)A +\frac{k}{k^{\frac{(p+1)(N-2)}{q+1}}} \left[-\frac{B_1}{r^{N-2}\lambda^{N-2}} + \frac{B_2 U_{0,\mu_0}(r)}{\lambda^{\frac{N}{q+1}}}\right] + O\(\frac{k}{k^{\frac{(p+1)(N-2)}{q+1}+\sigma}}\).
\end{equation}

\medskip
We define
\begin{equation}\label{n30-31-12}
F(\mu_0,r,\lambda) := -\frac{B_1}{r^{N-2}\lambda^{N-2}} + \frac{B_2 U_{0,\mu_0}(r)}{\lambda^{\frac{N}{q+1}}}, \quad (\mu_0,r,\lambda) \in (0,+\infty)^3.
\end{equation}
Then, we see from \eqref{6} that $F(\mu_0,r,\lambda) = F(r\mu_0,1,r\lambda)$. We also set
\[F^*(M_0,\Lambda) := F(M_0,1,\Lambda), \quad (M_0,\Lambda) \in (0,+\infty)^2.\]
For any fixed $M_0 > 0$, the function $\Lambda \mapsto F^*(M_0,\Lambda)$ attains its maximum at $\Lambda(M_0)$ determined by
\begin{equation}\label{n33-31-12}
\frac{(N-2)B_1}{\Lambda^{N-2}} - \frac{N}{q+1}\frac{B_2U_{0,M_0}(1)}{\Lambda^{\frac{N}{q+1}}}=0.
\end{equation}
Solving \eqref{n33-31-12}, we find
\[\Lambda(M_0) = \left[\frac{(q+1)(N-2)B_1}{B_2 N U_{0,M_0}(1)}\right]^{\frac{p+1}{N}}.\]
Now, we have
\begin{align*}
F_1^*(M_0) &:= F^*(M_0,\Lambda(M_0))
= \frac N{(p+1)(N-2)}\frac{B_2 U_{0,M_0}(1)}{\Lambda(M_0)^{\frac{N}{q+1}}}\\
&= \frac {B_2 N}{(p+1)(N-2)} \left[\frac{B_2 N}{(q+1)(N-2)B_1}\right]^{\frac{p+1}{q+1}} \cdot \(M_0^{\frac N{q+1}} U_{0,1}(M_0)\)^{\frac{(p+1)(N-2)}{N}}
\end{align*}
for any $M_0 > 0$. Since $U_{0,1}(t) \simeq t^{-(N-2)}$ for $t > 0$ large, there exist a large number $M_{00} > 1$ and a small number $\ep_0 > 0$ such that
\[\max\{F_1^*(M_0): M_{00}^{-1} < M_0 < M_{00}\} \ge \max\left\{F_1^*(M_{00}), F_1^*(M_{00}^{-1})\right\} + \ep_0.\]
Thus one can find a large number $\Lambda_0 > 1$ and a small number $\ep_1 > 0$ such that
\[\max\{F^*(M_0,\Lambda): (M_0,\Lambda) \in \mcq\} \ge \max\{F^*(M_0,\Lambda): (M_0,\Lambda) \in \partial \mcq\} + \ep_1,\]
where $\mcq := (M_{00}^{-1},M_{00}) \times (\Lambda_0^{-1},\Lambda_0)$ is an open cube in $(0,+\infty)^2$.

By \eqref{61-24-2}, the map $K^*(M_0,\Lambda) := K(M_0,1,\Lambda)$ has a maximum point $(M_0^*,\Lambda^*) \in \mcq$ provided $k$ sufficiently large.
By slightly increasing the values of $M_{00}$ and $\Lambda_0$ if needed, we can take a small number $\delta_0 > 0$ such that
\[(M_0^*,\Lambda^*) \in \mcq_{\delta_0} := (\mu_{00}^{-1},\mu_{00}) \times (\lambda_0^{-1},\lambda_0) \subset \mcq,\]
where we set $(\mu_{00},\lambda_0) := ((1-\delta_0)M_{00},(1-\delta_0)\Lambda_0)$. Because \eqref{10-12-3} tells us that
\begin{equation}\label{Kinv}
K(\mu_0,r,\lambda) = K(r\mu_0,1,r\lambda) =  K^*(r\mu_0,r\lambda), \quad (\mu_0,r,\lambda) \in (0,+\infty)^3
\end{equation}
and $(r\mu_0,r\lambda) \in \mcq$ whenever $r \in (1-\delta_0,1+\delta_0)$ and $(\mu_0,\lambda) \in \mcq_{\delta_0}$, we have
\[K(M_0^*,1,\Lambda^*) \ge K(\mu_0,r,\lambda), \quad (\mu_0,r,\lambda) \in (\mu_{00}^{-1},\mu_{00}) \times (r_0^{-1},r_0) \times (\lambda_0^{-1},\lambda_0)\]
where $r_0 := 1+\delta_0$.
Therefore, $(M_0^*,1,\Lambda^*)$ is the maximum point of $K$, which is clearly a critical point of $K$.
\end{proof}
\begin{remark}
The maximum point $(M_0^*,1,\Lambda^*)$ of the function $K$ is neither strict nor isolated.
In view of \eqref{Kinv}, every member of the set $\{(r^{-1}M_0^*,r,r^{-1}\Lambda^*): r > 0\}$ is a maximum point of $K$.
\end{remark}

\section{The case $p \in (1,\frac N{N-2})$}\label{sec:plow}
In this section, we briefly describe how to handle when $N \ge 7$ and $p \in (1,\frac{N}{N-2})$.
We mainly point out the necessary changes.
\medskip

Firstly, we discuss the estimate for $\varphi= U -\sum_{j=1}^k U_j.$
Using the representation \eqref{nn2-19-1} and the condition $p(N-2)<N$,
one can derive the following estimate for  $\varphi$:
\begin{equation}\label{1-19-1}
\begin{aligned}
\varphi(y)
&=\frac1{\mu^{\frac {pN}{q+1}}}\int_{\R^N} \frac{\ga_N}{|y-z|^{N-2}} \left[\(\sum_{j=1}^k \frac{b_{N,p}}{|z-x_j|^{N-2}}\)^p - \sum_{j=1}^k \frac{b_{N,p}^p}{|z-x_j|^{p(N-2)}}\right] dz\\
&\ + O\Bigl(   \frac{k^{p(N-2)-2}}{\mu^{\frac{pN}{q+1}+\sigma}}\Bigr)
\\
&=\frac {k^{p(N-2)-2}}{r^{p(N-2)-2}\mu^{\frac {Np}{q+1}}}H(r^{-1} k y)+O\Bigl(   \frac{k^{p(N-2)-2}}{\mu^{\frac{pN}{q+1}+\sigma}}\Bigr),
\end{aligned}
\end{equation}
where
\begin{equation}\label{H}
H(y) := \int_{\R^N} \frac{\ga_N}{|y-z|^{N-2}} \left[\(\sum_{j=1}^k \frac{b_{N,p}}{|z-\tilde x_j|^{N-2}}\)^p - \sum_{j=1}^k \frac{b_{N,p}^p}{|z-\tilde x_j|^{p(N-2)}}\right] dz
\end{equation}
for $y \in \R^N$, $\tilde x_j= (k \cos\frac{2(j-1)\pi }k, k\sin\frac{2(j-1)\pi }k, 0)$, $\ga_N = \frac{1}{(N-2)|\Sn^{N-1}|}$, and $b_{N,p} > 0$ is the constant in \eqref{eq:HV}. It is not difficult to prove that following result.

\begin{lemma}
Assume that $N \ge 6$ and $p \in (1,\frac{N}{N-2})$. Then The function $H$ is well-defined in $\R^N$. Furthermore, the following properties hold: Let $\delta,\, \theta > 0$ be sufficiently small numbers.
\begin{itemize}
\item [(a)] For $y\in \cup_{j=1}^k B_\delta(\tilde x_j)$, it holds that
    \[
    0<H(y)\le C
    \]
for some constant $C > 0$ depending only on $N$ and $p$.
\item [(b)] For $y\in \R^N\setminus \cup_{j=1}^k B_\delta(\tilde x_j)$, it  holds that
\[|H(y)|\le C\sum_{j=1}^k \left[\frac 1{|y-\tilde x_j|^{(p-1)(N-2)+N-5-\theta}} + \frac 1{|y-\tilde x_j|^{p(N-3-\theta)-2}}\right].\]
\end{itemize}
\end{lemma}
With  estimate \eqref{1-19-1}, we have the following energy expansion
for the case $p\in (1,\frac N{N-2})$.
\begin{proposition}
Assume that $N \ge 7$ and $p \in (1,\frac{N}{N-2})$. For $k \in \N$ large enough, we have
\begin{multline*}
I(U_*, V_*) = (k+1)A\\
+ k\left[-\frac{(B_3+B_4 H(\tilde x_{1})) k^{p(N-2)-2}}{r^{p(N-2)-2} \mu^{p(N-2)-2}} + \frac{B_2 U_{0,\mu_0}(r)}{\mu^{\frac{N}{q+1}}}
+ O\(\frac{k^{p(N-2)-2}}{\mu^{p(N-2)-2+\sigma}} + \frac{1}{\mu^{\frac{N}{q+1}+\sigma}}\)\right].
\end{multline*}
Here, $A = I(U_{0, \mu_0}, V_{0, \mu_0}) = I(U_{0,1}, V_{0,1})$, $B_1$, $B_2$, $B_3$ and $B_4$ are positive constants depending only on $N$ and $p$, and  $\sigma>0$ is a sufficiently small number.
\end{proposition}
\noindent The reader is advised to compare this with the  corresponding result in Proposition~\ref{prop:Iexpan} for the case $p \in (\frac N{N-2},\frac{N+2}{N-2}]$.

\medskip
With the above result established, we can proceed as in Sections \ref{sec:lin}--\ref{sec:exist} to build solutions $(u, v)\approx (U_*, V_*)$ defined in \eqref{U*V*}.
Especially, we can use the same $*$- and $**$-norms introduced in \eqref{eq:*-norm}--\eqref{eq:**-norm}.

\appendix
\section{Green's functions}\label{sec:Green}
Throughout this appendix, we assume that $N \ge 5$ and $(p,q)$ satisfies $p \in (1,\frac{N+2}{N-2}]$ and \eqref{cri}, which implies that $q \in [\frac{N+2}{N-2},\frac{N+4}{N-4})$.

\medskip
Given a function $w$, the rotation operator $\Phi_j$
 and the reflection operator $\Psi_h$ defined in \eqref{Aj} and \eqref{Bi}
 respectively, let
\[
\hat w(y) := \frac1 k\sum_{j=1}^k w(\Phi_j y)
\]
and
\begin{equation}\label{wstar}
w^*(y) := \frac1{2(N-1)}\sum_{h=2}^N \(\hat w(y)+ \hat w(\Psi_h y)\)
\end{equation}
for $y \in \R^N$. We call $w^*$ the symmetrization of $w$. If $(u,v)\in
C(\R^N)\times C(\R^N)$,  then $(u^*,v^*) \in \mathbf{L}_s$,
where  $\mathbf{L}_s$ is the function space in \eqref{Hs}.

For $\mu_0 \in (\mu_{00}^{-1},\mu_{00})$ where $\mu_{00} > 1$ is the number determined in Section \ref{sec:exist}, we introduce an operator
\begin{multline*}
L(u,v) := \(-\Delta u - p V_{0,\mu_0}^{p-1} v, -\Delta v  -q U_{0,\mu_0}^{q-1} u\),\\
(u,v) \in \mathbf{L}_s \cap \left[\dot{W}^{2,\frac{p+1}{p}}(\R^N) \times \dot{W}^{2,\frac{q+1}{q}}(\R^N)\right]
\end{multline*}
and its formal dual operator
\begin{multline}\label{1-25-10n2}
L^*(u, v) := \(-\Delta u - q U_{0,\mu_0}^{q-1} v, -\Delta v - p V_{0,\mu_0}^{p-1} u\),\\
(u,v) \in \mathbf{L}_s \cap \left[\dot{W}^{2,\frac{q+1}{q}}(\R^N) \times \dot{W}^{2,\frac{p+1}{p}}(\R^N)\right].
\end{multline}
By the non-degeneracy of $L$ depicted in Lemma~\ref{lemma:FKP},
\begin{equation}\label{kerL*}
\text{the kernel of $L^*$} = \text{span} \left\{(Z_0,Y_0)\right\}.
\end{equation}

\medskip
In this appendix, we will investigate the Green's function of $L^*$ tailored to our setting.
If $x \in \R^N \setminus \{0\}$, then the Dirac delta function $\delta_x$ is not invariant under the action of $\Phi_j$ and $\Psi_h$.
To retain symmetry under these actions, we define
\[
\delta_x^* := \frac1{2(N-1)k} \sum_{h=2}^N \sum_{j=1}^k \(\delta_{\Phi_j x} + \delta_{\Psi_h \Phi_j x}\)
\]
and consider the problems
\begin{equation}\label{2-25-10n}
L^* (u, v) =  (\delta_x^*, 0)+ c(x) \(q U_{0,\mu_0}^{q-1}Y_0, p V_{0,\mu_0}^{p-1}Z_0\) \quad \text{in } \R^N
\end{equation}
and
\begin{equation}\label{nn2-25-10n}
L^*(u, v) =  (0, \delta_x^*)+ \tilde{c}(x) \(q U_{0,\mu_0}^{q-1}Y_0,\;p V_{0,\mu_0}^{p-1}Z_0\) \quad \text{in } \R^N
\end{equation}
for each $x \in \R^N$, where $c, \tilde{c}: \R^N \to \R$ are functions of $x$.

We first deduce an existence result for \eqref{2-25-10n}. Henceforth, we will indicate when the universal constant $C > 0$ depends on $R$ by writing it with a subscript as $C_R$.
\begin{proposition}\label{p1-25-10n}
Given a fixed large number $R>1$, there exists a function $(G_{1,k}, G_{2,k})$ such that $(G_{1,k}(\cdot, x), G_{2,k}(\cdot, x)) \in \mathbf{L}_s$ solves \eqref{2-25-10n} for each $x \in B_R(0)$.
\end{proposition}
\begin{proof}
We set $u_1 = \frac{\ga_N}{|\cdot-x|^{N-2}}$ and $ v_1=0$, where $\ga_N > 0$ is the number appearing in \eqref{H}. Clearly,
\[\begin{cases}
-\Delta u_1- q U_{0, \mu_0}^{q-1} v_1=\delta_x & \text{in } B_{2R}(0),\\
-\Delta v_1- p V_{0, \mu_0}^{p-1} u_1=-p V_{0, \mu_0}^{p-1}\frac{\ga_N}{ |\cdot-x|^{N-2}} & \text{in } B_{2R}(0).
\end{cases}\]
Let $\bar v_2$ be the solution of
\[
\begin{cases}
-\Delta v= -p  V_{0, \mu_0}^{p-1}\frac{\ga_N}{|\cdot-x|^{N-2}} &\text{in } B_{2R}(0),\\
v=0 &\text{on}\; \partial B_{2R}(0).
\end{cases}
\]
Then,
\begin{align*}
|\bar v_2(y)| &\le \int_{B_{2R}(0)} G(y,z)p V_{0, \mu_0}^{p-1}(z)\frac{\ga_N}{ |z-x|^{N-2}}\, dz\\
&\le C \int_{B_{2R}(0)}\frac1{|y-z|^{N-2}}\frac1{|z-x|^{N-2}}\,dz \le \frac{C}{|y-x|^{N-4}}
\end{align*}
for $y \in B_{2R}(0)$, where $G(z,y)$ is the Green's function of $-\Delta$ in $B_{2R}(0)$ with zero Dirichlet boundary condition.
Moreover, $(u_2,v_2) := (u_1, v_1-\bar v_2)$ satisfies
\[\begin{cases}
-\Delta u_2- q U_{0, \mu_0}^{q-1} v_2=\delta_x+q U_{0, \mu_0}^{q-1}\bar v_2 & \text{in } B_{2R}(0),\\
-\Delta v_2- p V_{0, \mu_0}^{p-1} u_2=0 & \text{in } B_{2R}(0).
\end{cases}\]
Let $\bar u_3$ be the solution of
\[
\begin{cases}
-\Delta u= q U_{0, \mu_0}^{q-1}\bar v_2 & \text{in } B_{2R}(0),\\
u=0 & \text{on}\; \partial B_{2R}(0).
\end{cases}
\]
Then,
\[
|\bar u_3(y)|\le \begin{cases}
\displaystyle \frac{C}{|y-x|^{N-6}} &\text{if } N \ge 7,\\
\displaystyle \log\frac{C_R}{|y-x|} &\text{if } N = 6,\\
\displaystyle C_R &\text{if } N = 5
\end{cases}
\]
for $y \in B_{2R}(0)$, and $(u_3,v_3) := (u_2-\bar u_3, v_2)$ satisfies
\[\begin{cases}
-\Delta u_3- q U_{0, \mu_0}^{q-1} v_3=\delta_x& \text{in } B_{2R}(0),\\
-\Delta v_3- p V_{0, \mu_0}^{p-1} u_3= p V_{0, \mu_0}^{p-1}\bar u_3& \text{in } B_{2R}(0).
\end{cases}\]
We can continue this process to build $(\bar u_l, \bar v_l)$ for an integer $l \ge 4$, and find
\[|\bar u_l(y)| + |\bar v_l(y)| \le
\begin{cases}
\displaystyle \frac{C_R}{|y-x|^{N-(2l-2)}} &\text{if } N \ge 2l-1,\\
\displaystyle \log\frac{C_R}{|y-x|} &\text{if } N = 2l-2,\\
\displaystyle C_R &\text{if } 5 \le N \le 2l-3
\end{cases}
\]
for $y \in B_{2R}(0)$ and $l \ge 2$.

We select an integer $l_0 \ge \frac{N+3}{2}$ so that $|\bar u_l(y)| + |\bar v_l(y)| \le C_R$ for $y \in B_{2R}(0)$. Then the associated pair $(u_{l_0},v_{l_0})$ satisfies
\[\begin{cases}
-\Delta u_{l_0} - q U_{0, \mu_0}^{q-1} v_{l_0} = \delta_x+\bar f_1& \text{in } B_{2R}(0),\\
-\Delta v_{l_0} - p V_{0, \mu_0}^{p-1} u_{l_0} = \bar f_2 & \text{in } B_{2R}(0)
\end{cases}\]
for some $(\bar f_1, \bar f_2)$ such that $\|\bar f_1\|_{L^{\infty}(B_{2R}(0))} + \|\bar f_2\|_{L^{\infty}(B_{2R}(0))} \le C_R$.

We take a radially symmetric cut-off function $\chi \in C^\infty_0(B_{2R}(0))$ such that $0\le \chi\le 1$ in $B_{2R}(0)$ and $\chi=1$ in $B_{\frac{3R}2}(0)$.
Also, let $(u_{l_0}^*,v_{l_0}^*)$ be the symmetrization of $(u_{l_0},v_{l_0})$ defined by \eqref{wstar}.
If we write $ \tilde u=\chi u_{l_0}^*$ and $\tilde v=\chi v_{l_0}^*$, it holds that $(\tilde u, \tilde v) \in \mathbf{L}_s$ and
\[
L^* (\tilde u, \tilde v) = (\delta^*_{x}, 0)+(f_1,f_2) \quad \text{in } \R^N
\]
for some $(f_1,f_2) \in \mathbf{L}_s$ such that $\|f_1\|_{L^{\infty}(\R^N)} + \|f_2\|_{L^{\infty}(\R^N)} \le C_R$ and $f_1=f_2=0$ in $\R^N\setminus B_{2R}(0)$.

Now we consider
\begin{equation}\label{1-1-6-21}
L^* (u, v) = \(f_1-c(x)q U_{0,\mu_0}^{q-1}Y_0, f_2-c(x)pV_{0,\mu_0}^{p-1}Z_0\) \quad \text{in } \R^N
\end{equation}
with $c(x) \in \R$ satisfying
\begin{equation}\label{1-27-1}
\int_{\R^N} \left[\(f_1-c(x)q U_{0,\mu_0}^{q-1}Y_0\)Y_0+
\(f_2-c(x)pV_{0,\mu_0}^{p-1}Z_0\)Z_0\right]=0.
\end{equation}
By \eqref{kerL*}, system \eqref{1-1-6-21}--\eqref{1-27-1} has a unique solution $(w_1,w_2) \in \mathbf{L}_s\cap (\dot{W}^{2,\frac{q+1}{q}}(\R^N) \times \dot{W}^{2,\frac{p+1}{p}}(\R^N))$ such that
\[p\int_{\R^N} w_1 V_{0,\mu_0}^{p-1}Z_0 + q\int_{\R^N} w_2 U_{0,\mu_0}^{q-1}Y_0 = 0.\]
As a result, the pair $(G_{1,k}(\cdot, x), G_{2,k}(\cdot, x)) := (\tilde u- w_1, \tilde v -w_2) \in \mathbf{L}_s$ solves \eqref{2-25-10n}.
\end{proof}

We turn to deriving a pointwise estimate of $(G_{1,k}(\cdot,x), G_{2,k}(\cdot,x))$.
\begin{proposition}\label{p1-25-10n2}
Under the setting of Proposition~\ref{p1-25-10n}, there exists a constant $C_R>0$ depending only on $N$, $p$, $\mu_{00}$, and $R$ such that
 for all $y \in B_{2R}(0)$ and $x \in B_R(0)$, it holds
\begin{equation}\label{G1k}
|G_{1,k}(y, x)| \le \frac{C_R}{k} \sum_{h=2}^N\sum_{j=1}^k \(\frac1{|y-\Phi_j x|^{N-2}} + \frac1{|y-\Psi_h\Phi_j x|^{N-2}}\)
\end{equation}
and
\begin{equation}\label{G2k1}
|G_{2,k}(y, x)| \le \displaystyle \frac{C_R}{k} \sum_{h=2}^N\sum_{j=1}^k \(\frac{1}{|y-\Phi_j x|^{N-4}} + \frac{1}{|y-\Psi_h\Phi_j x|^{N-4}}\).
\end{equation}
Moreover, for all $y \in \R^N \setminus B_{2R}(0)$ and $x \in B_R(0)$, we have
\begin{equation}\label{G1k2}
|G_{1,k}(y, x)| \le \frac{C}{|y|^{N-2}}
\end{equation}
and
\begin{equation}\label{G2k2}
|G_{2,k}(y, x)| \le \begin{cases}
\displaystyle \frac{C_R}{|y|^{N-2}} &\text{if } p \in (\frac{N}{N-2},\frac{N+2}{N-2}],\\
\displaystyle \frac{C_R \log|y|}{|y|^{N-2}} &\text{if } p = \frac{N}{N-2},\\
\displaystyle \frac{C_R}{|y|^{p(N-2)-2}} &\text{if } p \in (1,\frac{N}{N-2}).
\end{cases}
\end{equation}

\end{proposition}
\begin{proof}
Here, we keep using the notations from the proof of Proposition~\ref{p1-25-10n}.
Firstly, we note that the functions $\tilde u$ and $\tilde v$ satisfy \eqref{G1k}
and \eqref{G2k1} respectively. Thus, to prove this proposition,
we just need to show that for any $x\in B_R(0)$, the solution $(w_1,w_2)$ of \eqref{1-1-6-21}
is bounded in $B_{2R}(0)$. This follows directly from the $L^p$-estimates for the
elliptic equations, since the functions on the right hand side of \eqref{1-1-6-21}
is uniformly bounded.

\medskip
Next, we proceed to prove the inequalities
\begin{equation}\label{w1w2decay}
|w_1(y)| \le \frac{C_R}{|y|^{N-2}}
\quad \text{and} \quad
|w_2(y)| \le \begin{cases}
\displaystyle \frac{C_R}{|y|^{N-2}} &\text{if } p \in (\frac{N}{N-2},\frac{N+2}{N-2}],\\
\displaystyle \frac{C_R \log|y|}{|y|^{N-2}} &\text{if } p = \frac{N}{N-2},\\
\displaystyle \frac{C_R}{|y|^{p(N-2)-2}} &\text{if } p \in (1,\frac{N}{N-2})
\end{cases}
\end{equation}
for $|y| \ge 2R$.

For the moment, we assume that $p \in (1,\frac{N}{N-2}) \cup (\frac{N}{N-2},\frac{N+2}{N-2}]$.
Suppose that there exists $\theta_0 \ge 0$ such that $|w_2(y)| \le \frac{C_R}{1+|y|^{\theta_0}}$ for $y \in \R^N$. By \eqref{1-1-6-21}--\eqref{1-27-1},
\begin{align*}
|w_1(y)| &\le C_R\int_{\R^N} \frac1{|y-z|^{N-2}} \(U_{0,\mu_0}^{q-1}|w_2| + |f_1| + U_{0,\mu_0}^{q}\)(z)\, dz \\
&\le C_R\int_{\R^N} \frac1{|y-z|^{N-2}} \left[\frac{U_{0,\mu_0}^{q-1}(z)}{1+|z|^{\theta_0}} + U_{0,\mu_0}^{q}(z)\right] dz \\
&\le \dfrac{C_R}{1+|y|^{\theta_1}}
\end{align*}
for $y \in \R^N$, where
\[\theta_1 := \begin{cases}
\min\{(q-1)(N-2)-2+\theta_0,N-2\} &\text{if } p \in (\frac{N}{N-2},\frac{N+2}{N-2}], \\
\min\{(q-1)(p(N-2)-2)-2+\theta_0,N-2\} &\text{if } p \in (1,\frac{N}{N-2}).
\end{cases}\]
In both cases, $\theta_1 \ge \min\{2+\theta_0,N-2\}$. Thus we observe
\begin{align*}
|w_2(y)| &\le C_R \left[\int_{\R^N} \frac1{|y-z|^{N-2}} \frac{V_{0,\mu_0}^{p-1}(z)}{1+|z|^{\theta_1}} dz
+ \int_{\R^N} \frac1{|y-z|^{N-2}} \(|f_2|+V_{0,\mu_0}^{p}\)(z)\, dz\right] \\
&\le \frac{C_R}{1+|y|^{\theta_2}}
\end{align*}
for $y \in \R^N$, where
\begin{align*}
\theta_2 &:= \min\{(p-1)(N-2)-2+\theta_1,N-2,p(N-2)-2\} \\
&\ge \min\{(p-1)(N-2)+\theta_0,N-2,p(N-2)-2\},
\end{align*}
which improves the decay of $w_2$. Therefore, reminding that $w_2 \in L^{\infty}(\R^N)$ (more precisely, $\|w_2\|_{L^{\infty}(\R^N)} \le C_R$),
we can iterate the above process finitely many times to establish \eqref{w1w2decay}.

The case $p = \frac{N}{N-2}$ can be treated in a similar way, which we omit.

\medskip
Now, from the construction of $(u_{l_0},v_{l_0})$ in the proof of Proposition~\ref{p1-25-10n} and \eqref{w1w2decay}, we infer that \eqref{G1k}--\eqref{G2k2} is true. This completes the proof.
\end{proof}

 Similar to Propositions~\ref{p1-25-10n} and \ref{p1-25-10n2}, we have the following results for the  solution of \eqref{nn2-25-10n}.
\begin{proposition}
The solution $(G_{3,k}(\cdot, x), G_{4,k}(\cdot, x)) \in \mathbf{L}_s$ of \eqref{nn2-25-10n} exists.
Furthermore, given a fixed large number $R>1$, there exists a constant $C_R>0$ depending only on $N$, $p$, $\mu_{00}$ and $R$ such that
 for all $y \in B_{2R}(0)$ and $x \in B_R(0)$, it holds
\begin{equation}\label{G3k1}
|G_{3,k}(y, x)| \le \frac{C_R}{k} \sum_{h=2}^N\sum_{j=1}^k \(\frac{1}{|y-\Phi_j x|^{N-4}} + \frac{1}{|y-\Psi_h\Phi_j x|^{N-4}}\)
\end{equation}
and
\begin{equation}\label{G4k1}
|G_{4,k}(y, x)|\le \frac{C_R}{k} \sum_{h=2}^N\sum_{j=1}^k \(\frac1{|y-\Phi_j x|^{N-2}} + \frac1{|y-\Psi_h\Phi_j x|^{N-2}}\).
\end{equation}
Moreover, for all $y \in \R^N \setminus B_{2R}(0)$ and $x \in B_R(0)$, we have
\begin{equation}\label{G3k2}
|G_{3,k}(y, x)| \le \frac{C}{|y|^{N-2}}
\end{equation}
and
\begin{equation}\label{G4k2}
|G_{4,k}(y, x)| \le \begin{cases}
\displaystyle \frac{C_R}{|y|^{N-2}} &\text{if } p \in (\frac{N}{N-2},\frac{N+2}{N-2}],\\
\displaystyle \frac{C_R \log|y|}{|y|^{N-2}} &\text{if } p = \frac{N}{N-2},\\
\displaystyle \frac{C_R}{|y|^{p(N-2)-2}} &\text{if } p \in (1,\frac{N}{N-2}).
\end{cases}
\end{equation}
\end{proposition}

Finally, recalling the function $(Y_0,Z_0)$ given in \eqref{Y012}--\eqref{Z012}, let
\begin{multline*}
\mathbf{E}_0 := \left\{(u,v) \in \mathbf{L}_s \cap \left[\dot{W}^{2,\frac{p+1}{p}}(\R^N) \times \dot{W}^{2,\frac{q+1}{q}}(\R^N)\right]: \right.\\
\left. p\int_{\R^N} v V_{0,\mu_0}^{p-1}Z_0 + q\int_{\R^N} u U_{0,\mu_0}^{q-1}Y_0 = 0 \right\}.
\end{multline*}
We can derive a representation formula for any $(u,v) \in \mathbf{E}_0$ satisfying $L(u, v)= (f,g)$.
\begin{proposition}
If $(u,v) \in \mathbf{E}_0$ satisfies $L(u, v)= (f,g)$, it holds that
\begin{equation}\label{10-27-1}
u(x)=\int_{\R^N} G_{1,k}(y, x) f(y)\,dy+  \int_{\R^N} G_{2,k}(y, x) g(y)\,dy
\end{equation}
and
\begin{equation}\label{11-27-1}
v(x)=\int_{\R^N} G_{3,k}(y, x) f(y)\,dy+  \int_{\R^N} G_{4,k}(y, x) g(y)\,dy
\end{equation}
for all $x \in B_R(0)$. Here, the functions $G_{1,k}$, $G_{2,k}$, $G_{3,k}$   and $ G_{4,k}$ satisfy \eqref{G1k}--\eqref{G2k2} and \eqref{G3k1}--\eqref{G4k2}.
\end{proposition}
\begin{proof}
For any $(u,v) \in \mathbf{L}_s$, we have
\begin{align*}
&\quad u(x)+ c(x)\int_{\R^N} \(q u U_{0,\mu_0}^{q-1}Y_0+p v V_{0,\mu_0}^{p-1}Z_0\) \\
&= \int_{\R^N} \left[
\(-\Delta u - p V_{0,\mu_0}^{p-1}v\)(y) G_{1,k}(y,x) +
\(-\Delta v - q U_{0,\mu_0}^{q-1}u\)(y) G_{2,k}(y,x) \right] dy\\
&= \int_{\R^N} G_{1,k}(y, x) f(y)\,dy+  \int_{\R^N} G_{2,k}(y, x) g(y)\,dy.
\end{align*}
Since $(u, v)\in \mathbf{E}_0$, we get \eqref{10-27-1}. The proof of \eqref{11-27-1} is similar.
\end{proof}

\section{Technical computations}\label{sec:tech}
In this appendix, we compile several technical estimates used throughout the paper.
We remind the definition of the sets $\Omega_i$ for $i = 1,\ldots,k$ and $S = B_{\frac{\pi}{2}r_0^{-1}k^{-1}}(x_1) \subset \Omega_1$ from \eqref{Omegaj} and \eqref{eq:S}.

\medskip
To begin with, we prove a lemma that plays a key role in deducing \eqref{5-30-7}.
\begin{lemma}\label{l5-30-7}
Assume that $N \ge 5$ and $p \in [\frac{N}{N-2},\frac{N+2}{N-2}]$. Then
for $y \in \Omega_1$, it holds
\begin{multline}\label{l5-30-71}
\sum_{i=2}^k \int_{\Omega_i} \frac{1}{|y-z|^{N-2}} \left[\frac{1}{|z-kx_i|^{(p-1)(N-2)}} \sum_{j \ne i} \frac{1}{|z-kx_j|^{N-2}} + \(\sum_{j \ne i} \frac{1}{|z-kx_j|^{N-2}}\)^p\right] dz \\
\le C.
\end{multline}
\end{lemma}
\begin{proof}
Fix $i = 2,\ldots,k$. By applying \eqref{e4}, we obtain
\begin{align*}
&\quad \int_{\Omega_i} \frac{1}{|y-z|^{N-2}} \frac{1}{|z-kx_i|^{(p-1)(N-2)}} \sum_{j \ne i} \frac{1}{|z-kx_j|^{N-2}} dz \\
&\le \int_{\Omega_i} \frac{1}{|y-z|^{N-2}} \frac{1}{|z-kx_i|^{(p-1)(N-2)+N-3-\theta}} dz  \\
&\le \frac{C}{|y-kx_i|^{(p-1)(N-2)+N-5-\theta}}
\end{align*}
for $y \in \Omega_1$, where we choose $\theta \in (0,N-3]$ such that $(p-1)(N-2)+N-5-\theta\in (1, N-2)$. Thus,
\begin{multline}\label{l5-30-72}
\sum_{i=2}^k \int_{\Omega_i} \frac{1}{|y-z|^{N-2}} \frac{1}{|z-kx_i|^{(p-1)(N-2)}} \sum_{j \ne i} \frac{1}{|z-kx_j|^{N-2}} dz \\
\le
\displaystyle \sum_{i=2}^k \frac{C}{(k|x_1-x_i|)^{(p-1)(N-2)+N-5-\theta}} \le C.
\end{multline}

Similarly, for $y \in \Omega_1$, choosing $\theta \in (0,N-3]$ such that $(N-3-\theta)p-2\in (1, N-2)$,
we can compute
\begin{equation}\label{l5-30-73}
\begin{aligned}
&\quad \sum_{i=2}^k \int_{\Omega_i} \frac{1}{|y-z|^{N-2}} \(\sum_{j \ne i} \frac{1}{|z-kx_j|^{N-2}}\)^p dz\\
&\le \sum_{i=2}^k \int_{\Omega_i} \frac{1}{|y-z|^{N-2}} \frac{C}{|z-kx_i|^{(N-3-\theta)p}} dz\le \sum_{i=2}^k\frac{C}{|y-kx_i|^{(N-3-\theta)p-2}}
\le C.
\end{aligned}
\end{equation}
 Inequality \eqref{l5-30-71} now follows from \eqref{l5-30-72} and \eqref{l5-30-73}.
\end{proof}

Let $U$ be the unique solution of \eqref{Udef}. In the proof of Lemma~\ref{l1-28-12}, we need a pointwise estimate of $U$.
\begin{lemma}\label{lemma:U}
Assume that $N \ge 5$ and $p \in (\frac{N}{N-2},\frac{N+2}{N-2}]$. Then it holds that
\[
U(y) \le C \sum_{i=1}^k \frac{\mu^{\frac{N}{q+1}}}{(1+\mu|y-x_i|)^{N-2}} + \frac{C}{\mu^{\frac{pN}{q+1}}} \sum_{i=1}^k \frac{k^{p(N-2)-2}}{(1+k|y-x_i|)^{\min\{N-2,p(N-3-\theta)-2\}}}
\]
for $y \in \R^N$, where $\theta \in (0,1)$ is small.
\end{lemma}
\begin{proof}
The representation formula for $U$ yields
\begin{align*}
0 < U(y) &= \int_{\R^N}\frac{\ga_N}{|y-z|^{N-2}} \bigg(\sum_{j=1}^k V_j\bigg)^p(z)\, dz \\
&=  \sum_{i=1}^k \left[\int_{\Omega_i} \frac{\ga_N}{|y-z|^{N-2}} \bigg(\sum_{j=1}^k V_j\bigg)^p(z)\, dz\right], \quad y \in \R^N.
\end{align*}

For $i=1,\cdots, k$, we have
\[
\int_{\Omega_i} \frac{1}{|y-z|^{N-2}} V_i^p(z)\, dz \le C
U_i(y) \le \frac{C\mu^{\frac{N}{q+1}}}{(1+\mu|y-x_i|)^{N-2}}, \quad y \in \R^N.
\]

On the other hand, there exists $c>0$ such that for $j\ne i$,
\[
|z-k x_j|\ge c (1+ |z-k x_i|),\quad z\in\Omega_i.
\]
Thus,  we find
\begin{align*}
&\quad \int_{\Omega_i} \frac{\ga_N}{|y-z|^{N-2}} \bigg(\sum_{j \ne i} V_j\bigg)^p(z)\, dz \\
&\le \frac{Ck^{p(N-2)-2}}{\mu^{\frac{pN}{q+1}}} \int_{\Omega_i} \frac{1}{|ky-z|^{N-2}} \(\sum_{j \ne i} \frac{1}{|z-kx_j|^{N-2}}\)^p dz \\
&\le \frac{Ck^{p(N-2)-2}}{\mu^{\frac{pN}{q+1}}} \int_{\Omega_i} \frac{1}{|ky-z|^{N-2}}\frac{1}{(1+|z-kx_i|)^{p(N-3-\theta)}} \(\sum_{j \ne i} \frac{1}{|k x_i-kx_j|^{1+\theta}}\)^p dz \\
&\le \frac{Ck^{p(N-2)-2}}{\mu^{\frac{pN}{q+1}}} \int_{\Omega_i } \frac{1}{|ky-z|^{N-2}} \frac{dz}{(1+|z-kx_i|)^{p(N-3-\theta)}} \\
&\le \frac{C}{\mu^{\frac{pN}{q+1}}}  \frac{k^{p(N-2)-2}}{(1+k|y-x_i|)^{\min[N-2,p(N-3-\theta)-2]}}
\end{align*}
for $q$ satisfying \eqref{cri} and $\theta \in (0,1)$ small. Thus the result follows.
\end{proof}

Lastly, we present lemmas employed in Sections \ref{sec:lin} and \ref{sec:red}.
\begin{lemma}\label{l10-18-2}
Suppose that $p > 1$ and \eqref{cri} holds. Let $\tau > 0$ be a number such that $k \simeq \mu^{\tau}$. If $\tau' \ge \tau$, then we have that for any $y \in \R^N$,
\begin{equation}\label{0-6-2}
\left[\sum_{j=1}^k \frac{1}{(1+\mu|y-x_j|)^{\frac{N}{p+1}+\tau'}}\right]^{p}
\le C\sum_{j=1}^k \frac{ 1 }{ (1+\mu |y-x_j| )^{  \frac{N}{q+1}+2+\tau' } }
\end{equation}
and
\[\left[\sum_{j=1}^k \frac{1}{(1+\mu|y-x_j|)^{\frac{N}{q+1}+\tau'}}\right]^{q}
\le C\sum_{j=1}^k \frac{ 1 }{ (1+\mu |y-x_j| )^{  \frac{N}{p+1}+2+\tau' } }.\]
\end{lemma}
\begin{proof}
We just prove \eqref{0-6-2}. By H\"older's inequality and \eqref{cri},
\begin{align*}
\left[\sum_{j=1}^k\frac{1}{(1+\mu|y-x_j|)^{\frac{N}{p+1}+\tau'}}\right]^{p}
&= \left[\sum_{j=1}^k \frac{1}{(1+\mu |y-x_j|)^{(\frac{N}{q+1}+2+\tau')\frac{1}{p}+\frac{N}{p+1}+\tau' - (\frac{N}{q+1}+2+\tau') \frac{1}{p}}}  \right]^{p} \\
&\le C \sum_{j=1}^k \frac{1}{(1+\mu |y-x_j|)^{\frac{N}{q+1}+2+\tau'}}
\left[\sum_{j=1}^k \frac{1}{(1+\mu |y-x_j|)^{\tau'}}\right]^{p-1}  \\
&\le C \sum_{j=1}^k \frac{1}{(1+\mu |y-x_j|)^{\frac{N}{q+1}+2+\tau'}},
\end{align*}
where we used $\tau' \ge \tau$ to get $\sum_{j=1}^k (1+\mu |y-x_j|)^{-\tau'} \le C$.
\end{proof}

\begin{lemma}\label{l1-23-4}
Suppose that $N \ge 5$, $p \in (\frac{N}{N-2},\frac{N+2}{N-2}]$ and \eqref{cri} hold. If $N = 5$, we also assume that $p \in [\frac{35}{18},\frac73]$.
Let $\tau = \frac{N}{(p+1)(N-2)} \in (0,1)$. For any $y \in \R^N$, we have
\begin{equation}\label{10-23-4}
 U(y)\le C\sum_{j=1}^k \frac{ \mu^{\frac{N}{q+1}}  }{ (1+\mu |y-x_j|)^{  \frac{N}{q+1}+\tau } }.
\end{equation}
\end{lemma}
\begin{proof}
By symmetry, it suffices to verify \eqref{10-23-4} for $y \in \Omega_1$.

\medskip
First, if $y \in S$, then $0\le \varphi(y) \le C$, and
\begin{equation}\label{11-23-4}
U_j(y)\le \frac{C \mu^{\frac{N}{q+1}}  }{ (1+\mu |y-x_j|)^{  \frac{N}{q+1}+\tau } }.
\end{equation}
Therefore, \eqref{10-23-4} holds.

\medskip
Suppose that $y\in \Omega_1\setminus S$.  If $p(N-3)>N$, then the inequality $U\le C\sum_{j=1}^k U_j$ in $\R^N$ and \eqref{11-23-4} give the result.

If $p(N-3)\le N$, from Lemma~\ref{lemma:U} and \eqref{11-23-4}, we
just need to  prove that
\begin{equation}\label{12-23-4}
 \frac{1}{\mu^{\frac{pN}{q+1}}} \frac{k^{p(N-2)-2}}{(1+k|y-x_j|)^{p(N-3-\theta)-2}}\le
 \frac{C \mu^{\frac{N}{q+1}}  }{ (1+\mu |y-x_j|)^{  \frac{N}{q+1}+\tau } }.
\end{equation}
Let $R_0 > 1$ be a large number satisfying $|x_i| \le \frac {R_0}2$ for all $i = 1,\ldots,k$. We also note
\[\min\{N-2, p(N-3)-2\}>\frac N{q+1}+\tau\]
for either all $N \ge 6$ and $p \in (\frac N{N-2},\frac{N+2}{N-2}]$, or $N = 5$ and $p \in [\frac{35}{18},\frac73]$.
Thus, if $|y|\ge R_0$, then \eqref{12-23-4} holds provided
\begin{equation}\label{14-23-4}
\frac{pN}{q+1}-(p+1)\tau =p\( N-2 -\frac{N}{p+1}\)-\frac N{N-2}>0,
\end{equation}
which is true for all $N \ge 5$ and $p>\frac N{N-2}$. If $|y| < R_0$, then \eqref{12-23-4} holds if
\begin{equation}\label{13-23-4}
\frac{1}{\mu^{\frac{pN}{q+1}}} \frac{k^{p(N-2)-2}}{(k|y-x_j|)^{p(N-3-\theta)-2}}\le
\frac{C k^{\frac{N}{q+1}+\tau}  }{\mu^\tau (k |y-x_j|)^{  \frac{N}{q+1}+\tau } }.
\end{equation}
However, since $p(N-2)-2 = \frac{(p+1)N}{q+1}$, we have
\[\frac{pN}{q+1}-\tau -\tau\left[p(N-2)-2- \frac{N}{q+1}-\tau  \right] = (1-\tau)\(\frac{pN}{q+1}-\tau\)\]
whose right-hand side is positive thanks to $\tau \in (0,1)$ and \eqref{14-23-4}. Consequently, \eqref{13-23-4} is valid.
\end{proof}

\begin{lemma}\label{l10-18-4}
Suppose that $N \ge 5$, $p \in (\frac{N}{N-2},\frac{N+2}{N-2}]$ and \eqref{cri} hold. Let $\tau = \frac{N}{(p+1)(N-2)} \in (0,1)$. For any $y \in \R^N$, we have
\begin{equation}\label{0-6-4}
\frac{1}{(1+|y|)^{q(N-2)}} \le C\sum_{j=1}^k \frac{ \mu^{\frac{qN}{q+1}}  }{ (1+\mu |y-x_j|)^{  \frac{N}{p+1}+2+\tau } }
\end{equation}
and
\begin{equation}\label{0-6-41}
\frac{1}{(1+|y|)^{p(N-2)}} \le C\sum_{j=1}^k \frac{ \mu^{\frac{pN}{p+1}}  }{ (1+\mu |y-x_j|)^{  \frac{N}{q+1}+2+\tau } }.
\end{equation}
\end{lemma}
\begin{proof}
Let us prove \eqref{0-6-4} for $y \in \Omega_1$.

First, if $y \in S$, then the left-hand side of \eqref{0-6-4} is bounded by $C > 0$, while the right-hand side satisfies
\[\sum_{j=1}^k \frac{ \mu^{\frac{qN}{q+1}}  }{ (1+\mu |y-x_j|)^{  \frac{N}{p+1}+2+\tau } } \ge \frac{ \mu^{\frac{qN}{q+1}}  }{ (1+\mu |y-x_1|)^{  \frac{N}{p+1}+2+\tau } } \ge \frac{ck^{\frac{N}{p+1}+2+\tau}}{\mu^{\tau}} \simeq k^{\frac{N}{p+1}+1+\tau} \gg 1\]
for some $c > 0$, since $\frac{N}{p+1}+2 = \frac{qN}{q+1}$. Hence, \eqref{0-6-4} is valid for $y \in S$.

Second, if $y \in \Omega_1$ and $|y| \ge R_0$ (see the proof of Lemma~\ref{l1-23-4}), then \eqref{0-6-4} reads
\[\frac{1}{|y|^{q(N-2)}} \le C\sum_{j=1}^k \frac{\mu^{-\tau}}{|y-x_j|^{\frac{N}{p+1}+2+\tau}} \simeq \frac{1}{|y|^{\frac{N}{p+1}+2+\tau}}.\]
Furthermore, $q(N-2) \ge \frac{N}{p+1}+2+\tau$ is equivalent to $N^2-2+p(N-2) > 0$, which clearly holds true, thereby confirming the validity of \eqref{0-6-4} in this case.

Finally, if $y \in \Omega_1 \setminus S$ and $|y| \le R_0$, the left-hand side of \eqref{0-6-4} is bounded by $C > 0$, while the right-hand side is bounded from below by
\[c\sum_{j=1}^k \frac{\mu^{-\tau}}{|y-x_j|^{\frac{N}{p+1}+2+\tau}} \ge c \sum_{j=1}^k \frac{\mu^{-\tau}}{R_0^{\frac{N}{p+1}+2+\tau}} \ge c > 0.\]

\medskip
To derive \eqref{0-6-41}, we argue as above, by applying the inequality $p(N-2) \ge \frac{N}{q+1}+2+\tau$, which holds since it is equivalent to $h^2 + 2(N-1)h + N(N-3) \ge 0$ for $h = p(N-2)-2 \in [0,2]$. We omit the detail.
\end{proof}

\medskip \small {\noindent \textbf{Acknowledgement.}
Guo is supported by NSFC (No. 12271283), (No. 12031235);
Kim is supported by Basic Science Research Program through the National Research Foundation of Korea (NRF) funded by the Ministry of Science and ICT (2020R1C1C1A0101013314, RS-2025-00558417).
He also expresses gratitude to Korea Institute for Advanced Study (KIAS) for its support through the visiting faculty program during his sabbatical year;
Pistoia is partially supported by the MUR-PRIN-20227HX33Z ``Pattern formation in nonlinear phenomena" and INDAM-GNAMPA project ``Problemi di doppia
curvatura su variet\`a a bordo e legami con le EDP di tipo ellittico";
Yan is supported by NSFC (No. 12171184).}


\begin{thebibliography}{999}
\bibitem{ALT} A. Alvino, P.-L. Lions, and G. Trombetti, A remark on comparison results via symmetrization, Proc. Roy. Soc. Edinburgh Sect. A \textbf{102} (1986), 37--48.

\bibitem{CH} W. Chen and X. Huang, New type of solutions for the critical Lane-Emden system, J. Differential Equations \textbf{429} (2025), 318--391.

\bibitem{CL} W. Chen and C. Li, An integral system and the Lane-Emden conjecture, Discrete Contin. Dyn. Syst. \textbf{24} (2009), 1167--1184.

\bibitem{CLO} W. Chen, C. Li, and B. Ou, Classification of solutions for a system of integral equations, Comm. Partial Differential Equations \textbf{30} (2005), 59--65.

\bibitem{CS} M. Clapp and A. Salda\~na, Entire nodal solutions to the critical Lane-Emden system, Comm. Partial Differential Equations \textbf{45} (2020), 285--302.

\bibitem{DMPP} M. del Pino, M. Musso, F. Pacard, and A. Pistoia, Large energy entire solutions for the Yamabe equation, J. Differential Equations \textbf{251} (2011), 2568--2597.

\bibitem{DMPP2} M. del Pino, M. Musso, F. Pacard, and A. Pistoia, Torus action on $S^n$ and sign-changing solutions for conformally invariant equations. Ann. Sc. Norm. Super. Pisa Cl. Sci. \textbf{12} (2013), 209--237.

\bibitem{ding} W. Y. Ding, On a conformally invariant elliptic equation on $R^n$, Comm. Math. Phys. \textbf{107} (1986), 331--335.

\bibitem{FKP} R. L. Frank, S. Kim, and A. Pistoia, Non-degeneracy for the critical Lane-Emden system, Proc. Amer. Math. Soc. \textbf{149} (2021), 265--278.

\bibitem{GLP} Q. Guo, J. Liu, and S. Peng, Existence and non-degeneracy of positive multi-bubbling solutions to critical elliptic systems of Hamiltonian type, J. Differential Equations \textbf{355} (2023), 16--61.

\bibitem{GHPY} Y. Guo, Y. Hu, S. Peng, and T. Yuan, Asymptotic behavior of the ground-state solutions of Lane-Emden system, Bull. London Math. Soc. \textbf{57} (2025), 1--15.

\bibitem{GLW} Y. Guo, B. Li, and J. Wei, Large energy entire solutions for the Yamabe type problem of polyharmonic operator, J. Differential Equations \textbf{254} (2013), 199-228.

\bibitem{GWY} Y. Guo, S. Wu, and S. Yan, Periodic solution for Hamiltonian type systems with critical growth, Calc. Var. Partial Differential Equations \textbf{63} (2024), Article No. 147, 47 pages.

\bibitem{HV2} J. Hulshof and R. C. A. M. Van der Vorst, Asymptotic behavior of ground states, Proc. Amer. Math. Soc. \textbf{124} (1996), 2423--2431.

\bibitem{KM} S. Kim and S. Moon, Asymptotic analysis on positive solutions of the Lane-Emden system with nearly critical exponents, Trans. Amer. Math. Soc., \textbf{376} (2023), 4835--4899.

\bibitem{KP} S. Kim and A. Pistoia, Multiple blowing-up solutions to critical elliptic systems in bounded domains, J. Funct. Anal. \textbf{281} (2021), Article No. 109023, 58 pages.

\bibitem{Li} P.-L. Lions, The concentration-compactness principle in the calculus of variations. The limit case. I, Rev. Mat. Iberoamericana \textbf{1} (1985), 145--201.

\bibitem{MM} M. Medina and M. Musso, Doubling nodal solutions to the Yamabe equation in $\R^n$ with maximal rank, J. Math. Pures Appl. \textbf{152} (2021), 145--188.

\bibitem{MMW} M. Medina, M. Musso and J. Wei, Desingularization of Clifford torus and nonradial solutions to Yamabe problem with maximal rank, J. Funct. Anal. \textbf{276} (2019), 2470--2523

\bibitem{So} P. Souplet, The proof of the Lane-Emden conjecture in four space dimensions, Adv. Math. \textbf{221} (2009), 1409--1427.

\bibitem{Wa} X. Wang, Sharp constant in a Sobolev inequality, Nonlinear Anal. \textbf{20} (1993), 261--268.

\bibitem{WY} J. Wei and S. Yan, Infinitely many solutions for the prescribed scalar curvature problem on $S^N,$ J. Funct. Anal. \textbf{258} (2010), 3048--3081.
\end{thebibliography}
\end{document}